\documentclass{article}
\usepackage[left=4cm, right=4cm, top=4cm, bottom=2.25cm]{geometry}
\usepackage{amsmath,color}
\usepackage{amsfonts,amsthm}
\usepackage{amssymb,enumerate,enumitem,verbatim}
\usepackage[pdftex]{graphicx}
\usepackage{hyperref}
\usepackage{amsmath,setspace,scalefnt}
\usepackage[usenames,dvipsnames,svgnames,table]{xcolor}
\usepackage{amsfonts,amsthm}
\usepackage{amssymb,enumitem,verbatim}
\usepackage{graphicx}
\usepackage{tikz,caption,subcaption}
\usetikzlibrary{calc}
\usepackage[titles]{tocloft}
\usepackage[capitalise]{cleveref}
\usepackage{thmtools}
\usepackage{thm-restate}

\newcounter{propcounter}
\newcounter{restatedpropcounter}
\newcounter{restatedpropcounter2}

\global\long\def\N{\mathbb{N}}
\global\long\def\PP{\mathbb{P}}
\global\long\def\EE{\mathbb{E}}
\global\long\def\eps{\varepsilon}
\newcommand{\itref}[1]{\emph{\ref{#1}}}
\newcommand{\eref}[1]{\emph{\ref{#1}}}

\newtheorem{lemma}{Lemma}[section]
\newtheorem{corollary}[lemma]{Corollary}
\newtheorem{theorem}[lemma]{Theorem}

\newtheorem{conjecture}[lemma]{Conjecture}

\theoremstyle{definition}
\newtheorem{definition}[lemma]{Definition}
\newtheorem{claim}[lemma]{Claim}

\newif\ifgridpicone
\newif\ifgridpictwo
\newif\ifgridpicbox
\newif\ifgridpicvx
\newif\ifgridpictrees

\global\long\def\eps{\varepsilon}

\title{Packing the largest trees in the tree packing conjecture}
\author{Barnab\'as Janzer\thanks{Mathematical Institute, University of Oxford, Oxford, UK, OX2 6GG.
barnabas.janzer@magd.ox.ac.uk}
\ and Richard\ Montgomery\thanks{Mathematics Institute, University of Warwick, Coventry, UK, CV4 7AL. richard.montgomery@warwick.ac.uk. 
\newline
\textcolor{white}{.}\hspace{0.15cm}$^*$$^\dagger$\! Supported by the European Research Council (ERC) under the European Union Horizon 2020 research and innovation programme (grant agreement No. 947978).}}

\begin{document}

\maketitle

\begin{abstract}
The famous tree packing conjecture of Gy\'arf\'as from 1976 says that any sequence of trees $T_1,\ldots,T_n$ such that $|T_i|=i$ for each $i\in [n]$ packs into the complete $n$-vertex graph $K_n$. Packing even just the largest trees in such a sequence has proven difficult, with Bollob\'as drawing attention to this in 1995 by conjecturing that, for each $k$, if $n$ is sufficiently large then the largest $k$ trees in any such sequence can be packed into $K_n$. This has only been shown for $k\leq 5$, by \.{Z}ak, despite many partial results and much related work on the full tree packing conjecture.
We prove Bollob\'as's conjecture, by showing that, moreover, a linear number of the largest trees can be packed in the tree packing conjecture.
\end{abstract}


\section{Introduction}\label{sec:intro}
We say a collection of graphs $F_1,\ldots,F_r$ \emph{packs} into a graph $G$ if $G$ contains edge-disjoint copies of $F_1,\ldots,F_r$. Where such subgraphs contain every edge of $G$, we say this is a \emph{perfect packing}, also known as a \emph{decomposition}. A lot of work in graph packing has centred on two major conjectures on packing trees into complete graphs. The first of these, from 1967, is Ringel's conjecture that every $(n+1)$-vertex tree can be packed $2n+1$ times into the $(2n+1)$-vertex complete graph $K_{2n+1}$,
which was recently proved for large $n$ by the second author, Pokrovskiy and Sudakov~\cite{montgomery2021proof}, with an alternative proof given by Keevash and Staden~\cite{keevash2020ringel}. The second major conjecture, and the subject of this paper, is the following conjecture of Gy\'arf\'as (see~\cite{gyar}) from 1976, known as the tree packing conjecture (TPC).
\begin{conjecture}[The tree packing conjecture (TPC)]\label{conj:tpc}
Any sequence of trees $T_1,\ldots,T_n$ such that $|T_i|=i$ for each $i\in [n]$ packs into the complete $n$-vertex graph $K_n$.
\end{conjecture}

Note that if the packing in Conjecture~\ref{conj:tpc} exists then it will be a perfect packing. Early results on the tree packing conjecture showed that this packing is possible if the trees all belong to some limited subclass. Indeed,  Gy\'arf\'as and Lehel~\cite{gyar} showed that the TPC holds if all but at most 2 of the trees are stars, or
if all the trees are stars or paths.
Zaks and Liu~\cite{zaks1977decomposition} gave an alternative proof that the TPC holds if all the trees are stars or paths, while Hobbs (see~\cite{handbook}) proved that the TPC holds if all the trees are stars or double stars.
Extending a result of Straight~\cite{straight1979packing}, Fishburn~\cite{fishburn1983packing} showed that the TPC holds for all $n\leq 9$, and also extended the result of Hobbs to cover some more very specialised classes of trees.

In 1983, Bollob\'as~\cite{bollotrees} showed that the smallest $\lfloor n/\sqrt{2}\rfloor$ trees in the TPC can be packed into $K_n$, by simply packing them greedily in descending order of size, and observed that, moreover, the smallest $\lfloor \sqrt{3}n/2\rfloor$ trees can be packed if the Erd\H{o}s-S\'os conjecture holds. However, packing the largest trees in the TPC is surprisingly difficult. In 1987, Hobbs, Bourgeois and Kasiraj~\cite{hobbs1987packing} showed that the largest 3 trees can always be packed in the TPC, while, in recent years, \.{Z}ak~\cite{zak} has shown that the largest 5 trees can be packed when $n$ is large.
Thus, the following conjecture of Bollob\'as~\cite{handbook} from 1995,
 is not as innocuous as it might seem: that, for each $k$, there is some $n_k$ such that, for each $n\geq n_k$, the largest $k$ trees in the TPC can be packed into $K_n$. This conjecture does, however, become significantly easier with only a small alteration. That is to say, Balogh and Palmer~\cite{palmer} were able to show that if the largest tree is skipped, then the next $\frac{1}{10}n^{1/4}$ largest trees can be packed in the TPC (for large $n$), or, if none of the trees are stars, then the $\frac{1}{10}n^{1/4}$ largest trees can be packed in the TPC (again, for large $n$).

Within the last decade, a substantial amount of work has approached the tree packing conjecture by imposing an additional maximum degree condition on the trees. We will only mention the results most directly applicable towards the TPC, but many of these results hold for a wider class of subgraphs than trees or for a suitably quasirandom host graph (for more details on this and other related results see \cite{allen2021tree}). When quoting these results, by an \emph{approximate version of the TPC} we mean a result showing that $T_1,\ldots,T_{(1-o(1))n}$ can be packed into $K_n$ in Conjecture~\ref{conj:tpc}. The first such version was shown by B{\"o}ttcher, Hladk{\'y}, Piguet and Taraz~\cite{bottcher2016approximate} in 2016, who, for each $\Delta\in \N$, gave an approximate version of the TPC when the trees all have maximum degree at most $\Delta$.
Joos, Kim, K{\"u}hn and Osthus~\cite{joos2019optimal} then showed that, for each $\Delta$ and sufficiently large $n$, the TPC holds when the trees all have maximum degree at most $\Delta$.
Ferber and Samotij~\cite{ferber2019packing} extended the approximate results in 2019, by showing there is some $c>0$ for which an approximate version of the TPC holds when all the trees have maximum degree at most $c n/\log n$.
Finally, Allen, B{\"o}ttcher, Clemens, Hladk{\'y}, Piguet and Taraz~\cite{allen2021tree} used very substantial methods to prove there is some $c>0$ for which the TPC holds when all the trees have maximum degree at most $cn/\log n$.

Here, we concentrate on packing the largest trees for the tree packing conjecture, with no additional imposition on the trees packed, and show that the $\Omega(n)$ largest trees will pack into $K_n$, as follows.

\begin{theorem}\label{theorem_treepacking}
There exists a constant $\eps>0$ such that the following holds with $r=\eps n$ for all $n$. If ${T}_{n-r+1},\dots,{T}_{n}$ are trees with $|{T}_i|=i$ for each $n-r<i\leq n$, then ${T}_{n-r+1},\dots,{T}_{n}$ pack into $K_n$.
\end{theorem}

In particular, then, Bollob\'as's weak version of the tree packing conjecture is true.
We have not attempted to optimise the value of $\eps$ in Theorem~\ref{theorem_treepacking}, and in several places there is the potential for improvement, particularly by randomising some embeddings which we do greedily. However, if $(1-o(1))n$ trees in the TPC are to be packed then it is more approachable to do this starting with the smallest trees, that is, to give an approximate version of the TPC without the degree bound in the versions quoted above (or, perhaps, as a next step, using a degree bound linear in $n$). Due to this, we present our methods without complicating them with any such optimisations.  In the next section, we sketch our methods before outlining the rest of this paper. We note that our theorem above does not say anything meaningful about small fixed values of $n$, since packing $\eps n$ trees when $n\leq 1/\eps$ is trivial.


\section{Proof sketch}\label{sec:sketch}
Suppose, for some $1/n\ll \eps\ll 1$ and $r=\eps n$, we have trees $T_1,\ldots,T_r$ with $|T_i|=n-i+1$ for each $i\in [r]$ (relabelling from Theorem~\ref{theorem_treepacking} to start with the largest trees) and we try to  embed them into the complete graph $K_n$ in order starting with $T_1$. For each $i\in [r-1]$, after embedding $i$  trees, in order to continue there must be some vertex with $\Delta(T_{i+1})$ neighbouring edges that we have not yet used. If $T_{i+1}$ is a star, then the vertex used for the centre of the star in $K_{n}$, $v_{i+1}$ say, must have at most $i$ neighbouring edges used when embedding $T_1,\ldots,T_i$. That is, on average, $v_{i+1}$ must be a leaf of the copy of each $T_j$, $j\leq i$. This average is hardest to maintain when these trees are paths with only two leaves that could be embedded to $v_{i+1}$. We naturally then have two extreme cases for each tree in the sequence: if a tree is a star then we cannot use too many neighbouring edges around where we want to embed its centre, whereas if a tree is a path it has to be embedded using two neighbouring edges for most of the vertices in $K_n$.
For our embeddings, we divide the trees into `star-like' trees (trees with many leaves) and `path-like' trees (trees with few leaves) and embed them very differently according to this.

Before we go further with our sketch, let us simplify things by noting that, indeed, stars are the most difficult `star-like' trees to appear in the sequence (as is suggested by the result of Balogh and Palmer~\cite{palmer} packing the $\frac{1}{10}n^{1/4}$ largest trees for large $n$ if none are stars). Essentially, given a sequence of trees to embed for Theorem~\ref{theorem_treepacking}, we replace the `star-like' tree $T_i$ in the sequence with a star of the same size, and embed the new sequence instead into $K_n$  (with some mild additional degree conditions on the embedded edges). The embedded stars reserve edges to possibly embed a particularly high degree vertex in the corresponding `star-like' tree, and, removing the embedded stars, we can then embed the corresponding trees. To embed the `star-like' trees here takes a little care, but essentially is done using a method inspired by Havet, Reed, Stein, and Wood~\cite{havet2020variant}, and we postpone any discussion of this to Section~\ref{section_starlike}. There, we prove Theorem~\ref{theorem_treepacking} from the following result, which is essentially a version of Theorem~\ref{theorem_treepacking} where all the star-like trees are stars.

\begin{theorem}\label{theorem_treepacking_starlikestars}
Let $1/n\ll \eps \ll \lambda \ll 1$ and $r=\eps n$, and let $v_1,\ldots,v_n$ be an ordering of $V(K_n)$. If ${T}_{1},\dots,{T}_{r}$ are trees with $|{T}_i|=n-i+1$ for each $i\in [r]$, such that each tree $T_i$ is either a star, or a tree with at most $\lambda n$ leaves, then ${T}_{1},\dots,{T}_{r}$ pack into $K_n$ so that every vertex which is not the centre of an embedded star has degree at most $n/10$ in the union of the embedded trees and, for each $i\in [r]$, if $T_i$ is a star then it is embedded into $K_n$ with its centre to the left (in the ordering $v_1,\dots,v_n$) of all of its leaves.
\end{theorem}

As our embedding of the `path-like' trees is at times quite involved, we now spend a little time setting up a visual way to think about these embeddings (see the figures on page~\pageref{fig:grid1} for two example embeddings). For ease of notation, let us also extend our sequence of trees so that there are equally many stars and `path-like' trees. That is, suppose we have $r=\eps n$ stars $S_1,\ldots,S_r$, and `path-like' trees $P_1,\ldots,P_r$ which each have at most $\lambda n$ leaves, for some $\eps\ll \lambda \ll 1$, where both the stars and the path-like trees are listed in decreasing order in size, and in some order these $2r$ trees form a sequence of trees decreasing stepwise in size from $n$ vertices to $n-2r+1$ vertices (we call this sequence the \emph{TPC sequence}).

Order the vertices of $K_n$ as $v_1,\ldots,v_n$, where we will use $v_i$ as the centre of the $i$th star for each $i\in [r]$. We will often think of this ordering as `left to right', with $v_1$ being the leftmost vertex and $v_n$ being the rightmost one. When embedding a path-like tree $P_i$ we will only use the intended centre of any star occurring in the TPC sequence after $P_i$.
We need to cover these vertices very carefully, so we put them into the set $$W_i=\{v_j:j\in [r]\text{ and }|P_i|>|S_j|\}$$ (these sets are depicted in a grid in Figure~\ref{fig:grid1}). Note that we have $W_1\supseteq W_2\supseteq \dots\supseteq W_r$. The vertices not intended as the centre of any star we cover with the embedding of each tree $P_i$, $i\in [r]$, and put these vertices into the set $X=\{v_{r+1},\ldots,v_n\}$.
 In summary, we will obey the following rule.

\stepcounter{propcounter}
\begin{enumerate}[label = \textbf{\Alph{propcounter}\arabic{enumi}}]
\item For each $i\in [r]$, $P_i$ is embedded into $K_n[W_i\cup X]$.\label{prop:rule1}
\end{enumerate}

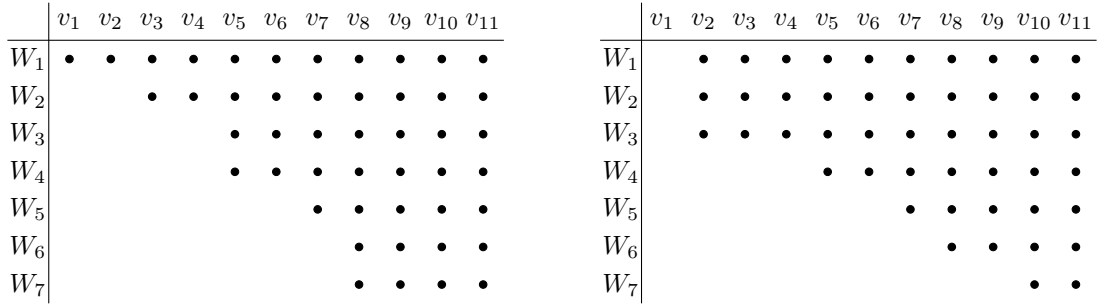
\begin{figure}[p!]
\centering
\gridpicvxtrue\gridpicboxfalse\gridpictreesfalse\gridpiconetrue\gridpictwofalse\begin{tikzpicture}

\def\wi{0.55cm}
\def\hgt{0.5cm}
\def\n{7}
\def\m{11}
\def\vx{0.05cm}

\foreach \x in {0,1,...,\n}
\foreach \y in {0,1,...,\m}
{
\coordinate (A\x\y) at ($\x*(0,-\hgt)+\y*(\wi,0)$);
}

\foreach \x in {1,...,\m}
{
\draw (A0\x) node {$v_{\x}$};
}
\foreach \x in {1,...,\n}
{
\draw (A\x0) node {$W_{\x}$};
}

\ifgridpicone
\ifgridpictrees
\foreach \x/\yy in {2/3,3/5,4/5,5/7,6/8,7/8}
{
\foreach \y in {\yy,...,\m}
{
\draw [fill=black!50,black!50] (A\x\y) circle [radius=\vx];
}
}
\fi
\foreach \x/\yy in {1/1,2/3,3/5,4/5,5/7,6/8,7/8}
{
\foreach \y in {\yy,...,\m}
{
\ifgridpicbox
\draw [fill=black!10,black!10] (A\x\y) circle [radius=\vx];
\fi
\ifgridpicvx
\draw [fill] (A\x\y) circle [radius=\vx];
\fi
}
}
\fi

\ifgridpictwo
\ifgridpictrees
\foreach \x/\yy in {2/2,3/2,4/5,5/7,6/8,7/10}
{
\foreach \y in {\yy,...,\m}
{
\draw [fill=black!50,black!50] (A\x\y) circle [radius=\vx];
}
}
\fi
\foreach \x/\yy in {1/2,2/2,3/2,4/5,5/7,6/8,7/10}
{
\foreach \y in {\yy,...,\m}
{
\ifgridpicbox
\draw [fill=black!10,black!10] (A\x\y) circle [radius=\vx];
\fi
\ifgridpicvx
\draw [fill] (A\x\y) circle [radius=\vx];
\fi
}
}
\fi

\ifgridpicone
\foreach \w/\x/\lab in {11/1\m/2,23/24/1,45/46/1}
{
\ifgridpicbox
\draw ($(A\w)+0.5*(-\wi,\hgt)$) -- ($(A\x)+0.5*(\wi,\hgt)$) -- ($(A\x)+0.5*(\wi,-\hgt)$) -- ($(A\w)+0.5*(-\wi,-\hgt)$) -- ($(A\w)+0.5*(-\wi,\hgt)$);
\draw ($0.5*(A\w)+0.5*(A\x)$) node {\lab};
\fi
\ifgridpictrees
\draw [black!10] ($(A\w)+0.5*(-\wi,\hgt)$) -- ($(A\x)+0.5*(\wi,\hgt)$) -- ($(A\x)+0.5*(\wi,-\hgt)$) -- ($(A\w)+0.5*(-\wi,-\hgt)$) -- ($(A\w)+0.5*(-\wi,\hgt)$);
\fi
}
\foreach \w/\x/\y/\z/\lab in {25/2\m/3\m/35/3,47/4\m/5\m/57/4,68/6\m/7\m/78/5}
{
\ifgridpicbox
\draw ($(A\w)+0.5*(-\wi,\hgt)$) -- ($(A\x)+0.5*(\wi,\hgt)$) -- ($(A\y)+0.5*(\wi,-\hgt)$) -- ($(A\z)+0.5*(-\wi,-\hgt)$) -- ($(A\w)+0.5*(-\wi,\hgt)$);
\draw ($0.25*(A\w)+0.25*(A\x)+0.25*(A\y)+0.25*(A\z)$) node {\lab};
\fi
\ifgridpictrees
\draw ($(A\w)+0.5*(-\wi,\hgt)$) -- ($(A\x)+0.5*(\wi,\hgt)$) -- ($(A\y)+0.5*(\wi,-\hgt)$) -- ($(A\z)+0.5*(-\wi,-\hgt)$) -- ($(A\w)+0.5*(-\wi,\hgt)$);
\draw ($0.25*(A\w)+0.25*(A\x)+0.25*(A\y)+0.25*(A\z)$) node {\lab};
\fi
}
\ifgridpictrees
\draw [thick] (A11)--(A12);
\draw [thick] (A12)to [out=40,in=140] (A14);
\draw [thick] (A14)to [out=40,in=140] (A16);
\draw [thick] (A13)to [out=40,in=140] (A15);
\draw [thick,->] (A15)-- ++(0.5*\wi,0);
\draw [thick] (A16)--(A111);
\draw [thick,->] (A111)-- ++(0.4*\wi,0);
\draw [thick] (A23)--(A24);
\draw [thick,->] (A24)-- ++(0.4*\wi,0);
\draw [thick] (A45)--(A46);
\draw [thick,->] (A46)-- ++(0.4*\wi,0);
\fi
\fi

\ifgridpictwo
\foreach \w/\x/\lab in {12/1\m/2,68/69/1,45/46/1}
{
\ifgridpicbox
\draw ($(A\w)+0.5*(-\wi,\hgt)$) -- ($(A\x)+0.5*(\wi,\hgt)$) -- ($(A\x)+0.5*(\wi,-\hgt)$) -- ($(A\w)+0.5*(-\wi,-\hgt)$) -- ($(A\w)+0.5*(-\wi,\hgt)$);
\draw ($0.5*(A\w)+0.5*(A\x)$) node {\lab};
\fi
\ifgridpictrees
\draw [black!10] ($(A\w)+0.5*(-\wi,\hgt)$) -- ($(A\x)+0.5*(\wi,\hgt)$) -- ($(A\x)+0.5*(\wi,-\hgt)$) -- ($(A\w)+0.5*(-\wi,-\hgt)$) -- ($(A\w)+0.5*(-\wi,\hgt)$);
\fi
}
\foreach \w/\x/\y/\z/\lab in {22/2\m/3\m/32/3,47/4\m/5\m/57/4,610/6\m/7\m/710/5}
{
\ifgridpicbox
\draw ($(A\w)+0.5*(-\wi,\hgt)$) -- ($(A\x)+0.5*(\wi,\hgt)$) -- ($(A\y)+0.5*(\wi,-\hgt)$) -- ($(A\z)+0.5*(-\wi,-\hgt)$) -- ($(A\w)+0.5*(-\wi,\hgt)$);
\draw ($0.25*(A\w)+0.25*(A\x)+0.25*(A\y)+0.25*(A\z)$) node {\lab};
\fi
\ifgridpictrees
\draw ($(A\w)+0.5*(-\wi,\hgt)$) -- ($(A\x)+0.5*(\wi,\hgt)$) -- ($(A\y)+0.5*(\wi,-\hgt)$) -- ($(A\z)+0.5*(-\wi,-\hgt)$) -- ($(A\w)+0.5*(-\wi,\hgt)$);
\draw ($0.25*(A\w)+0.25*(A\x)+0.25*(A\y)+0.25*(A\z)$) node {\lab};
\fi
}
\ifgridpictrees
\draw [thick] (A12)--(A14);
\draw [thick] (A14)to [out=40,in=140] (A16);
\draw [thick] (A16) -- (A17);
\draw [thick] (A17)to [out=40,in=140] (A19);
\draw [thick] (A15)to [out=25,in=155] (A18);
\draw [thick,->] (A18)-- ++(0.5*\wi,0);
\draw [thick] (A19)--(A111);
\draw [thick,->] (A111)-- ++(0.4*\wi,0);
\draw [thick] (A45)--(A46);
\draw [thick,->] (A46)-- ++(0.4*\wi,0);
\draw [thick] (A68)--(A69);
\draw [thick,->] (A69)-- ++(0.4*\wi,0);
\fi
\fi

\ifgridpicone
\ifgridpictrees
\foreach \y in {11,10,9,8,7,6,4,2,1}
{
\draw [fill=blue,blue] (A1\y) circle [radius=\vx];
}
\foreach \y in {5,3}
{
\draw [red,fill=red] (A1\y) circle [radius=\vx];
}
\fi
\fi
\ifgridpictwo
\ifgridpictrees
\foreach \y in {11,10,9,7,6,4,3,2}
{
\draw [fill=blue,blue] (A1\y) circle [radius=\vx];
}
\foreach \y in {5,8}
{
\draw [red,fill=red] (A1\y) circle [radius=\vx];
}
\fi
\fi

\draw ($(A00)-0.5*(\wi,\hgt)$) -- ($(A0\m)+0.5*(\wi,-\hgt)$);
\draw ($(A00)+0.5*(\wi,\hgt)$) -- ($(A\n0)+0.5*(\wi,-\hgt)$);

\end{tikzpicture}\hspace{1cm}\gridpictwotrue\gridpiconefalse\begin{tikzpicture}

\def\wi{0.55cm}
\def\hgt{0.5cm}
\def\n{7}
\def\m{11}
\def\vx{0.05cm}

\foreach \x in {0,1,...,\n}
\foreach \y in {0,1,...,\m}
{
\coordinate (A\x\y) at ($\x*(0,-\hgt)+\y*(\wi,0)$);
}

\foreach \x in {1,...,\m}
{
\draw (A0\x) node {$v_{\x}$};
}
\foreach \x in {1,...,\n}
{
\draw (A\x0) node {$W_{\x}$};
}

\ifgridpicone
\ifgridpictrees
\foreach \x/\yy in {2/3,3/5,4/5,5/7,6/8,7/8}
{
\foreach \y in {\yy,...,\m}
{
\draw [fill=black!50,black!50] (A\x\y) circle [radius=\vx];
}
}
\fi
\foreach \x/\yy in {1/1,2/3,3/5,4/5,5/7,6/8,7/8}
{
\foreach \y in {\yy,...,\m}
{
\ifgridpicbox
\draw [fill=black!10,black!10] (A\x\y) circle [radius=\vx];
\fi
\ifgridpicvx
\draw [fill] (A\x\y) circle [radius=\vx];
\fi
}
}
\fi

\ifgridpictwo
\ifgridpictrees
\foreach \x/\yy in {2/2,3/2,4/5,5/7,6/8,7/10}
{
\foreach \y in {\yy,...,\m}
{
\draw [fill=black!50,black!50] (A\x\y) circle [radius=\vx];
}
}
\fi
\foreach \x/\yy in {1/2,2/2,3/2,4/5,5/7,6/8,7/10}
{
\foreach \y in {\yy,...,\m}
{
\ifgridpicbox
\draw [fill=black!10,black!10] (A\x\y) circle [radius=\vx];
\fi
\ifgridpicvx
\draw [fill] (A\x\y) circle [radius=\vx];
\fi
}
}
\fi

\ifgridpicone
\foreach \w/\x/\lab in {11/1\m/2,23/24/1,45/46/1}
{
\ifgridpicbox
\draw ($(A\w)+0.5*(-\wi,\hgt)$) -- ($(A\x)+0.5*(\wi,\hgt)$) -- ($(A\x)+0.5*(\wi,-\hgt)$) -- ($(A\w)+0.5*(-\wi,-\hgt)$) -- ($(A\w)+0.5*(-\wi,\hgt)$);
\draw ($0.5*(A\w)+0.5*(A\x)$) node {\lab};
\fi
\ifgridpictrees
\draw [black!10] ($(A\w)+0.5*(-\wi,\hgt)$) -- ($(A\x)+0.5*(\wi,\hgt)$) -- ($(A\x)+0.5*(\wi,-\hgt)$) -- ($(A\w)+0.5*(-\wi,-\hgt)$) -- ($(A\w)+0.5*(-\wi,\hgt)$);
\fi
}
\foreach \w/\x/\y/\z/\lab in {25/2\m/3\m/35/3,47/4\m/5\m/57/4,68/6\m/7\m/78/5}
{
\ifgridpicbox
\draw ($(A\w)+0.5*(-\wi,\hgt)$) -- ($(A\x)+0.5*(\wi,\hgt)$) -- ($(A\y)+0.5*(\wi,-\hgt)$) -- ($(A\z)+0.5*(-\wi,-\hgt)$) -- ($(A\w)+0.5*(-\wi,\hgt)$);
\draw ($0.25*(A\w)+0.25*(A\x)+0.25*(A\y)+0.25*(A\z)$) node {\lab};
\fi
\ifgridpictrees
\draw ($(A\w)+0.5*(-\wi,\hgt)$) -- ($(A\x)+0.5*(\wi,\hgt)$) -- ($(A\y)+0.5*(\wi,-\hgt)$) -- ($(A\z)+0.5*(-\wi,-\hgt)$) -- ($(A\w)+0.5*(-\wi,\hgt)$);
\draw ($0.25*(A\w)+0.25*(A\x)+0.25*(A\y)+0.25*(A\z)$) node {\lab};
\fi
}
\ifgridpictrees
\draw [thick] (A11)--(A12);
\draw [thick] (A12)to [out=40,in=140] (A14);
\draw [thick] (A14)to [out=40,in=140] (A16);
\draw [thick] (A13)to [out=40,in=140] (A15);
\draw [thick,->] (A15)-- ++(0.5*\wi,0);
\draw [thick] (A16)--(A111);
\draw [thick,->] (A111)-- ++(0.4*\wi,0);
\draw [thick] (A23)--(A24);
\draw [thick,->] (A24)-- ++(0.4*\wi,0);
\draw [thick] (A45)--(A46);
\draw [thick,->] (A46)-- ++(0.4*\wi,0);
\fi
\fi

\ifgridpictwo
\foreach \w/\x/\lab in {12/1\m/2,68/69/1,45/46/1}
{
\ifgridpicbox
\draw ($(A\w)+0.5*(-\wi,\hgt)$) -- ($(A\x)+0.5*(\wi,\hgt)$) -- ($(A\x)+0.5*(\wi,-\hgt)$) -- ($(A\w)+0.5*(-\wi,-\hgt)$) -- ($(A\w)+0.5*(-\wi,\hgt)$);
\draw ($0.5*(A\w)+0.5*(A\x)$) node {\lab};
\fi
\ifgridpictrees
\draw [black!10] ($(A\w)+0.5*(-\wi,\hgt)$) -- ($(A\x)+0.5*(\wi,\hgt)$) -- ($(A\x)+0.5*(\wi,-\hgt)$) -- ($(A\w)+0.5*(-\wi,-\hgt)$) -- ($(A\w)+0.5*(-\wi,\hgt)$);
\fi
}
\foreach \w/\x/\y/\z/\lab in {22/2\m/3\m/32/3,47/4\m/5\m/57/4,610/6\m/7\m/710/5}
{
\ifgridpicbox
\draw ($(A\w)+0.5*(-\wi,\hgt)$) -- ($(A\x)+0.5*(\wi,\hgt)$) -- ($(A\y)+0.5*(\wi,-\hgt)$) -- ($(A\z)+0.5*(-\wi,-\hgt)$) -- ($(A\w)+0.5*(-\wi,\hgt)$);
\draw ($0.25*(A\w)+0.25*(A\x)+0.25*(A\y)+0.25*(A\z)$) node {\lab};
\fi
\ifgridpictrees
\draw ($(A\w)+0.5*(-\wi,\hgt)$) -- ($(A\x)+0.5*(\wi,\hgt)$) -- ($(A\y)+0.5*(\wi,-\hgt)$) -- ($(A\z)+0.5*(-\wi,-\hgt)$) -- ($(A\w)+0.5*(-\wi,\hgt)$);
\draw ($0.25*(A\w)+0.25*(A\x)+0.25*(A\y)+0.25*(A\z)$) node {\lab};
\fi
}
\ifgridpictrees
\draw [thick] (A12)--(A14);
\draw [thick] (A14)to [out=40,in=140] (A16);
\draw [thick] (A16) -- (A17);
\draw [thick] (A17)to [out=40,in=140] (A19);
\draw [thick] (A15)to [out=25,in=155] (A18);
\draw [thick,->] (A18)-- ++(0.5*\wi,0);
\draw [thick] (A19)--(A111);
\draw [thick,->] (A111)-- ++(0.4*\wi,0);
\draw [thick] (A45)--(A46);
\draw [thick,->] (A46)-- ++(0.4*\wi,0);
\draw [thick] (A68)--(A69);
\draw [thick,->] (A69)-- ++(0.4*\wi,0);
\fi
\fi

\ifgridpicone
\ifgridpictrees
\foreach \y in {11,10,9,8,7,6,4,2,1}
{
\draw [fill=blue,blue] (A1\y) circle [radius=\vx];
}
\foreach \y in {5,3}
{
\draw [red,fill=red] (A1\y) circle [radius=\vx];
}
\fi
\fi
\ifgridpictwo
\ifgridpictrees
\foreach \y in {11,10,9,7,6,4,3,2}
{
\draw [fill=blue,blue] (A1\y) circle [radius=\vx];
}
\foreach \y in {5,8}
{
\draw [red,fill=red] (A1\y) circle [radius=\vx];
}
\fi
\fi

\draw ($(A00)-0.5*(\wi,\hgt)$) -- ($(A0\m)+0.5*(\wi,-\hgt)$);
\draw ($(A00)+0.5*(\wi,\hgt)$) -- ($(A\n0)+0.5*(\wi,-\hgt)$);

\end{tikzpicture}
\caption{Grids of vertices in $W_i$, $i\in [7]$, corresponding to the sequences\newline
\centering $P_1,S_1,S_2,P_2,S_3,S_4,P_3,P_4,S_5,S_6,P_5,S_7,P_6,P_7,S_8,S_9,S_{10},S_{11}$ (left) and\newline
\centering $S_1,P_1,P_2,P_3,S_2,S_3,S_4,P_4,S_5,S_6,P_5,S_7,P_6,S_8,S_9,P_7,S_{10},S_{11}$ (right).}
\label{fig:grid1}
\end{figure}

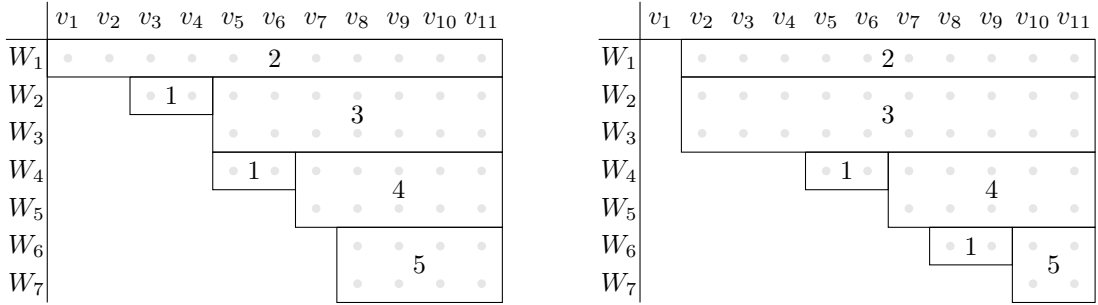
\begin{figure}[p!]
\gridpicvxfalse\gridpicboxtrue\gridpictreesfalse\gridpiconetrue\gridpictwofalse\begin{tikzpicture}

\def\wi{0.55cm}
\def\hgt{0.5cm}
\def\n{7}
\def\m{11}
\def\vx{0.05cm}

\foreach \x in {0,1,...,\n}
\foreach \y in {0,1,...,\m}
{
\coordinate (A\x\y) at ($\x*(0,-\hgt)+\y*(\wi,0)$);
}

\foreach \x in {1,...,\m}
{
\draw (A0\x) node {$v_{\x}$};
}
\foreach \x in {1,...,\n}
{
\draw (A\x0) node {$W_{\x}$};
}

\ifgridpicone
\ifgridpictrees
\foreach \x/\yy in {2/3,3/5,4/5,5/7,6/8,7/8}
{
\foreach \y in {\yy,...,\m}
{
\draw [fill=black!50,black!50] (A\x\y) circle [radius=\vx];
}
}
\fi
\foreach \x/\yy in {1/1,2/3,3/5,4/5,5/7,6/8,7/8}
{
\foreach \y in {\yy,...,\m}
{
\ifgridpicbox
\draw [fill=black!10,black!10] (A\x\y) circle [radius=\vx];
\fi
\ifgridpicvx
\draw [fill] (A\x\y) circle [radius=\vx];
\fi
}
}
\fi

\ifgridpictwo
\ifgridpictrees
\foreach \x/\yy in {2/2,3/2,4/5,5/7,6/8,7/10}
{
\foreach \y in {\yy,...,\m}
{
\draw [fill=black!50,black!50] (A\x\y) circle [radius=\vx];
}
}
\fi
\foreach \x/\yy in {1/2,2/2,3/2,4/5,5/7,6/8,7/10}
{
\foreach \y in {\yy,...,\m}
{
\ifgridpicbox
\draw [fill=black!10,black!10] (A\x\y) circle [radius=\vx];
\fi
\ifgridpicvx
\draw [fill] (A\x\y) circle [radius=\vx];
\fi
}
}
\fi

\ifgridpicone
\foreach \w/\x/\lab in {11/1\m/2,23/24/1,45/46/1}
{
\ifgridpicbox
\draw ($(A\w)+0.5*(-\wi,\hgt)$) -- ($(A\x)+0.5*(\wi,\hgt)$) -- ($(A\x)+0.5*(\wi,-\hgt)$) -- ($(A\w)+0.5*(-\wi,-\hgt)$) -- ($(A\w)+0.5*(-\wi,\hgt)$);
\draw ($0.5*(A\w)+0.5*(A\x)$) node {\lab};
\fi
\ifgridpictrees
\draw [black!10] ($(A\w)+0.5*(-\wi,\hgt)$) -- ($(A\x)+0.5*(\wi,\hgt)$) -- ($(A\x)+0.5*(\wi,-\hgt)$) -- ($(A\w)+0.5*(-\wi,-\hgt)$) -- ($(A\w)+0.5*(-\wi,\hgt)$);
\fi
}
\foreach \w/\x/\y/\z/\lab in {25/2\m/3\m/35/3,47/4\m/5\m/57/4,68/6\m/7\m/78/5}
{
\ifgridpicbox
\draw ($(A\w)+0.5*(-\wi,\hgt)$) -- ($(A\x)+0.5*(\wi,\hgt)$) -- ($(A\y)+0.5*(\wi,-\hgt)$) -- ($(A\z)+0.5*(-\wi,-\hgt)$) -- ($(A\w)+0.5*(-\wi,\hgt)$);
\draw ($0.25*(A\w)+0.25*(A\x)+0.25*(A\y)+0.25*(A\z)$) node {\lab};
\fi
\ifgridpictrees
\draw ($(A\w)+0.5*(-\wi,\hgt)$) -- ($(A\x)+0.5*(\wi,\hgt)$) -- ($(A\y)+0.5*(\wi,-\hgt)$) -- ($(A\z)+0.5*(-\wi,-\hgt)$) -- ($(A\w)+0.5*(-\wi,\hgt)$);
\draw ($0.25*(A\w)+0.25*(A\x)+0.25*(A\y)+0.25*(A\z)$) node {\lab};
\fi
}
\ifgridpictrees
\draw [thick] (A11)--(A12);
\draw [thick] (A12)to [out=40,in=140] (A14);
\draw [thick] (A14)to [out=40,in=140] (A16);
\draw [thick] (A13)to [out=40,in=140] (A15);
\draw [thick,->] (A15)-- ++(0.5*\wi,0);
\draw [thick] (A16)--(A111);
\draw [thick,->] (A111)-- ++(0.4*\wi,0);
\draw [thick] (A23)--(A24);
\draw [thick,->] (A24)-- ++(0.4*\wi,0);
\draw [thick] (A45)--(A46);
\draw [thick,->] (A46)-- ++(0.4*\wi,0);
\fi
\fi

\ifgridpictwo
\foreach \w/\x/\lab in {12/1\m/2,68/69/1,45/46/1}
{
\ifgridpicbox
\draw ($(A\w)+0.5*(-\wi,\hgt)$) -- ($(A\x)+0.5*(\wi,\hgt)$) -- ($(A\x)+0.5*(\wi,-\hgt)$) -- ($(A\w)+0.5*(-\wi,-\hgt)$) -- ($(A\w)+0.5*(-\wi,\hgt)$);
\draw ($0.5*(A\w)+0.5*(A\x)$) node {\lab};
\fi
\ifgridpictrees
\draw [black!10] ($(A\w)+0.5*(-\wi,\hgt)$) -- ($(A\x)+0.5*(\wi,\hgt)$) -- ($(A\x)+0.5*(\wi,-\hgt)$) -- ($(A\w)+0.5*(-\wi,-\hgt)$) -- ($(A\w)+0.5*(-\wi,\hgt)$);
\fi
}
\foreach \w/\x/\y/\z/\lab in {22/2\m/3\m/32/3,47/4\m/5\m/57/4,610/6\m/7\m/710/5}
{
\ifgridpicbox
\draw ($(A\w)+0.5*(-\wi,\hgt)$) -- ($(A\x)+0.5*(\wi,\hgt)$) -- ($(A\y)+0.5*(\wi,-\hgt)$) -- ($(A\z)+0.5*(-\wi,-\hgt)$) -- ($(A\w)+0.5*(-\wi,\hgt)$);
\draw ($0.25*(A\w)+0.25*(A\x)+0.25*(A\y)+0.25*(A\z)$) node {\lab};
\fi
\ifgridpictrees
\draw ($(A\w)+0.5*(-\wi,\hgt)$) -- ($(A\x)+0.5*(\wi,\hgt)$) -- ($(A\y)+0.5*(\wi,-\hgt)$) -- ($(A\z)+0.5*(-\wi,-\hgt)$) -- ($(A\w)+0.5*(-\wi,\hgt)$);
\draw ($0.25*(A\w)+0.25*(A\x)+0.25*(A\y)+0.25*(A\z)$) node {\lab};
\fi
}
\ifgridpictrees
\draw [thick] (A12)--(A14);
\draw [thick] (A14)to [out=40,in=140] (A16);
\draw [thick] (A16) -- (A17);
\draw [thick] (A17)to [out=40,in=140] (A19);
\draw [thick] (A15)to [out=25,in=155] (A18);
\draw [thick,->] (A18)-- ++(0.5*\wi,0);
\draw [thick] (A19)--(A111);
\draw [thick,->] (A111)-- ++(0.4*\wi,0);
\draw [thick] (A45)--(A46);
\draw [thick,->] (A46)-- ++(0.4*\wi,0);
\draw [thick] (A68)--(A69);
\draw [thick,->] (A69)-- ++(0.4*\wi,0);
\fi
\fi

\ifgridpicone
\ifgridpictrees
\foreach \y in {11,10,9,8,7,6,4,2,1}
{
\draw [fill=blue,blue] (A1\y) circle [radius=\vx];
}
\foreach \y in {5,3}
{
\draw [red,fill=red] (A1\y) circle [radius=\vx];
}
\fi
\fi
\ifgridpictwo
\ifgridpictrees
\foreach \y in {11,10,9,7,6,4,3,2}
{
\draw [fill=blue,blue] (A1\y) circle [radius=\vx];
}
\foreach \y in {5,8}
{
\draw [red,fill=red] (A1\y) circle [radius=\vx];
}
\fi
\fi

\draw ($(A00)-0.5*(\wi,\hgt)$) -- ($(A0\m)+0.5*(\wi,-\hgt)$);
\draw ($(A00)+0.5*(\wi,\hgt)$) -- ($(A\n0)+0.5*(\wi,-\hgt)$);

\end{tikzpicture}\hspace{1cm}\gridpictwotrue\gridpiconefalse\begin{tikzpicture}

\def\wi{0.55cm}
\def\hgt{0.5cm}
\def\n{7}
\def\m{11}
\def\vx{0.05cm}

\foreach \x in {0,1,...,\n}
\foreach \y in {0,1,...,\m}
{
\coordinate (A\x\y) at ($\x*(0,-\hgt)+\y*(\wi,0)$);
}

\foreach \x in {1,...,\m}
{
\draw (A0\x) node {$v_{\x}$};
}
\foreach \x in {1,...,\n}
{
\draw (A\x0) node {$W_{\x}$};
}

\ifgridpicone
\ifgridpictrees
\foreach \x/\yy in {2/3,3/5,4/5,5/7,6/8,7/8}
{
\foreach \y in {\yy,...,\m}
{
\draw [fill=black!50,black!50] (A\x\y) circle [radius=\vx];
}
}
\fi
\foreach \x/\yy in {1/1,2/3,3/5,4/5,5/7,6/8,7/8}
{
\foreach \y in {\yy,...,\m}
{
\ifgridpicbox
\draw [fill=black!10,black!10] (A\x\y) circle [radius=\vx];
\fi
\ifgridpicvx
\draw [fill] (A\x\y) circle [radius=\vx];
\fi
}
}
\fi

\ifgridpictwo
\ifgridpictrees
\foreach \x/\yy in {2/2,3/2,4/5,5/7,6/8,7/10}
{
\foreach \y in {\yy,...,\m}
{
\draw [fill=black!50,black!50] (A\x\y) circle [radius=\vx];
}
}
\fi
\foreach \x/\yy in {1/2,2/2,3/2,4/5,5/7,6/8,7/10}
{
\foreach \y in {\yy,...,\m}
{
\ifgridpicbox
\draw [fill=black!10,black!10] (A\x\y) circle [radius=\vx];
\fi
\ifgridpicvx
\draw [fill] (A\x\y) circle [radius=\vx];
\fi
}
}
\fi

\ifgridpicone
\foreach \w/\x/\lab in {11/1\m/2,23/24/1,45/46/1}
{
\ifgridpicbox
\draw ($(A\w)+0.5*(-\wi,\hgt)$) -- ($(A\x)+0.5*(\wi,\hgt)$) -- ($(A\x)+0.5*(\wi,-\hgt)$) -- ($(A\w)+0.5*(-\wi,-\hgt)$) -- ($(A\w)+0.5*(-\wi,\hgt)$);
\draw ($0.5*(A\w)+0.5*(A\x)$) node {\lab};
\fi
\ifgridpictrees
\draw [black!10] ($(A\w)+0.5*(-\wi,\hgt)$) -- ($(A\x)+0.5*(\wi,\hgt)$) -- ($(A\x)+0.5*(\wi,-\hgt)$) -- ($(A\w)+0.5*(-\wi,-\hgt)$) -- ($(A\w)+0.5*(-\wi,\hgt)$);
\fi
}
\foreach \w/\x/\y/\z/\lab in {25/2\m/3\m/35/3,47/4\m/5\m/57/4,68/6\m/7\m/78/5}
{
\ifgridpicbox
\draw ($(A\w)+0.5*(-\wi,\hgt)$) -- ($(A\x)+0.5*(\wi,\hgt)$) -- ($(A\y)+0.5*(\wi,-\hgt)$) -- ($(A\z)+0.5*(-\wi,-\hgt)$) -- ($(A\w)+0.5*(-\wi,\hgt)$);
\draw ($0.25*(A\w)+0.25*(A\x)+0.25*(A\y)+0.25*(A\z)$) node {\lab};
\fi
\ifgridpictrees
\draw ($(A\w)+0.5*(-\wi,\hgt)$) -- ($(A\x)+0.5*(\wi,\hgt)$) -- ($(A\y)+0.5*(\wi,-\hgt)$) -- ($(A\z)+0.5*(-\wi,-\hgt)$) -- ($(A\w)+0.5*(-\wi,\hgt)$);
\draw ($0.25*(A\w)+0.25*(A\x)+0.25*(A\y)+0.25*(A\z)$) node {\lab};
\fi
}
\ifgridpictrees
\draw [thick] (A11)--(A12);
\draw [thick] (A12)to [out=40,in=140] (A14);
\draw [thick] (A14)to [out=40,in=140] (A16);
\draw [thick] (A13)to [out=40,in=140] (A15);
\draw [thick,->] (A15)-- ++(0.5*\wi,0);
\draw [thick] (A16)--(A111);
\draw [thick,->] (A111)-- ++(0.4*\wi,0);
\draw [thick] (A23)--(A24);
\draw [thick,->] (A24)-- ++(0.4*\wi,0);
\draw [thick] (A45)--(A46);
\draw [thick,->] (A46)-- ++(0.4*\wi,0);
\fi
\fi

\ifgridpictwo
\foreach \w/\x/\lab in {12/1\m/2,68/69/1,45/46/1}
{
\ifgridpicbox
\draw ($(A\w)+0.5*(-\wi,\hgt)$) -- ($(A\x)+0.5*(\wi,\hgt)$) -- ($(A\x)+0.5*(\wi,-\hgt)$) -- ($(A\w)+0.5*(-\wi,-\hgt)$) -- ($(A\w)+0.5*(-\wi,\hgt)$);
\draw ($0.5*(A\w)+0.5*(A\x)$) node {\lab};
\fi
\ifgridpictrees
\draw [black!10] ($(A\w)+0.5*(-\wi,\hgt)$) -- ($(A\x)+0.5*(\wi,\hgt)$) -- ($(A\x)+0.5*(\wi,-\hgt)$) -- ($(A\w)+0.5*(-\wi,-\hgt)$) -- ($(A\w)+0.5*(-\wi,\hgt)$);
\fi
}
\foreach \w/\x/\y/\z/\lab in {22/2\m/3\m/32/3,47/4\m/5\m/57/4,610/6\m/7\m/710/5}
{
\ifgridpicbox
\draw ($(A\w)+0.5*(-\wi,\hgt)$) -- ($(A\x)+0.5*(\wi,\hgt)$) -- ($(A\y)+0.5*(\wi,-\hgt)$) -- ($(A\z)+0.5*(-\wi,-\hgt)$) -- ($(A\w)+0.5*(-\wi,\hgt)$);
\draw ($0.25*(A\w)+0.25*(A\x)+0.25*(A\y)+0.25*(A\z)$) node {\lab};
\fi
\ifgridpictrees
\draw ($(A\w)+0.5*(-\wi,\hgt)$) -- ($(A\x)+0.5*(\wi,\hgt)$) -- ($(A\y)+0.5*(\wi,-\hgt)$) -- ($(A\z)+0.5*(-\wi,-\hgt)$) -- ($(A\w)+0.5*(-\wi,\hgt)$);
\draw ($0.25*(A\w)+0.25*(A\x)+0.25*(A\y)+0.25*(A\z)$) node {\lab};
\fi
}
\ifgridpictrees
\draw [thick] (A12)--(A14);
\draw [thick] (A14)to [out=40,in=140] (A16);
\draw [thick] (A16) -- (A17);
\draw [thick] (A17)to [out=40,in=140] (A19);
\draw [thick] (A15)to [out=25,in=155] (A18);
\draw [thick,->] (A18)-- ++(0.5*\wi,0);
\draw [thick] (A19)--(A111);
\draw [thick,->] (A111)-- ++(0.4*\wi,0);
\draw [thick] (A45)--(A46);
\draw [thick,->] (A46)-- ++(0.4*\wi,0);
\draw [thick] (A68)--(A69);
\draw [thick,->] (A69)-- ++(0.4*\wi,0);
\fi
\fi

\ifgridpicone
\ifgridpictrees
\foreach \y in {11,10,9,8,7,6,4,2,1}
{
\draw [fill=blue,blue] (A1\y) circle [radius=\vx];
}
\foreach \y in {5,3}
{
\draw [red,fill=red] (A1\y) circle [radius=\vx];
}
\fi
\fi
\ifgridpictwo
\ifgridpictrees
\foreach \y in {11,10,9,7,6,4,3,2}
{
\draw [fill=blue,blue] (A1\y) circle [radius=\vx];
}
\foreach \y in {5,8}
{
\draw [red,fill=red] (A1\y) circle [radius=\vx];
}
\fi
\fi

\draw ($(A00)-0.5*(\wi,\hgt)$) -- ($(A0\m)+0.5*(\wi,-\hgt)$);
\draw ($(A00)+0.5*(\wi,\hgt)$) -- ($(A\n0)+0.5*(\wi,-\hgt)$);

\end{tikzpicture}
\caption{The order of the vertices (mostly) covered by the embedding. Firstly, part of $P_i$ is embedded in $W_i\setminus W_{i+1}$, for each even $i$ (\ref{stepA}). Then, $P_1$ is embedded covering $W_1$ (\ref{stepB}). Then, for each even $i$, $P_i$ and $P_{i+1}$ are embedded together covering most of $W_{i+1}$ (\ref{stepC}).}
\label{fig:grid2}
\end{figure}

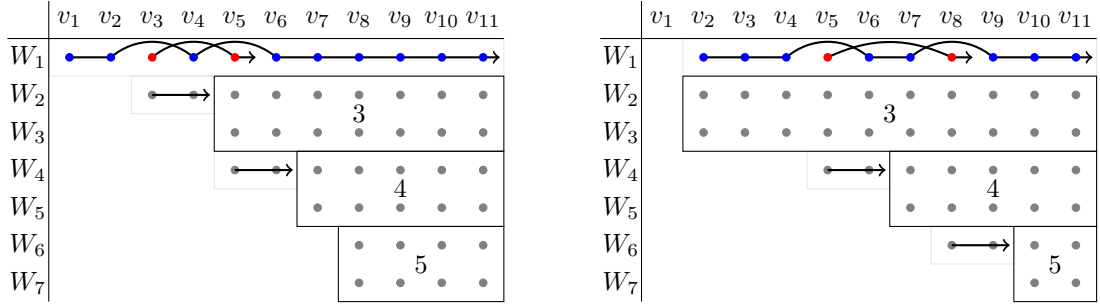
\begin{figure}[p!]
\centering
\gridpicvxfalse\gridpicboxfalse\gridpictreestrue\gridpiconetrue\gridpictwofalse\begin{tikzpicture}

\def\wi{0.55cm}
\def\hgt{0.5cm}
\def\n{7}
\def\m{11}
\def\vx{0.05cm}

\foreach \x in {0,1,...,\n}
\foreach \y in {0,1,...,\m}
{
\coordinate (A\x\y) at ($\x*(0,-\hgt)+\y*(\wi,0)$);
}

\foreach \x in {1,...,\m}
{
\draw (A0\x) node {$v_{\x}$};
}
\foreach \x in {1,...,\n}
{
\draw (A\x0) node {$W_{\x}$};
}

\ifgridpicone
\ifgridpictrees
\foreach \x/\yy in {2/3,3/5,4/5,5/7,6/8,7/8}
{
\foreach \y in {\yy,...,\m}
{
\draw [fill=black!50,black!50] (A\x\y) circle [radius=\vx];
}
}
\fi
\foreach \x/\yy in {1/1,2/3,3/5,4/5,5/7,6/8,7/8}
{
\foreach \y in {\yy,...,\m}
{
\ifgridpicbox
\draw [fill=black!10,black!10] (A\x\y) circle [radius=\vx];
\fi
\ifgridpicvx
\draw [fill] (A\x\y) circle [radius=\vx];
\fi
}
}
\fi

\ifgridpictwo
\ifgridpictrees
\foreach \x/\yy in {2/2,3/2,4/5,5/7,6/8,7/10}
{
\foreach \y in {\yy,...,\m}
{
\draw [fill=black!50,black!50] (A\x\y) circle [radius=\vx];
}
}
\fi
\foreach \x/\yy in {1/2,2/2,3/2,4/5,5/7,6/8,7/10}
{
\foreach \y in {\yy,...,\m}
{
\ifgridpicbox
\draw [fill=black!10,black!10] (A\x\y) circle [radius=\vx];
\fi
\ifgridpicvx
\draw [fill] (A\x\y) circle [radius=\vx];
\fi
}
}
\fi

\ifgridpicone
\foreach \w/\x/\lab in {11/1\m/2,23/24/1,45/46/1}
{
\ifgridpicbox
\draw ($(A\w)+0.5*(-\wi,\hgt)$) -- ($(A\x)+0.5*(\wi,\hgt)$) -- ($(A\x)+0.5*(\wi,-\hgt)$) -- ($(A\w)+0.5*(-\wi,-\hgt)$) -- ($(A\w)+0.5*(-\wi,\hgt)$);
\draw ($0.5*(A\w)+0.5*(A\x)$) node {\lab};
\fi
\ifgridpictrees
\draw [black!10] ($(A\w)+0.5*(-\wi,\hgt)$) -- ($(A\x)+0.5*(\wi,\hgt)$) -- ($(A\x)+0.5*(\wi,-\hgt)$) -- ($(A\w)+0.5*(-\wi,-\hgt)$) -- ($(A\w)+0.5*(-\wi,\hgt)$);
\fi
}
\foreach \w/\x/\y/\z/\lab in {25/2\m/3\m/35/3,47/4\m/5\m/57/4,68/6\m/7\m/78/5}
{
\ifgridpicbox
\draw ($(A\w)+0.5*(-\wi,\hgt)$) -- ($(A\x)+0.5*(\wi,\hgt)$) -- ($(A\y)+0.5*(\wi,-\hgt)$) -- ($(A\z)+0.5*(-\wi,-\hgt)$) -- ($(A\w)+0.5*(-\wi,\hgt)$);
\draw ($0.25*(A\w)+0.25*(A\x)+0.25*(A\y)+0.25*(A\z)$) node {\lab};
\fi
\ifgridpictrees
\draw ($(A\w)+0.5*(-\wi,\hgt)$) -- ($(A\x)+0.5*(\wi,\hgt)$) -- ($(A\y)+0.5*(\wi,-\hgt)$) -- ($(A\z)+0.5*(-\wi,-\hgt)$) -- ($(A\w)+0.5*(-\wi,\hgt)$);
\draw ($0.25*(A\w)+0.25*(A\x)+0.25*(A\y)+0.25*(A\z)$) node {\lab};
\fi
}
\ifgridpictrees
\draw [thick] (A11)--(A12);
\draw [thick] (A12)to [out=40,in=140] (A14);
\draw [thick] (A14)to [out=40,in=140] (A16);
\draw [thick] (A13)to [out=40,in=140] (A15);
\draw [thick,->] (A15)-- ++(0.5*\wi,0);
\draw [thick] (A16)--(A111);
\draw [thick,->] (A111)-- ++(0.4*\wi,0);
\draw [thick] (A23)--(A24);
\draw [thick,->] (A24)-- ++(0.4*\wi,0);
\draw [thick] (A45)--(A46);
\draw [thick,->] (A46)-- ++(0.4*\wi,0);
\fi
\fi

\ifgridpictwo
\foreach \w/\x/\lab in {12/1\m/2,68/69/1,45/46/1}
{
\ifgridpicbox
\draw ($(A\w)+0.5*(-\wi,\hgt)$) -- ($(A\x)+0.5*(\wi,\hgt)$) -- ($(A\x)+0.5*(\wi,-\hgt)$) -- ($(A\w)+0.5*(-\wi,-\hgt)$) -- ($(A\w)+0.5*(-\wi,\hgt)$);
\draw ($0.5*(A\w)+0.5*(A\x)$) node {\lab};
\fi
\ifgridpictrees
\draw [black!10] ($(A\w)+0.5*(-\wi,\hgt)$) -- ($(A\x)+0.5*(\wi,\hgt)$) -- ($(A\x)+0.5*(\wi,-\hgt)$) -- ($(A\w)+0.5*(-\wi,-\hgt)$) -- ($(A\w)+0.5*(-\wi,\hgt)$);
\fi
}
\foreach \w/\x/\y/\z/\lab in {22/2\m/3\m/32/3,47/4\m/5\m/57/4,610/6\m/7\m/710/5}
{
\ifgridpicbox
\draw ($(A\w)+0.5*(-\wi,\hgt)$) -- ($(A\x)+0.5*(\wi,\hgt)$) -- ($(A\y)+0.5*(\wi,-\hgt)$) -- ($(A\z)+0.5*(-\wi,-\hgt)$) -- ($(A\w)+0.5*(-\wi,\hgt)$);
\draw ($0.25*(A\w)+0.25*(A\x)+0.25*(A\y)+0.25*(A\z)$) node {\lab};
\fi
\ifgridpictrees
\draw ($(A\w)+0.5*(-\wi,\hgt)$) -- ($(A\x)+0.5*(\wi,\hgt)$) -- ($(A\y)+0.5*(\wi,-\hgt)$) -- ($(A\z)+0.5*(-\wi,-\hgt)$) -- ($(A\w)+0.5*(-\wi,\hgt)$);
\draw ($0.25*(A\w)+0.25*(A\x)+0.25*(A\y)+0.25*(A\z)$) node {\lab};
\fi
}
\ifgridpictrees
\draw [thick] (A12)--(A14);
\draw [thick] (A14)to [out=40,in=140] (A16);
\draw [thick] (A16) -- (A17);
\draw [thick] (A17)to [out=40,in=140] (A19);
\draw [thick] (A15)to [out=25,in=155] (A18);
\draw [thick,->] (A18)-- ++(0.5*\wi,0);
\draw [thick] (A19)--(A111);
\draw [thick,->] (A111)-- ++(0.4*\wi,0);
\draw [thick] (A45)--(A46);
\draw [thick,->] (A46)-- ++(0.4*\wi,0);
\draw [thick] (A68)--(A69);
\draw [thick,->] (A69)-- ++(0.4*\wi,0);
\fi
\fi

\ifgridpicone
\ifgridpictrees
\foreach \y in {11,10,9,8,7,6,4,2,1}
{
\draw [fill=blue,blue] (A1\y) circle [radius=\vx];
}
\foreach \y in {5,3}
{
\draw [red,fill=red] (A1\y) circle [radius=\vx];
}
\fi
\fi
\ifgridpictwo
\ifgridpictrees
\foreach \y in {11,10,9,7,6,4,3,2}
{
\draw [fill=blue,blue] (A1\y) circle [radius=\vx];
}
\foreach \y in {5,8}
{
\draw [red,fill=red] (A1\y) circle [radius=\vx];
}
\fi
\fi

\draw ($(A00)-0.5*(\wi,\hgt)$) -- ($(A0\m)+0.5*(\wi,-\hgt)$);
\draw ($(A00)+0.5*(\wi,\hgt)$) -- ($(A\n0)+0.5*(\wi,-\hgt)$);

\end{tikzpicture}\hspace{1cm}\gridpictwotrue\gridpiconefalse\begin{tikzpicture}

\def\wi{0.55cm}
\def\hgt{0.5cm}
\def\n{7}
\def\m{11}
\def\vx{0.05cm}

\foreach \x in {0,1,...,\n}
\foreach \y in {0,1,...,\m}
{
\coordinate (A\x\y) at ($\x*(0,-\hgt)+\y*(\wi,0)$);
}

\foreach \x in {1,...,\m}
{
\draw (A0\x) node {$v_{\x}$};
}
\foreach \x in {1,...,\n}
{
\draw (A\x0) node {$W_{\x}$};
}

\ifgridpicone
\ifgridpictrees
\foreach \x/\yy in {2/3,3/5,4/5,5/7,6/8,7/8}
{
\foreach \y in {\yy,...,\m}
{
\draw [fill=black!50,black!50] (A\x\y) circle [radius=\vx];
}
}
\fi
\foreach \x/\yy in {1/1,2/3,3/5,4/5,5/7,6/8,7/8}
{
\foreach \y in {\yy,...,\m}
{
\ifgridpicbox
\draw [fill=black!10,black!10] (A\x\y) circle [radius=\vx];
\fi
\ifgridpicvx
\draw [fill] (A\x\y) circle [radius=\vx];
\fi
}
}
\fi

\ifgridpictwo
\ifgridpictrees
\foreach \x/\yy in {2/2,3/2,4/5,5/7,6/8,7/10}
{
\foreach \y in {\yy,...,\m}
{
\draw [fill=black!50,black!50] (A\x\y) circle [radius=\vx];
}
}
\fi
\foreach \x/\yy in {1/2,2/2,3/2,4/5,5/7,6/8,7/10}
{
\foreach \y in {\yy,...,\m}
{
\ifgridpicbox
\draw [fill=black!10,black!10] (A\x\y) circle [radius=\vx];
\fi
\ifgridpicvx
\draw [fill] (A\x\y) circle [radius=\vx];
\fi
}
}
\fi

\ifgridpicone
\foreach \w/\x/\lab in {11/1\m/2,23/24/1,45/46/1}
{
\ifgridpicbox
\draw ($(A\w)+0.5*(-\wi,\hgt)$) -- ($(A\x)+0.5*(\wi,\hgt)$) -- ($(A\x)+0.5*(\wi,-\hgt)$) -- ($(A\w)+0.5*(-\wi,-\hgt)$) -- ($(A\w)+0.5*(-\wi,\hgt)$);
\draw ($0.5*(A\w)+0.5*(A\x)$) node {\lab};
\fi
\ifgridpictrees
\draw [black!10] ($(A\w)+0.5*(-\wi,\hgt)$) -- ($(A\x)+0.5*(\wi,\hgt)$) -- ($(A\x)+0.5*(\wi,-\hgt)$) -- ($(A\w)+0.5*(-\wi,-\hgt)$) -- ($(A\w)+0.5*(-\wi,\hgt)$);
\fi
}
\foreach \w/\x/\y/\z/\lab in {25/2\m/3\m/35/3,47/4\m/5\m/57/4,68/6\m/7\m/78/5}
{
\ifgridpicbox
\draw ($(A\w)+0.5*(-\wi,\hgt)$) -- ($(A\x)+0.5*(\wi,\hgt)$) -- ($(A\y)+0.5*(\wi,-\hgt)$) -- ($(A\z)+0.5*(-\wi,-\hgt)$) -- ($(A\w)+0.5*(-\wi,\hgt)$);
\draw ($0.25*(A\w)+0.25*(A\x)+0.25*(A\y)+0.25*(A\z)$) node {\lab};
\fi
\ifgridpictrees
\draw ($(A\w)+0.5*(-\wi,\hgt)$) -- ($(A\x)+0.5*(\wi,\hgt)$) -- ($(A\y)+0.5*(\wi,-\hgt)$) -- ($(A\z)+0.5*(-\wi,-\hgt)$) -- ($(A\w)+0.5*(-\wi,\hgt)$);
\draw ($0.25*(A\w)+0.25*(A\x)+0.25*(A\y)+0.25*(A\z)$) node {\lab};
\fi
}
\ifgridpictrees
\draw [thick] (A11)--(A12);
\draw [thick] (A12)to [out=40,in=140] (A14);
\draw [thick] (A14)to [out=40,in=140] (A16);
\draw [thick] (A13)to [out=40,in=140] (A15);
\draw [thick,->] (A15)-- ++(0.5*\wi,0);
\draw [thick] (A16)--(A111);
\draw [thick,->] (A111)-- ++(0.4*\wi,0);
\draw [thick] (A23)--(A24);
\draw [thick,->] (A24)-- ++(0.4*\wi,0);
\draw [thick] (A45)--(A46);
\draw [thick,->] (A46)-- ++(0.4*\wi,0);
\fi
\fi

\ifgridpictwo
\foreach \w/\x/\lab in {12/1\m/2,68/69/1,45/46/1}
{
\ifgridpicbox
\draw ($(A\w)+0.5*(-\wi,\hgt)$) -- ($(A\x)+0.5*(\wi,\hgt)$) -- ($(A\x)+0.5*(\wi,-\hgt)$) -- ($(A\w)+0.5*(-\wi,-\hgt)$) -- ($(A\w)+0.5*(-\wi,\hgt)$);
\draw ($0.5*(A\w)+0.5*(A\x)$) node {\lab};
\fi
\ifgridpictrees
\draw [black!10] ($(A\w)+0.5*(-\wi,\hgt)$) -- ($(A\x)+0.5*(\wi,\hgt)$) -- ($(A\x)+0.5*(\wi,-\hgt)$) -- ($(A\w)+0.5*(-\wi,-\hgt)$) -- ($(A\w)+0.5*(-\wi,\hgt)$);
\fi
}
\foreach \w/\x/\y/\z/\lab in {22/2\m/3\m/32/3,47/4\m/5\m/57/4,610/6\m/7\m/710/5}
{
\ifgridpicbox
\draw ($(A\w)+0.5*(-\wi,\hgt)$) -- ($(A\x)+0.5*(\wi,\hgt)$) -- ($(A\y)+0.5*(\wi,-\hgt)$) -- ($(A\z)+0.5*(-\wi,-\hgt)$) -- ($(A\w)+0.5*(-\wi,\hgt)$);
\draw ($0.25*(A\w)+0.25*(A\x)+0.25*(A\y)+0.25*(A\z)$) node {\lab};
\fi
\ifgridpictrees
\draw ($(A\w)+0.5*(-\wi,\hgt)$) -- ($(A\x)+0.5*(\wi,\hgt)$) -- ($(A\y)+0.5*(\wi,-\hgt)$) -- ($(A\z)+0.5*(-\wi,-\hgt)$) -- ($(A\w)+0.5*(-\wi,\hgt)$);
\draw ($0.25*(A\w)+0.25*(A\x)+0.25*(A\y)+0.25*(A\z)$) node {\lab};
\fi
}
\ifgridpictrees
\draw [thick] (A12)--(A14);
\draw [thick] (A14)to [out=40,in=140] (A16);
\draw [thick] (A16) -- (A17);
\draw [thick] (A17)to [out=40,in=140] (A19);
\draw [thick] (A15)to [out=25,in=155] (A18);
\draw [thick,->] (A18)-- ++(0.5*\wi,0);
\draw [thick] (A19)--(A111);
\draw [thick,->] (A111)-- ++(0.4*\wi,0);
\draw [thick] (A45)--(A46);
\draw [thick,->] (A46)-- ++(0.4*\wi,0);
\draw [thick] (A68)--(A69);
\draw [thick,->] (A69)-- ++(0.4*\wi,0);
\fi
\fi

\ifgridpicone
\ifgridpictrees
\foreach \y in {11,10,9,8,7,6,4,2,1}
{
\draw [fill=blue,blue] (A1\y) circle [radius=\vx];
}
\foreach \y in {5,3}
{
\draw [red,fill=red] (A1\y) circle [radius=\vx];
}
\fi
\fi
\ifgridpictwo
\ifgridpictrees
\foreach \y in {11,10,9,7,6,4,3,2}
{
\draw [fill=blue,blue] (A1\y) circle [radius=\vx];
}
\foreach \y in {5,8}
{
\draw [red,fill=red] (A1\y) circle [radius=\vx];
}
\fi
\fi

\draw ($(A00)-0.5*(\wi,\hgt)$) -- ($(A0\m)+0.5*(\wi,-\hgt)$);
\draw ($(A00)+0.5*(\wi,\hgt)$) -- ($(A\n0)+0.5*(\wi,-\hgt)$);

\end{tikzpicture}
\caption{Steps \ref{stepA} and \ref{stepB} of the embedding when all the path-like trees are paths. The arrows indicate the embedded edge is connected to $X=\{v_{r+1},\ldots,v_n\}$.}
\label{fig:grid3}
\end{figure}

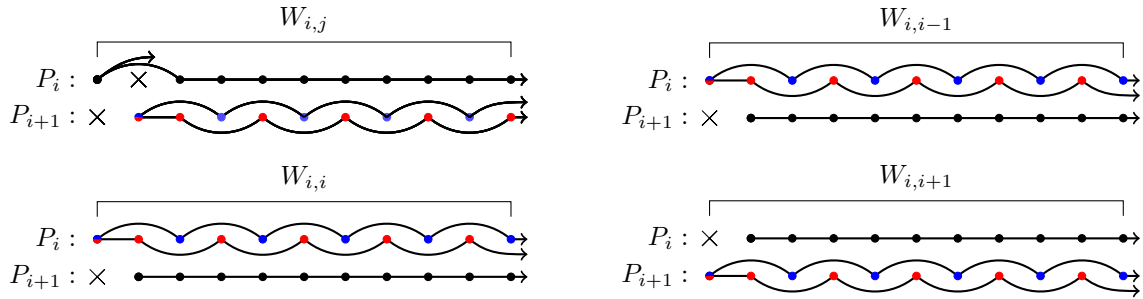
\begin{figure}[p!]
\centering
\begin{tikzpicture}

\def\wi{0.55cm}
\def\hgt{0.5cm}
\def\n{2}
\def\m{11}
\def\vx{0.05cm}

\foreach \x in {1,...,\n}
\foreach \y in {1,...,\m}
{
\coordinate (A\x\y) at ($\x*(0,-\hgt)+\y*(\wi,0)$);
}

\draw ($(A11)-(\wi,0)$) node {$P_i:$};
\draw ($(A21)-(\wi,0)-(0.175,0)$) node {$P_{i+1}:$};

\draw ($(A11)+(0,0.3)$) -- ($(A11)+(0,0.5)$) -- ($(A111)+(0,0.5)$) -- ($(A111)+(0,0.3)$);
\draw ($0.5*(A11)+0.5*(A111)+(0,0.8)$) node {$W_{i,j}$};

\foreach \x in {1}
\foreach \y in {3,...,\m}
{
\draw [fill] (A\x\y)  circle [radius=\vx];
}
\foreach \y in {10,8,6,4}
{

\draw [fill] (A11)  circle [radius=\vx];
\draw [fill] (A12)  -- ++(0.1,0.1) -- ++(-0.2,-0.2)  -- ++(0.1,0.1)  -- ++(0.1,-0.1)  -- ++(-0.2,0.2);
\draw [fill] (A21)  -- ++(0.1,0.1) -- ++(-0.2,-0.2)  -- ++(0.1,0.1)  -- ++(0.1,-0.1)  -- ++(-0.2,0.2);

\draw [thick] (A1\m) -- (A13);
\draw [thick,->] (A1\m) -- ++(0.4*\wi,0);
\draw [thick] (A11)to [out=40,in=140] (A13);
\draw [thick,->] (A11)to [out=40,in=180] ($(A12)+(0,0.3)+(0.4*\wi,0)$);

\foreach \x/\y in {2/4,4/6,6/8,8/10}
{
\draw [thick] (A2\x)to [out=40,in=140] (A2\y);
}
\foreach \x/\y in {3/5,5/7,7/9,9/11}
{
\draw [thick] (A2\x)to [out=-40,in=-140] (A2\y);
}
\draw [thick] (A22) -- (A23);
\draw [thick,->] (A2\m) -- ++(0.4*\wi,0);
\draw [thick,->] (A210)to [out=40,in=180] ($(A211)+(0,0.2)+(0.4*\wi,0)$);

\draw [blue!70,fill=blue!70] (A2\y)  circle [radius=\vx];
}
\foreach \y in {11,9,7,5,3}
{
\draw [red,fill=red] (A2\y)  circle [radius=\vx];
}
\fill[blue] ($(A22)$) -- ++(180:\vx) arc (180:0:\vx) -- cycle;
\draw [blue] ($(A22)-(\vx,0)$) arc (180:0:\vx);
\fill[red] ($(A22)$) -- ++(180:\vx) arc (-180:0:\vx) -- cycle;
\draw [red] ($(A22)-(\vx,0)$) arc (-180:0:\vx);
\end{tikzpicture}\hspace{1cm}\begin{tikzpicture}

\def\wi{0.55cm}
\def\hgt{0.5cm}
\def\n{2}
\def\m{11}
\def\vx{0.05cm}

\foreach \x in {1,...,\n}
\foreach \y in {1,...,\m}
{
\coordinate (A\x\y) at ($\x*(0,-\hgt)+\y*(\wi,0)$);
}

\draw ($(A11)-(\wi,0)$) node {$P_i:$};
\draw ($(A21)-(\wi,0)-(0.175,0)$) node {$P_{i+1}:$};

\draw ($(A11)+(0,0.3)$) -- ($(A11)+(0,0.5)$) -- ($(A111)+(0,0.5)$) -- ($(A111)+(0,0.3)$);
\draw ($0.5*(A11)+0.5*(A111)+(0,0.8)$) node {$W_{i,i-1}$};

\foreach \x in {2}
\foreach \y in {2,...,\m}
{
\draw [fill] (A\x\y)  circle [radius=\vx];
}

\draw [fill] (A11)  circle [radius=\vx];
\draw (A22)  circle [radius=\vx];
\draw [fill] (A21)  -- ++(0.1,0.1) -- ++(-0.2,-0.2)  -- ++(0.1,0.1)  -- ++(0.1,-0.1)  -- ++(-0.2,0.2);

\draw [thick] (A2\m) -- (A22);
\draw [thick,->] (A2\m) -- ++(0.4*\wi,0);

\foreach \x/\y in {1/3,3/5,5/7,7/9,9/11}
{
\draw [thick] (A1\x)to [out=40,in=140] (A1\y);
}
\foreach \x/\y in {2/4,4/6,6/8,8/10}
{
\draw [thick] (A1\x)to [out=-40,in=-140] (A1\y);
}
\draw [thick] (A11) -- (A12);
\draw [thick,->] (A1\m) -- ++(0.4*\wi,0);
\draw [thick,->] (A110)to [out=-40,in=180] ($(A111)-(0,0.2)+(0.4*\wi,0)$);

\fill[blue] ($(A11)$) -- ++(180:\vx) arc (180:0:\vx) -- cycle;
\draw [blue] ($(A11)-(\vx,0)$) arc (180:0:\vx);
\fill[red] ($(A11)$) -- ++(180:\vx) arc (-180:0:\vx) -- cycle;
\draw [red] ($(A11)-(\vx,0)$) arc (-180:0:\vx);

\foreach \y in {10,8,6,4,2}
{
\draw [red,fill=red] (A1\y)  circle [radius=\vx];
}
\foreach \y in {11,9,7,5,3}
{
\draw [blue,fill=blue] (A1\y)  circle [radius=\vx];
}
\end{tikzpicture}

\vspace{0.2cm}

\begin{tikzpicture}

\def\wi{0.55cm}
\def\hgt{0.5cm}
\def\n{2}
\def\m{11}
\def\vx{0.05cm}

\foreach \x in {1,...,\n}
\foreach \y in {1,...,\m}
{
\coordinate (A\x\y) at ($\x*(0,-\hgt)+\y*(\wi,0)$);
}

\draw ($(A11)-(\wi,0)$) node {$P_i:$};
\draw ($(A21)-(\wi,0)-(0.175,0)$) node {$P_{i+1}:$};

\draw ($(A11)+(0,0.3)$) -- ($(A11)+(0,0.5)$) -- ($(A111)+(0,0.5)$) -- ($(A111)+(0,0.3)$);
\draw ($0.5*(A11)+0.5*(A111)+(0,0.8)$) node {$W_{i,i}$};

\foreach \x in {2}
\foreach \y in {2,...,\m}
{
\draw [fill] (A\x\y)  circle [radius=\vx];
}
\foreach \y in {10,8,6,4,2}
{
}
\foreach \y in {11,9,7,5,3,1}
{
}
\draw [fill] (A11)  circle [radius=\vx];
\draw (A22)  circle [radius=\vx];
\draw [fill] (A21)  -- ++(0.1,0.1) -- ++(-0.2,-0.2)  -- ++(0.1,0.1)  -- ++(0.1,-0.1)  -- ++(-0.2,0.2);

\draw [thick] (A2\m) -- (A22);
\draw [thick,->] (A2\m) -- ++(0.4*\wi,0);

\foreach \x/\y in {1/3,3/5,5/7,7/9,9/11}
{
\draw [thick] (A1\x)to [out=40,in=140] (A1\y);
}
\foreach \x/\y in {2/4,4/6,6/8,8/10}
{
\draw [thick] (A1\x)to [out=-40,in=-140] (A1\y);
}
\draw [thick] (A11) -- (A12);
\draw [thick,->] (A1\m) -- ++(0.4*\wi,0);
\draw [thick,->] (A110)to [out=-40,in=180] ($(A111)-(0,0.2)+(0.4*\wi,0)$);

\fill[blue] ($(A11)$) -- ++(180:\vx) arc (180:0:\vx) -- cycle;
\draw [blue] ($(A11)-(\vx,0)$) arc (180:0:\vx);
\fill[red] ($(A11)$) -- ++(180:\vx) arc (-180:0:\vx) -- cycle;
\draw [red] ($(A11)-(\vx,0)$) arc (-180:0:\vx);

\foreach \y in {10,8,6,4,2}
{
\draw [red,fill=red] (A1\y)  circle [radius=\vx];
}
\foreach \y in {11,9,7,5,3}
{
\draw [blue,fill=blue] (A1\y)  circle [radius=\vx];
}
\end{tikzpicture}\hspace{1cm}\begin{tikzpicture}

\def\wi{0.55cm}
\def\hgt{0.5cm}
\def\n{2}
\def\m{11}
\def\vx{0.05cm}

\foreach \x in {1,...,\n}
\foreach \y in {1,...,\m}
{
\coordinate (A\x\y) at ($-\x*(0,-\hgt)+\y*(\wi,0)+3*(\wi,0)$);
}

\draw ($(A21)-(\wi,0)$) node {$P_i:$};
\draw ($(A11)-(\wi,0)-(0.175,0)$) node {$P_{i+1}:$};

\draw ($(A21)+(0,0.3)$) -- ($(A21)+(0,0.5)$) -- ($(A211)+(0,0.5)$) -- ($(A211)+(0,0.3)$);
\draw ($0.5*(A21)+0.5*(A211)+(0,0.8)$) node {$W_{i,i+1}$};

\foreach \x in {2}
\foreach \y in {2,...,\m}
{
\draw [fill] (A\x\y)  circle [radius=\vx];
}
\foreach \y in {10,8,6,4,2}
{
}
\foreach \y in {11,9,7,5,3,1}
{
}
\draw [fill] (A11)  circle [radius=\vx];
\draw (A22)  circle [radius=\vx];
\draw [fill] (A21)  -- ++(0.1,0.1) -- ++(-0.2,-0.2)  -- ++(0.1,0.1)  -- ++(0.1,-0.1)  -- ++(-0.2,0.2);

\draw [thick] (A2\m) -- (A22);
\draw [thick,->] (A2\m) -- ++(0.4*\wi,0);

\foreach \x/\y in {1/3,3/5,5/7,7/9,9/11}
{
\draw [thick] (A1\x)to [out=40,in=140] (A1\y);
}
\foreach \x/\y in {2/4,4/6,6/8,8/10}
{
\draw [thick] (A1\x)to [out=-40,in=-140] (A1\y);
}
\draw [thick] (A11) -- (A12);
\draw [thick,->] (A1\m) -- ++(0.4*\wi,0);
\draw [thick,->] (A110)to [out=-40,in=180] ($(A111)-(0,0.2)+(0.4*\wi,0)$);

\fill[blue] ($(A11)$) -- ++(180:\vx) arc (180:0:\vx) -- cycle;
\draw [blue] ($(A11)-(\vx,0)$) arc (180:0:\vx);
\fill[red] ($(A11)$) -- ++(180:\vx) arc (-180:0:\vx) -- cycle;
\draw [red] ($(A11)-(\vx,0)$) arc (-180:0:\vx);

\foreach \y in {10,8,6,4,2}
{
\draw [red,fill=red] (A1\y)  circle [radius=\vx];
}
\foreach \y in {11,9,7,5,3}
{
\draw [blue,fill=blue] (A1\y)  circle [radius=\vx];
}
\end{tikzpicture}
\caption{Embedding part of $P_i$ and $P_{i+1}$ together to cover $W_{i,j}$ in the four cases $j\leq i-2$, $j=i-1$, $j=i$ and $j=i+1$, while omitting only the vertices marked by a $\times$.}
\label{fig:turnaround1}
\end{figure}


\noindent For each $i\in [r]$, $W_i\cup X$ contains every vertex in $V(K_n)$ except for one vertex for each star $S_j$ that is larger than $P_i$ (which together form a leftmost interval in the ordering $v_1,\ldots,v_n$), and therefore
 \begin{align}
 |W_i\cup X|-|P_i|&=\big(n-|\{j\in [r]:|P_i|<|S_j|\}|\big)-\big(n-|\{j\in [r]:|P_i|<|S_j|\}|-|\{j\in [r]:|P_i|<|P_j|\}|\big)\nonumber\\
 &=|\{j\in [r]:|P_i|<|P_j|\}|=i-1,\label{eqn:WiXsketch}
 \end{align}
so we have the following principle for each $i\in [r]$.

 \stepcounter{propcounter}
 \begin{enumerate}[label = \textbf{\Alph{propcounter}}]
\item If $P_i$ is embedded into $K_n[W_i\cup X]$ then there are exactly $(i-1)$ unused vertices in $W_i\cup X$.\label{prop:embeddingprinciple}
\end{enumerate}

As we will see (using a weaker condition, even), if $P_1,\ldots,P_r$ are embedded so that each vertex in $W_i$ has at most 1 rightward neighbour (according to our ordering $v_1,\dots,v_n$) in the embedding of each $P_i$, then each $v_j$ with $j\in [r]$ will have enough remaining rightward neighbours to embed $S_j$ with $v_j$ as its centre and using edges going to the right from $v_j$ (so that these stars are then disjoint). It is easy to embed $P_1$ like this: starting by embedding an arbitrary vertex of $P_1$ to $v_n$, embed $P_1$ vertex by vertex, each time attaching a neighbour of an already embedded vertex of $P_1$ using the right-most unoccupied vertex in $v_1,\ldots,v_n$. It is even fairly straightforward to embed $P_2$ like this while avoiding the edges of such an embedding of $P_1$, but after this it becomes increasingly difficult to embed the trees. Therefore, essentially (omitting only one initial step we introduce later in the sketch), after embedding $P_1$ we embed the trees in pairs $P_i$ and $P_{i+1}$ (for increasing even $i$), allowing some vertices in the embeddings to have two rightward neighbours in one embedding if they have no rightward neighbours in the other embedding (as in Figure~\ref{fig:turnaround1}). In fact, if a vertex has two rightward neighbours in one embedding then we will use \ref{prop:embeddingprinciple} to omit it entirely from the other embedding. To summarise,
we will embed the trees $P_1,\ldots,P_r$ disjointly into $K_n$ under the following rule (as well as \ref{prop:rule1}).

\begin{enumerate}[label = \textbf{A\arabic{enumi}}]\stepcounter{enumi}
\item For each $i,j\in [r]$ with $v_j\in W_i$, the following hold.\label{prop:rule2}
\begin{itemize}
\item If $i=1$, then $v_j$ has at most 1 rightward neighbour in the embedding of $P_i$.
\item If $i$ is even and $v_j\notin W_{i+1}$, then $v$ has at most 1 rightward neighbour in the embedding of $P_i$.
\item If $i$ is even and $v_j\in W_{i+1}$, then $v$ has at most 2 rightward neighbours in total in the embeddings of $P_i$ and $P_{i+1}$.
\end{itemize}
\end{enumerate}

If the trees $P_1,\ldots,P_r$ are embedded obeying \ref{prop:rule1} and \ref{prop:rule2}, then we can embed the stars $S_1,\ldots,S_r$ as each vertex $v_i$, $i\in [r]$, will have at least $|S_i|-1$ neighbouring rightward edges which have not been used in the embeddings of $P_1,\ldots,P_r$. To see this, take $i\in [r]$ and let $j$ be the largest $j\in [r]$ such that $v_i\in W_j$, or, equally, $j$ is the number of path-like trees larger than $S_i$ in the TPC sequence. By \ref{prop:rule1} and \ref{prop:rule2}, $v_i$ is in at most $j$ rightward edges in the embeddings of $P_1,\ldots,P_r$, and therefore has at least $n-i-j$ rightward neighbouring edges that have not been used. However, there are $i+j-1$ trees which appear in the TPC sequence before $S_i$, and therefore $v_i$ has at least $|S_i|-1$ unused rightward neighbouring edges.

Thus, it is sufficient to embed $P_1,\ldots,P_r$ such that \ref{prop:rule1} and \ref{prop:rule2} hold. The rough form of our embedding will be that, to embed say $P_i$ and $P_{i+1}$ for some even $i$, we look at the embedding of $P_1,\ldots,P_{i-1}$, and deduce from \ref{prop:rule2} that $W_i$ can be partitioned into $i+1$ sets as $W_{i,1}\cup \ldots\cup W_{i,i+1}$ so that these sets are independent sets in the graph $H_i$ of the previously embedded edges. Then, for each $j\in [i+1]$, letting $G_i$ be the complement of $H_i$ in $K_n$, we can embed part of the trees $P_i$ and $P_{i+1}$ in $G_i$ to cover (most of) the vertices in $W_i\cap W_{i,j}$ and $W_{i+1}\cap W_{i,j}$, respectively, where it is very helpful for this embedding that $G_i[W_{i,j}]$ is a complete graph.
To discuss the parts of the trees we use for this, let us assume for now that $P_1,\ldots,P_r$ are all paths (the more general case is discussed at the end of the sketch). We will also focus on embedding the parts to cover vertices in $\{v_1,\ldots,v_r\}$ and postpone to the end of the sketch discussing how they can be embedded with the rest of the tree so that the parts are correctly connected up and that all the vertices in $X=\{v_{r+1},\ldots,v_n\}$ are used.

Assume then, that each $P_i$, $i\in [r]$, is a path. Each $P_i$ has two `ends' (see Definition~\ref{defn:ends}) -- here, subpaths which end in a leaf of $P_i$ -- which are easy to embed in a complete subgraph of $K_n$ like $G_i[W_{i,j}]$ so that every vertex has one rightward neighbour and the vertex which needs to be connected to the rest of the embedding of $P_i$ is in $X$. However, we have in general more sets $W_{i,j}$ than ends in $P_i$, so we have to use what we call  `artificial ends' (see Definition~\ref{defn:ends}) -- here, long subpaths of $P_i$ -- which have 2 vertices that need to be embedded to $X$ to be connected to the rest of the embedding of $P_i$. We can embed such an artificial end into a complete subgraph of $K_n$ like $G_i[W_{i,j}]$ as if it were an end, but so that one vertex (the leftmost vertex) has a second rightward edge. We think of this vertex as the `turnaround vertex', and note that it poses no trouble to \ref{prop:rule2}, as long as it is then omitted from the embedding of $P_{i+1}$ (see Figure~\ref{fig:turnaround1}).

We now introduce the added initial step before embedding $P_1$. Note that, with \ref{prop:rule2} in mind, for each even $i\in [r]$, 
the vertices in $W_i\setminus W_{i+1}$ must have just one (or zero) rightward neighbours in $P_i$; in particular, we must avoid these vertices being one of our `turnaround' vertices.
Therefore, we start our embedding by, for each even $i\in [r]$, embedding part of $P_i$ (one of its ends) to cover $W_i\setminus W_{i+1}$, guaranteeing that each of these vertices have exactly one rightward neighbour (note that these embeddings are vertex-disjoint). Including this, the steps of the embedding are then the following (see Figures~\ref{fig:grid2} and~\ref{fig:grid3}).

\begin{enumerate}[label = \textbf{\Roman{enumi}}]
\item For each even $i\in [r]$, we embed part of $P_i$ to cover $W_i\setminus W_{i+1}$.\label{stepA}
\item We embed $P_1$.\label{stepB}
\item For each even $i\in [r]$ in turn, we embed the rest of $P_i$ and $P_{i+1}$ together.\label{stepC}
\end{enumerate}

After \ref{stepA}, let $H_0$ be the union of the embedded parts of the trees, noting that $H_0[W_1]$ is a forest (if each $P_i$ is a path then it is a path forest), and therefore has chromatic number at most 2. Thus, to help us perform~\ref{stepB}, we can partition $W_1=W_{1,1}\cup W_{1,2}$ so that $H_0[W_{1,1}]$ and $H_0[W_{1,2}]$ are independent sets. 
To embed $P_1$, using vertex set $W_1\cup X$ (as required by~\ref{prop:rule1}),  such that each vertex in $W_1$ has just one rightward edge (as required by \ref{prop:rule2}), we can use the two ends of $P_1$ to cover $W_{1,1}$ and $W_{1,2}$ respectively (see Figure~\ref{fig:grid3}).

Now, suppose $i\in [r]$ is even and we have embedded $P_j$ for each $2\leq j<i$ in addition to carrying out steps \ref{stepA} and \ref{stepB}, obeying \ref{prop:rule1} and \ref{prop:rule2}, and wish to embed $P_i$ and $P_{i+1}$ (as in~\ref{stepC}). Let $H_i$ be the set of edges embedded so far. Note that every vertex in $W_{i+1}$ has at most $i$ rightward edges in $H_i$ --- at most one in the embedding of $P_1$, at most two in the embedding of $P_j$ and $P_{j+1}$ for each even $j<i$ (for at most $i-2$ in total), and at most one in $H_0$. Therefore $H_i[W_{i+1}]$ has chromatic number at most $i+1$, and thus can be partitioned as $W_{i,1}\cup \ldots\cup W_{i,i+1}$ so that each set of this partition is an independent set in $H_i$.

Then, for each $j\in [i+1]$, we can take an artificial end in each of $P_i$ and $P_{i+1}$ and embed these together so that the turnaround vertex in one embedding is omitted from the other embedding, and vice versa, and no other vertex in $W_{i,j}$ is omitted, using the first pattern in  Figure~\ref{fig:turnaround1}. However, we actually cannot do this for all $j\in [i+1]$ as it will omit slightly too many vertices from the embedding for \ref{prop:embeddingprinciple}, and therefore we only do this for each $j\leq i-2$. For each $j$ with $i-1\leq j\leq i+1$, we use that we have one end of $P_i$ and two ends of $P_{i+1}$ that we have yet not embedded (only one end of $P_i$ is embedded so far, in \ref{stepA}).
As depicted in Figure~\ref{fig:turnaround1}, we can embed an end of $P_{i+1}$ into $W_{i,i-1}$ to cover all but the leftmost vertex which we can then use as the turnaround vertex to embed an artificial end of $P_i$ to cover all of the vertices in $W_{i,i-1}$. Similarly, we take the remaining end of $P_{i+1}$ and do likewise for $W_{i,i}$, omitting again the leftmost vertex of $W_{i,i}$ in the embedding. Finally, for $W_{i,i+1}$, we do the same with $i$ and $i+1$ swapped, using the remaining end of $P_i$ and omitting the leftmost vertex from $W_{i,i+1}$ in the embedding. Over all $j\in [i+1]$, then, we will have omitted $(i-2)+1=i-1$ vertices from $W_i\setminus W_{i+1}$ in the embedding of parts of $P_i$ and $(i-2)+2=i$ vertices from $W_i\setminus W_{i+1}$ in the embedding of parts of $P_{i+1}$. Thus, by \ref{prop:embeddingprinciple}, there are enough unused vertices in $X$ to complete the embeddings of both $P_i$ and $P_{i+1}$, if we can connect the embedded parts together correctly and embed the remaining vertices, so that these embeddings are edge-disjoint and use edges in $G_i$ (which is actually rather straightforward in the special case when the trees $P_i$ are all paths).

\medskip

\noindent\textbf{Non-path path-like trees.} To carry out (something like) the above embedding when, more generally, the trees  $P_1,\ldots,P_r$ are only `path-like' (with at most $\lambda n$ leaves), we take each tree $P_i$ and decompose it into pieces, finding many small pieces that are `ends' or `artificial ends'. An end, effectively, is part of the tree we can embed into a complete subgraph of $K_n$ so that each vertex has only one rightward neighbour, while, roughly, for an artificial end we can do so similarly, but must use one `turnaround vertex', which will be the leftmost vertex covered. Ends and artificial ends are defined precisely in Definition~\ref{defn:ends} and embedded in Section~\ref{section_ends}, but to keep things simpler here we only include Figure~\ref{fig:morecomplicated} as an example of embedding a more complicated artificial end for $P_i$ by embedding it from the right with a turnaround vertex on the far left.

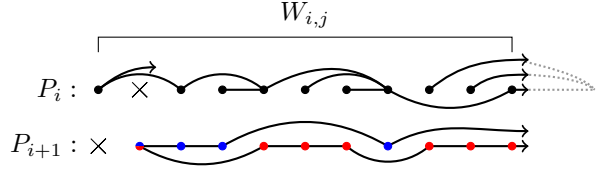
\begin{figure}
\centering
\begin{tikzpicture}

\def\wi{0.55cm}
\def\hgt{0.75cm}
\def\n{2}
\def\m{11}
\def\vx{0.05cm}

\foreach \x in {1,...,\n}
\foreach \y in {1,...,\m}
{
\coordinate (A\x\y) at ($\x*(0,-\hgt)+\y*(\wi,0)$);
}

\draw ($(A11)-(\wi,0)$) node {$P_i:$};
\draw ($(A21)-(\wi,0)-(0.175,0)$) node {$P_{i+1}:$};

\def\lift{0.2}
\draw ($(A11)+(0,0.3)+(0,\lift)$) -- ($(A11)+(0,0.5)+(0,\lift)$) -- ($(A111)+(0,0.5)+(0,\lift)$) -- ($(A111)+(0,0.3)+(0,\lift)$);
\draw ($0.5*(A11)+0.5*(A111)+(0,0.8)+(0,\lift)$) node {$W_{i,j}$};

\foreach \x in {1}
\foreach \y in {3,...,\m}
{
\draw [fill] (A\x\y)  circle [radius=\vx];
}

\draw [fill] (A11)  circle [radius=\vx];
\draw [fill] (A12)  -- ++(0.1,0.1) -- ++(-0.2,-0.2)  -- ++(0.1,0.1)  -- ++(0.1,-0.1)  -- ++(-0.2,0.2);
\draw [fill] (A21)  -- ++(0.1,0.1) -- ++(-0.2,-0.2)  -- ++(0.1,0.1)  -- ++(0.1,-0.1)  -- ++(-0.2,0.2);

\draw [thick,densely dotted,black!40]  ($(A111)+(0,0.2)+(0.4*\wi,0)$) to [out=180,in=140] ($(A111)+(2*\wi,0)$);
\draw [thick,densely dotted,black!40] ($(A111)+(0,0.4)+(0.4*\wi,0)$) to [out=180,in=130]  ($(A111)+(2*\wi,0)$);
\draw [thick,densely dotted,black!40] ($(A111)+(0.4*\wi,0)$) -- ($(A111)+(2*\wi,0)$);

\draw [thick] (A18)to [out=-30,in=-150] (A111);
\draw [thick,->] (A110)to [out=40,in=180] ($(A111)+(0,0.2)+(0.4*\wi,0)$);
\draw [thick,->] (A19)to [out=40,in=180] ($(A111)+(0,0.4)+(0.4*\wi,0)$);
\draw [thick] (A17) -- (A18);
\draw [thick] (A15)to [out=30,in=140] (A18);
\draw [thick] (A16)to [out=40,in=140] (A18);
\draw [thick] (A13)to [out=40,in=140] (A15);
\draw [thick] (A15) -- (A14);
\draw [thick,->] (A1\m) -- ++(0.4*\wi,0);
\draw [thick] (A11)to [out=40,in=140] (A13);
\draw [thick,->] (A11)to [out=40,in=180] ($(A12)+(0,0.3)+(0.4*\wi,0)$);

\foreach \x/\y in {4/8}
{
\draw [thick] (A2\x)to [out=30,in=150] (A2\y);
}
\foreach \x/\y in {7/9}
{
\draw [thick] (A2\x)to [out=-40,in=-140] (A2\y);
}
\foreach \x/\y in {2/5}
{
\draw [thick] (A2\x)to [out=-30,in=-150] (A2\y);
}
\draw [thick] (A22) -- (A24);
\draw [thick] (A211) -- (A29);
\draw [thick] (A25) -- (A27);
\draw [thick,->] (A2\m) -- ++(0.4*\wi,0);
\draw [thick,->] (A28)to [out=30,in=180] ($(A211)+(0,0.2)+(0.4*\wi,0)$);

\foreach \y in {8,4,3}
{
\draw [fill=blue,blue](A2\y)  circle [radius=\vx];
}
\foreach \y in {11,10,9,7,6,5}
{
\draw [fill=red,red] (A2\y)  circle [radius=\vx];
}
\fill[blue] ($(A22)$) -- ++(180:\vx) arc (180:0:\vx) -- cycle;
\draw [blue] ($(A22)-(\vx,0)$) arc (180:0:\vx);
\fill[red] ($(A22)$) -- ++(180:\vx) arc (-180:0:\vx) -- cycle;
\draw [red] ($(A22)-(\vx,0)$) arc (-180:0:\vx);

\end{tikzpicture}
\caption{Embedding part of $P_i$ and $P_{i+1}$ together to cover $W_{i,j}$ when $j\leq i-2$, while omitting only the vertices marked by a $\times$, when the `artificial end' of $P_i$ used is not a path but a more complicated tree. Here the multiple arrows leading to the right make it more difficult to embed this artificial end, so additional vertices in $W_{i,j}$ may be omitted, to be later covered by other leaves of the tree.}\label{fig:morecomplicated}
\end{figure}

The remaining issue to discuss is how these parts of the tree $P_i$ can be embedded along with the  rest of the tree, so that they are all connected up and use all the vertices in $X=\{v_{r+1},\ldots,v_{n}\}$. We do not embed any of the edges of the stars $S_1,\ldots,S_r$ within $K_n[X]$, and as the trees $P_i$ have maximum degree at most $\lambda n$ (due to the limited number of leaves) it is not hard to embed the remaining part of the tree in $K_n[X]$ once the ends and artificial ends are embedded and connected together. The challenge then is to embed the ends and artificial ends appropriately connected together, which in part we make easier by reserving for each tree $P_i$ a set $A_i$ of 20 vertices in $X$ and making sure that when we embed $P_i$ we have used no edges in $K_n$ between $\{v_1,\ldots,v_r\}$ and $A_i$. The main solution to this, however, is to first observe that when we embed an end (or artificial end) to cover most of $W_{i,j}$, the more edges that will lie across the graph between $W_{i,j}$ and $\{v_{r+1},\ldots,v_n\}$, the more challenging this will be (see Figure~\ref{fig:morecomplicated}). However, the more such cross edges we need, the more leaves that end will have. We relax our embedding then, to allow an end (or artificial end) to miss some vertices in $W_{i,j}$ relative to the number of leaves in the artificial end, before later embedding another part of the tree so that the embedded leaves cover these missed vertices.

That is, we effectively start by splitting each path-like tree $P_i$ into five trees $E_{i,1},E_{i,2},F_{i,1}, F_{i,2}$ and $F_{i,3}$ (which pairwise intersect on at most one vertex), so that $E_{i,1}$ and $E_{i,2}$ are ends, $F_{i,2}$ has more leaves than $F_{i,1}$, and $F_{i,1},F_{i,2},F_{i,3}$ are all large trees. We then take $F_{i,1}$ and split it into pieces including $i+1$ artificial ends. Broadly, then, we carry out the embedding as described above, while embedding $E_{i,1}\cup E_{i,2}\cup F_{i,1}$ for each $i\in [r]$, additionally omitting some vertices in the sets $W_{i,j}$, $j\leq i+1$,
 as long as the number of vertices omitted is altogether at most a small proportion of the leaves of $F_{i,1}$. We then embed $F_{i,2}$, if necessary using vertices omitted in the sets $W_{i,j}$ as some leaves of the embedding of $F_{i,2}$ (so that \ref{prop:rule2} is obeyed). This embedding of $F_{i,2}$ is done similarly to how we embed star-like trees which are not stars, so we similarly postpone any further discussion of this, noting only that it is carried out in Section~\ref{section_key5}. In the last part of the embedding, we embed $F_{i,3}$, using this to make sure we use exactly the unused vertices in $X$. We additionally do this embedding so that a vertex with degree at most 2 is embedded to any vertex that has had many more neighbouring edges used than average so far in the embeddings. This allows our embeddings to be done iteratively, without using too many edges around any one vertex.

\medskip

\noindent \textbf{Paper outline.} In Section~\ref{section_prelim}, we cover some preliminaries including the main tree decomposition we use. In Section~\ref{section_mainembeddingscheme}, we cover our main embedding scheme using steps \ref{stepA} to \ref{stepC},
proving Theorem~\ref{theorem_treepacking_starlikestars} subject only to two key lemmas. In Section~\ref{section_starlike}, we show how to embed the star-like trees, thus proving Theorem~\ref{theorem_treepacking} from Theorem~\ref{theorem_treepacking_starlikestars}. Finally, in Section~\ref{section_ends}, we prove the two key lemmas.

\section{Preliminaries}\label{section_prelim}

\subsection{Notation}
Given a graph $G$, $|G|=|V(G)|$ is the number of vertices in $G$ and $e(G)=|E(G)|$ is the number of edges. For each $A\subset V(G)$, $N(A)=(\cup_{v\in A}N(v))\setminus A$, i.e., the set of vertices outside of $A$ adjacent to at least one vertex in $A$. Given $v\in V(G)$ and $A\subset V(G)$, $N(v,A)=N(v)\cap A$ and $d(v,A)=|N(v,A)|$. Given two disjoint sets $A,B\subset V(G)$, $e(A,B)$ is the number of edges with one endvertex in each of $A$ and $B$. The complement of a graph $G$ is the graph with vertex set $V(G)$ and edge set $\{uv:u,v\in V(G), u\neq v,uv\notin E(G)\}$. Given $A\subset V(G)$, $G[A]$ is the subgraph of $G$ induced on the vertex set $A$, and $G-A=G[V(G)\setminus A]$. If $v\in V(G)$, then we set $G-v=G-\{v\}$. If $e\in V(G)^{(2)}$, then $G+e$ and $G-e$ are the graphs with vertex set $V(G)$ and edge set $E(G)\cup\{e\}$ and $E(G)\setminus\{e\}$, respectively.

When considering an embedding $\phi$ of some graph $T$ in another graph $G$, we will sometimes use $\phi(T)$ to denote the embedded graph and $\mathrm{Im}(\phi)$ to denote the set of vertices in the image, thus, $\mathrm{Im}(\phi)=\phi(V(T))=V(\phi(T))$. When we extend an embedding $\phi$ and there is no risk of confusion, we will often abuse notation and also use $\phi$ for the extended embedding. When writing $x\ll y$, we mean that there is some non-decreasing function $f:(0,1]\to (0,1]$ such that if $x\leq f(y)$ then the following statements hold. Where there are more than two variables in a hierarchy, for example $x\ll y\ll z$, the implicit functions are chosen from right to left. 

\subsection{Tree decomposition}
We decompose our trees into edge-disjoint subtrees, calling this a tree decomposition, as follows.

\begin{definition}
Given any tree $T$, we say subgraphs $T_1,\dots, T_k$ of $T$ form a \emph{tree decomposition} of $T$ if they are edge-disjoint subtrees of $T$ such that $\bigcup_{i\in [k]} E(T_i)=E(T)$. The vertices appearing in more than one subtree are the \emph{connectors} of the tree decomposition.
\end{definition}

Our first decomposition lemma simply says that every tree has a decomposition into two trees where one subtree contains a fixed vertex and the other has roughly some desired size.

\begin{lemma}\label{lemma_divisionavoidingavertex}
	Let $T$ be a tree and let $t\in V(T)$. Let $m\in \mathbb{R}$ satisfy $1\leq m< |T|$. Then, $T$ has a tree decomposition into  $T_1$ and $T_2$ such that $m<|T_1|\leq 2m$ and $t\in V(T_2)$.
\end{lemma}
\begin{proof}
	Among all tree decompositions $T_1,T_2$ of $T$ with $|T_1|>m$ and $t\in V(T_2)$, choose one minimising $|T_1|$. (Note that the set of such decompositions is non-empty, as we can simply pick $T_1=T$ and $T_2=T[\{t\}]$.) We will show that $|T_1|\leq 2m$.

	Let $s$ be the unique vertex in $V(T_1)\cap V(T_2)$, and let $N(s)\cap V(T_1)=\{x_1,\dots,x_k\}$. Note that $k\geq 1$ (since if $k=0$ then $V(T_1)=\{s\}$, contradicting $1\leq m<|T_1|$). If $k=1$, then $T_1-s$ and $T_2+sx_1$ give a tree decomposition, so we must have $|T_1|-1\leq m$ and hence $|T_1|\leq m+1\leq 2m$. On the other hand, if $k\geq 2$ then, letting $V_1,\dots,V_k$ be the vertex sets of the connected components of $T_1-s$, we may assume that $|V_k|\leq (|T_1|-1)/k$. As $T_1-V_k$ and $T[V(T_2)\cup V_k]$ give a tree decomposition, we must have  $|T_1|-|V_k|\leq m$. As $|V_k|\leq (|T_1|-1)/k$, we get $(1-1/k)|T_1|\leq m$ and hence $|T_1|\leq m/(1-1/k)\leq 2m$.
\end{proof}

For convenience, we record the following useful corollary of Lemma~\ref{lemma_divisionavoidingavertex}.

\begin{corollary}\label{corollary_decomp_three}
Let $T$ be a tree, and let $m_1,m_2\geq 1$ satisfy $|T|\geq 2m_1+m_2$. Then $T$ has a decomposition into three subtrees $T_1$, $T_2$ and $T_3$ such that $T_1$ and $T_2$ each contain exactly one connector of the decomposition and $m_i<|T_i|\leq 2m_i$ for each $i\in [2]$.
\end{corollary}
\begin{proof}
Picking $t\in V(T)$ arbitrarily, use Lemma~\ref{lemma_divisionavoidingavertex} to take a tree decomposition of $T$ into $T_1$ and $S$ so that $m_1<|T_1|\leq 2m_1$. Let $s$ be the unique connector in this decomposition, and note that $|S|\geq |T|-2m_1+1>m_2$. Using Lemma~\ref{lemma_divisionavoidingavertex}, take a tree decomposition of $S$ into $T_2$ and $T_3$ with $s\in V(T_3)$ and $m_2<|T_2|\leq 2m_2$. Let $s'$ be the unique connector in this decomposition of $S$. Note that $s'\notin V(T_1)$ and $s\notin V(T_2)$ unless $s=s'$, so that $T_1$ and $T_2$ each contain exactly one connector in the decomposition of $T$ into $T_1$, $T_2$ and $T_3$.
\end{proof}

The following lemma is used in the proof of both of the two key lemmas we use in our main embedding. It finds a tree decomposition such that some of the subtrees have in total exactly $m$ non-connector vertices, one of these pieces has size at least $m/2$, and each of these pieces have only one connector.
\begin{restatable}{lemma}{lemmaenddivision}\label{lemma_genuineenddivision}
	Let $t$ be a vertex of a tree $T$ and let $m$ be a positive integer with $1\leq m<|T|$. Then, there exists for some $k\geq 2$ a tree decomposition of $T$ as $T_1,\ldots,T_k$ such that $t\in V(T_k)$, $|T_k|=|T|-m$, $|T_1|-1\geq m/2$, and, for each $1\leq i\leq k-1$, $T_i$ contains exactly one connector.
\end{restatable}
\begin{proof}
We prove this by induction on $m$, for all trees $T$ and vertices $t\in V(T)$. If $m=1$, then, using $|T|>m$, let $T_1$ be a subtree of $T$ having one edge and containing (at least) one leaf of $T$, $s$ say, which is not $t$. Then, $T_2=T-s$ and $T_1$ give the required decomposition of $T$ with $k=2$.

Assume then that $m>1$ and that the result holds for all $m'<m$. Let $t$ be a vertex of a tree $T$ with $|T|>m$.
Using Lemma~\ref{lemma_divisionavoidingavertex} as $1\leq (m+1)/2<|T|$, take a tree decomposition of $T$ into $S$ and $S'$ so that $t\in V(S)$ and $(m+1)/2<|S'|\leq m+1$, and let $s$ be the connector in this decomposition. Note that $|S'|-1\geq m/2$. Let $m'=m-(|S'|-1)$, so $0\leq m'<m$.
If $m'=0$, then $T_1=S'$ and $T_2=S$ gives the required decomposition, so assume that $m'\geq 1$. Therefore, by induction, for some $k'$ there is a decomposition of $S$ into subtrees $T_1,S_1,\ldots,S_{k'}$ such that $t\in V(T_1)$, $|T_1|=|S|-m'$ and, for each $i\in [k']$, $S_i$ contains exactly one connector (note that this vertex must be in $T_1$). By not taking up the size condition on $S_1$, we can assume without loss of generality that, if $s$ is not in $T_1$, then $s$ is in $S_1$ (and hence $s$ is in no other subtree in this decomposition of $S$).
Now, we have that $t\in V(T_1)$ and $|T_1|=|S|-m'=|S|+|S'|-1-m=|T|-m$. If $s$ is in $T_1$ then the sequence $S',S_1,\ldots,S_{k'},T_1$ gives the required decomposition with $k=k'+2$ as $|S'|-1\geq m/2$. If $s$ is not in $T_1$, then the sequence $S'\cup S_1,S_2,\ldots,S_{k'},T_1$ gives the required decomposition with $k=k'+1$ as $|S'\cup S_1|-1\geq |S'|-1\geq m/2$. This completes the inductive step, and hence the proof of the lemma.
\end{proof}

In what follows, in order to be precise, we use the following definition of a connector vertex of a subtree.
\begin{definition}\label{defn:convx}
    Given a subtree $S$ of a tree $T$, a vertex $v\in V(S)$ is a \emph{connector vertex} of $S$ in $T$ if $d_S(v)<d_T(v)$.
\end{definition}

Note that, for any tree decomposition $T_1,\ldots,T_k$ of a tree $T$, for each $i\in [k]$ any connector vertices of $T_i$ are connectors for the decomposition, but if some of the trees have only 1 vertex then $v$ can be a connector of the decomposition but not a connector vertex of $T_i$ in $T$, for some $i\in [k]$. However, although in this section we use some single vertex trees in our decompositions for convenience, in the applications later no trees in the decomposition will have only one vertex, so that the definitions will coincide.

Using Definition~\ref{defn:convx}, we turn to our main tree decomposition, Lemma~\ref{lemma_endblobs}, which is needed for the proof of the second key lemma,~\cref{lemma_finalembedding}, which is in turn used to carry out step~\ref{stepC} of the embedding in the proof sketch. 
For integers $m_1,\ldots,m_k$, and any tree $T$ with many times more than $\sum_{i\in [k]}m_i$ vertices, it gives a tree decomposition of $T$ into subtrees $T_1,\dots,T_\ell$ with $\ell\geq k$ which each have at most $2$ connector vertices and for which, for each $i\in [k]$, $T_i$ has of the order of $m_i$ vertices. We build up to this by showing first that a version of this is true for $k=1$ (Lemma~\ref{lem_endblobs0}) and 
$k\leq 3$
(Lemma~\ref{lem_endblobs1}), before proving the full lemma, Lemma~\ref{lemma_endblobs}.


\begin{lemma}\label{lem_endblobs0}
	Let $T$ be a tree, and let $S$ be a subtree of $T$ containing at most $2$ connector vertices in $T$. Let $m\geq 1$ with $|S|\geq m$. Then, $S$ has a decomposition into at most $4$ trees, each of which contains at most $2$ connector vertices in $T$, and one of which has at least $m$ and at most $6m$ vertices.
\end{lemma}
\begin{proof} Note that there is nothing to prove if $|S|\leq 6m$, so assume $|S|> 6m$.
Using Lemma~\ref{lemma_divisionavoidingavertex}, take a tree decomposition of $S$ into subtrees $S_1$ and $S_2$ such that $3m\leq |S_1|\leq 6m$.

First, assume that $S_1$ contains at most $2$ connector vertices in $T$. We need to show that we can decompose $S_2$ into at most 3 trees, each of which contain at most $2$ connector vertices in $T$. Note that $S_2$ has at most $3$ connector vertices in $T$, so we may assume that $S_2$ has exactly $3$ such connector vertices $s_1,s_2,s_3$, for otherwise we are done. Consider the (unique) path $P$ between $s_1$ and $s_2$ in $S_2$, and let $v\in V(P)$ be the vertex of $P$ closest to $s_3$. Let $U_1,\dots,U_\ell$ be the vertex sets of the connected components of $S_1-v$, and write $U_i'=U_i\cup\{v\}$ for each $i\in[\ell]$. First assume that $v\not \in\{s_1,s_2,s_3\}$. Note that, if $i,j\in[3]$ are distinct, then $s_i,s_j$ cannot belong to the same component $U_i$, and hence we may assume that $s_i\in U_i$ for all $i\in [3]$. Letting $T_1=S_2[U_1']$, $T_2=S_2[U_2']$ and $T_3=S_2[\bigcup_{i=3}^\ell U_i']$, we see $T_1,T_2,T_3$ is a tree decomposition of $S_2$ in which the connector vertices of each $T_i$ come from the set $\{s_1,s_2,s_3,v\}$. Thus, each $T_i$ has at most $2$ connector vertices in $T$, as desired. Next, assume that $v=s_j$ for some $j\in[3]$, where we may assume $j\not =1$ and $s_1\in U_1$. Let $T_1=S_2[U_1']$ and $T_2=S_2[\bigcup_{i=1}^\ell U_i']$. It is again easy to check that $T_1,T_2$ provide a tree decomposition of $S_2$ into two trees containing at most $2$ connector vertices each, finishing the case when $S_1$ has at most $2$ connector vertices.

Assume then that $S_1$ contains more than $2$ connector vertices in $T$. Thus, $S_1$ has exactly $3$ connector vertices and $S_2$ has at most 1 connector vertex in $T$. As above, we can decompose $S_1$ into at most $3$ trees, each of which contains at most $2$ connector vertices, and  one of which will have at least $m$ and at most $6m$ vertices as $3m\leq |S_1|\leq 6m$.\end{proof}

Using Lemma~\ref{lem_endblobs0}, we now show a similar lemma where the size of 
up to $3$ of the subtrees have been roughly specified.

\begin{lemma}\label{lem_endblobs1}
	Let $T$ be a tree, and let $S$ be a subtree of $T$ containing at most $2$ connector vertices in $T$. Let $k\in 
 [3]$ and let $m_1\geq m_2\geq\ldots\geq m_k$ be positive integers. If $|S|\geq (8k-7)m_1$, then $S$ has a tree decomposition into $T_1,\dots,T_\ell$ (for some positive integer $\ell\geq k$) which each have at most $2$ connector vertices in $T$, and such that, for each  $i\in [k]$, $m_i\leq |T_i|\leq 6m_i$.
\end{lemma}
\begin{proof} We prove this by induction on $k$, where it is true for $k=1$ by Lemma~\ref{lem_endblobs0}. For $k>1$, by Lemma~\ref{lem_endblobs0}, take a tree decomposition of $S$ into subtrees $S_1,S_2,S_3,S_4$, which each have at most $2$ connector vertices in $T$, such that $m_1\leq |S_4|\leq 6m_1$. (Note that, in applications of Lemma~\ref{lem_endblobs0}, we may assume that the number of trees we get is exactly $4$, otherwise we can simply add some single vertex trees.)

First assume $k=2$. If $|S_i|\geq m_2$ for some $i\in[3]$, say $i=1$, then by induction (or Lemma~\ref{lem_endblobs0}) we can further decompose $S_1$ into some subtrees $S_{1,1},\dots,S_{1,a}$ such that $m_2\leq |S_{1,1}|\leq6m_2$, and each $S_{1,i}$ has at most $2$ connector vertices in $T$. Then these trees, together with $S_2,S_3$ and $S_4$, satisfy the conditions. So we may assume that $|S_i|< m_2$ for all $i\in [3]$. But then $|S_1|+|S_2|+|S_3|+|S_4|<3m_2+6m_1\leq 9m_1$, a contradiction to $|S|\geq (8k-7)m_1$.

If $k=3$, then a similar argument shows that we may assume $|S_i|<9m_2$ for each $i\in [3]$ (otherwise we are done by induction). We are similarly done if at least two of $S_1,S_2,S_3$ have size at least $m_2$. Indeed, if say $|S_1|,|S_2|\geq m_2$, then we can decompose $S_1$ into some subtrees with at most $2$ connector vertices each such that one such tree has between $m_2$ and $6m_2$ vertices, and similarly decompose $S_2$ into some subtrees with at most $2$ connector vertices each such that one such tree has between $m_3$ and $6m_3$ vertices, giving the claim. But if each $S_i$ ($i\in [3]$) has size at most $9m_2$, and at most one of them has size at least $m_2$, then $|S_1|+|S_2|+|S_3|<2m_2+9m_2\leq 11m_1$ and thus $|S_1|+|S_2|+|S_3|+|S_4|< 17m_1$, a contradiction  to $|S|\geq (8k-7)m_1$.
%
\end{proof}

Finally, using Lemmas~\ref{lem_endblobs0} and~\ref{lem_endblobs1}, we prove our full tree decomposition, Lemma~\ref{lemma_endblobs}.

\begin{lemma}\label{lemma_endblobs}
	Let $T$ be a tree, and let $S$ be a subtree of $T$ containing at most $2$ connector vertices in $T$. Let $k$ be a non-negative integer, and let $m_1\geq m_2\geq\ldots\geq m_k$ be positive integers. If $|S|\geq 150\sum_{i\in [k]}m_i$, then $S$ has a tree decomposition into $T_1,\dots,T_\ell$ (for some positive integer $\ell\geq k$) which each have at most $2$ connector vertices in $T$, such that, for each  $i\in [k]$, $m_i\leq |T_i|\leq 6m_i$.
\end{lemma}
\begin{proof}
We prove this by induction on $k$. The case $k=0$ is trivial, and $k=1$ is immediate from Lemma~\ref{lem_endblobs0}, so from now on we may assume that $k\geq 2$ and the statement holds for sequences of length less than $k$. By Lemma~\ref{lem_endblobs0}, we can take a tree decomposition of $S$ into subtrees $S_1,S_2,S_3,S_4$, which each have at most $2$ connector vertices in $T$ and such that $25m_1\leq |S_4|\leq 150 m_1$. Let $I_1,I_2,I_3,I_4\subset [k]\setminus \{1\}$ be a partition which maximises $|I_1\cup I_2\cup I_3|$ subject to $150\sum_{i\in I_j}m_i\leq |S_j|$ for each $j\in [3]$.

Now, as $|S_4|\leq 150m_1$,
\[
\sum_{j\in [3]}\left(|S_j|-150\sum_{i\in I_j}m_i\right)\geq |S|-150\sum_{i\in [k]\setminus(I_4\cup \{1\})}m_i-|S_4|\geq 150\sum_{i\in I_4}m_j.
\]
Thus, $|I_4|\leq 2$ (otherwise, writing $t$ for the largest element of $I_4$, we have $\sum_{j\in [3]}\left(|S_j|-150\sum_{i\in I_j}m_i\right)\geq 3\cdot 150m_t$ and hence $|S_j|-150\sum_{i\in I_j}m_i\geq 150m_t$ for some $j\in[3]$, contradicting the definition of our partition). Therefore, as $|S_4|\geq 25m_1$, by Lemma~\ref{lem_endblobs1}, there is some $\ell_0\geq |I_4|+1$ such that we can decompose $S_4$ into $\ell_0$ subtrees with at most $2$ connector vertices each with the property that, for each $i\in I_4\cup\{1\}$, we get a different subtree with at least $m_i$ and at most $6m_i$ vertices in the decomposition. Moreover, by induction, for some $\ell_j\geq |I_j|$, we can decompose each $S_j$ ($j\in [3]$) into $\ell_j$ subtrees with at most $2$ connector vertices each such that, for each $i\in I_j$, we get a different subtree with at least $m_i$ and at most $6m_i$ vertices. Putting together these decompositions for each $S_j$ ($j\in [4]$), we get a decomposition satisfying the conditions.
\end{proof}

\subsection{Results for embedding 
trees}
In our embeddings, we will sometimes use the following common generalised form of Hall's matching criterion to attach leaves to trees.

\begin{lemma}[see, for example,~\cite{bollobas1998modern}]\label{lem:hall}
Let $H$ be a bipartite graph with vertex classes $A$ and $B$, and let $d_a\geq 0$ be an integer for each $a\in A$. Suppose that, for each $U\subset A$, $|N(U)|\geq \sum_{a\in U}d_a$. Then, there are disjoint sets $B_a$, $a\in A$, of $B$ such that, for each $a\in A$, $B_a\subset N(a)$ and $|B_a|=d_a$.
\end{lemma}

To analyse a random tree embedding, we will use Azuma's inequality for submartingales, as follows, where a sequence of random variables  $(X_i)_{i \geq 0}$ is a submartingale if $\EE(X_{i+1} \mid X_0, \dots, X_i) \geq  X_i$ for each $i\geq 0$.

\begin{lemma}[see, for example, \cite{wormald1999differential}]\label{lem:azuma}
Let $(X_i)_{i \geq 0}$ be a submartingale and let $c_i>0$ for each $i\geq 1$. If $|X_i -X_{i-1}| \leq c_i$ for each $i\geq 1$, then, for each $n\geq 1$,
\[
\PP(X_n\leq X_0- t ) \leq  \exp \left( -\frac{t^2}{2\sum_{i=1}^n c_i^2} \right).
\]
\end{lemma}


\section{Proof of Theorem~\ref{theorem_treepacking_starlikestars} using the key lemmas}\label{section_mainembeddingscheme}
In this section we show that Theorem~\ref{theorem_treepacking_starlikestars} can be deduced from two key lemmas. 
These lemmas are a little technical, so we will first set up the situation for the proof of Theorem~\ref{theorem_treepacking_starlikestars} and introduce some notation in Section~\ref{sec_main_initial}. We then give the key lemmas in Section~\ref{sec_main_key}, relating them to the proof sketch. As these lemmas are proved in later sections, they are self-contained (i.e.\ the set-up is restated). We then continue our proof of Theorem~\ref{theorem_treepacking_starlikestars} using these key lemmas, starting with splitting the trees to be embedded into subtrees in Section~\ref{sec_main_split} and carrying out step \ref{stepA} in Section~\ref{sec_main_A}. We then set up a definition of a \emph{good sequence of embeddings} in Section~\ref{sec_good} and also show that if we find a sufficiently long good sequence of embeddings then we will have embedded all the trees required. In Section~\ref{sec_main_B} we find our initial good sequence by carrying out step~\ref{stepB}, before showing in Section~\ref{sec_main_C} that we can inductively extend until we get a good sequence of the required length by carrying out step \ref{stepC}. This completes the proof of Theorem~\ref{theorem_treepacking_starlikestars} subject to the proof of the two key lemmas, Lemma~\ref{lemma_forends} and Lemma~\ref{lemma_finalembedding}.


\subsection{Initial setup}\label{sec_main_initial}
Let $1/n\ll \eps \ll \lambda \ll 1$ and $r=\eps n$.
For convenience of notation, to prove Theorem~\ref{theorem_treepacking_starlikestars}, take an arbitrary sequence of $2r$ trees $T_1,\ldots,T_{2r}$, such that $|T_i|=n-i+1$ for each $i\in [2r]$, and $r$ of these trees are stars while the other $r$ `path-like' trees have at most $\lambda n$ leaves (extending the original sequence appropriately). We may also assume that $r$ is odd.
Let the stars be $S_1,\ldots,S_r$ and let the path-like trees be $P_1,\ldots,P_r$ (both listed in descending order of size).

Label the vertices of $K_n$ as $v_1,\ldots,v_n$, and use this to assign an ordering of the vertices from left to right. 
Thus, for example, given a vertex $v_i$ and a tree (or any subgraph of $K_n$) containing $v_i$, by \emph{rightward neighbours} 
we mean all vertices $v_j$ adjacent to $v_i$ with $j>i$.
For each $i\in [r]$, we will use $v_i$ as the centre for the embedding of $S_i$, and we set $W=\{v_1,\ldots,v_r\}$ to be the set of these vertices. Let $X=\{v_{r+1},\ldots,v_n\}$ be the set of vertices not used as the centre of these stars, and, for each $i\in [r]$, let
\begin{equation}\label{eqn:widefn}
W_i=\{v_j:j\in [r],|P_i|>|S_j|\}.
\end{equation}
For each $i\in [r]$, we will embed $P_i$ into $W_i\cup X$. Note that, as in \eqref{eqn:WiXsketch}, we have
\begin{equation}\label{eqn:wix}
|W_i\cup X|=|P_i|+i-1,
\end{equation}
giving us a version of \ref{prop:embeddingprinciple}.

To help with certain steps in the proof (such as connecting up the ends/artificial ends described in Section~\ref{sec:sketch}), take disjoint sets $A_1,\ldots,A_r$, each of 20 vertices, in $X$. For each $i\in [r]$, the edges between $\{v_1,\ldots,v_r\}$ and $A_i$ will only be used in the embedding of $P_i$, as we might need this set $A_i$ to embed vertices of $P_i$ which need particularly many neighbours embedded in $W_i$. 
Let
\[
A=A_1\cup \ldots \cup A_r,
\label{eqn:A}
\]
and, for convenience, assume this is the set of the $20r$ rightmost vertices of $K_n$.
We will use $H$ and $G$ (with appropriate subscripts) to denote the graphs of the embedded edges and the remaining edges, respectively, at various points in our embedding process.




\subsection{The key lemmas}\label{sec_main_key}
Our first key lemma will be used in step~\ref{stepC} of the embedding (as per the proof sketch) to embed the ends of the path-like trees (only one end of $P_i$ will be embedded this way if $i$ is even, as the other end is embedded in step~\ref{stepA} by a simpler method). In the statement of Lemma~\ref{lemma_forends} below, the tree $T$ will correspond to one of these ends in each application, and $T$ will have a special vertex $t$ where the end connects to the rest of the path-like tree. Within the set up in Section~\ref{sec_main_initial}, the end $T$ is to be embedded within $W^-\cup X$ for some specified subset $W^-$ of $W$. 
Furthermore, $T$ will be embedded with the connector $t$ embedded to some specified vertex $a\in A^-$, where in the application $A^-$ will be the unused vertices in $A_i$, the set put aside in $X$ to help make connections for $P_i$. Finally, while embedding $T$, in the applications we want to avoid overlapping with vertices used to embed other parts of the tree $P_i$, as well as any vertices in $X$ which have already had many neighbouring edges used in the embeddings so far. In Lemma~\ref{lemma_forends}, these `forbidden' vertices are represented by the set $X^{\mathrm{forb}}$.

\begin{restatable}{lemma}{lemmaforends}\label{lemma_forends}
Let $1/n\ll \eps
\ll 1$ and $r=\eps n$. Let $v_1,\ldots,v_n$ be an ordering of the vertices of $K_n$ and let $W=\{v_1,\ldots,v_r\}$ and $X=\{v_{r+1},\ldots,v_{n}\}$. Let $H\subset K_n$ have $\Delta(H)\leq n/50$ and $e(H)\leq \eps n^2$, and let $G$ be the complement of $H$.

Let $W^-\subset W$ and $A^-\subset X$ be such that $|A^-|\geq 2$ and there are no edges inside the set $W^-\cup A^-$ in $H$. 
Let $X^{\mathrm{forb}}\subset X\setminus A^-$ satisfy $|X^{\mathrm{forb}}|\leq n/100$. Let $T$ be a tree with $|W^-|<|T|\leq 4\eps n$ and let $t\in V(T)$.

Then, for every $a\in A^-$, there exists an embedding $\phi$ of $T$ in $G$ with $\phi(t)=a$ such that $W^-\subseteq \operatorname{Im}(\phi)\subset W^-\cup (X\setminus X^{\mathrm{forb}})$ and every vertex in $W^-$ has at most 1 rightward neighbour in $\phi(T)$ in the ordering $v_1,\ldots,v_n$.
\end{restatable}

The second, and more difficult, key lemma (Lemma~\ref{lemma_finalembedding}) is also used in step \ref{stepC} of our embedding. For each even $i\in [r]$, after the ends of $P_i$ or $P_{i+1}$ are embedded, this lemma will embed the rest of $P_i$ or $P_{i+1}$ appropriately.
In Lemma~\ref{lemma_finalembedding}, $T$ represents this remaining portion of $P_i$ or $P_{i+1}$ to be embedded. The slightly more complicated case, which requires the full strength of Lemma~\ref{lemma_finalembedding}, is the embedding of $P_{i+1}$. To relate this further to the proof sketch, this final part of the embedding of $P_{i+1}$ is to use vertices in the disjoint sets $W_{i,j}$, $j\in [i]$, covering each of these sets except for the leftmost vertices -- these appear in the lemma as sets $W^-_j$, $j\in [k]$ (which is then applied with $k=i$).
As in the proof sketch (and in particular in Figure~\ref{fig:turnaround1}), the graph $H$ of edges embedded at this stage has the property that $H[W_{i,j}]$ is 2-colourable (with only edges of $P_i$ embedded within $W_{i,j}$), and the leftmost vertex of $W_j^-$ (i.e., the leftmost vertex of $W_{i,j}$ that we want to cover) has no neighbouring edges within $W_{i,j}$ used in the embedding (as it is not in the image of $P_i$).
This property is represented in the lemma as \itref{cond_finalembedlemma2}. Some vertices in $X$ will have been used for the ends of $P_{i+1}$, so for embedding the remaining portion $T$, we will only use some subset $X^-\subset X$ avoiding these vertices. The embedding of $T$ in the application must agree with the embedding of the ends of $P_{i+1}$ where they overlap, i.e., at the connector vertices.
As discussed before the statement of Lemma~\ref{lemma_forends}, these connector vertices will be embedded into the reserved set $A_{i+1}$. We will also need some other vertices from $A_{i+1}$ to assist in the embedding of $T$ -- together these vertices will be the set $A^-$ in the lemma. Condition~\itref{cond_finalembedlemma1} reflects the property of these assisting vertices: no edges have been used between them and the set $W^-$ we are trying to cover with the embedding of $T$. The final condition~\itref{cond_finalembedlemma3} will come from the fact that each vertex of $W$ will have at most $\lambda n$ rightward edges used in previous embeddings.

Given, then, the embedding of up to 2 connector vertices (the set $U$ in the conclusion of Lemma~\ref{lemma_finalembedding}, coming from the two ends) in $A^-$, $T$ is embedded extending this to cover exactly $W^-\cup X^-$, so that every vertex in $W^-$ (except for the turnaround vertices) has at most 1 rightward neighbour in the embedding~(\itref{prop:S1}), so that furthermore the maximum degree of the total embedded edges has not increased too much (\itref{prop:S2}) and so that there have not been any edges used in the embedding of $T$ between $W$ and the vertices set aside for assisting making connections in other embeddings (\itref{prop:S3}).

  \setcounter{restatedpropcounter}{\value{propcounter}}

\begin{restatable}{lemma}{lemmafinalembedding}\label{lemma_finalembedding} Let $1/n\ll \eps\ll \lambda \ll 1$ and $r=\eps n$.
Let $v_1,\ldots,v_n$ be an ordering of the vertices of $K_n$ and let $W=\{v_1,\ldots,v_r\}$ and $X=\{v_{r+1},\ldots,v_{n}\}$. Let $G\subset K_n$ have $\delta(G)\geq \frac{49n}{50}$, and let $H$ be the complement of $G$ in $K_n$. Suppose that $e(H)\leq \eps n^2$.

Let $k\in [r]$. Let $T$ be a tree with at most $\lambda n$ leaves and with $\frac{99}{100}n\leq |T|\leq n$. Let $W^-\subset W$ and $X^-\subset X$ such that $|W^-\cup X^-|=|T|$. Let $A^-\subset X^-$ satisfy $|A^-|=16$ and let $A\subset X$ satisfy $|A|\leq 20r$. Let $W_{1}^-\cup \ldots\cup W_{k}^-$ be a partition of $W^-$. Suppose the following hold.

\stepcounter{propcounter}
\begin{enumerate}[label = \textbf{\emph{\Alph{propcounter}\arabic{enumi}}}]
\item There are no edges between $W^-$ and $A^-$, nor inside the set $A^-$, in $H$.\label{cond_finalembedlemma1}
\item For each $j\in [k]$, $H[W_{j}^-]$ is $2$-colourable and the leftmost vertex of $W_j^-$ is an isolated vertex in this subgraph (or $W_j^- =\emptyset$).\label{cond_finalembedlemma2}
\item Each vertex in $W^-$ has at most $2\lambda n$ neighbouring edges in $H$.
\label{cond_finalembedlemma3}
\end{enumerate}

  Then, for any $U\subset V(T)$ with $|U|\leq 2$, any embedding of $T[U]$ in $G[A^-]$ extends to an embedding of $T$ in $G[W^-\cup X^-]$, so that, if $S$ is the embedded copy of $T$, then the following conditions are satisfied.
  \stepcounter{propcounter}
  \begin{enumerate}[label = \textbf{\emph{\Alph{propcounter}\arabic{enumi}}}]
  \item Each vertex of $S$ in $W^-$ has at most $1$ rightward neighbour in $S$ unless it is the leftmost vertex in a set in the partition $W^-_{1},\ldots,W^-_{k}$, in which case it has at most $2$ rightward neighbours in $S$.\label{prop:S1}
		\item Every vertex in $V(G)\setminus A^-$ with more than $\frac{n}{200}$ neighbours in $X$ in $H$, as well as every vertex in $A\setminus A^-$, has at most $2$ neighbours in $X$ in $S$.\label{prop:S2}
\item $S$ has no edges between $W^-$ and $A\setminus A^-$, nor inside $A\setminus A^-$.\label{prop:S3}
    \end{enumerate}
\end{restatable}


\subsection{Splitting the path-like trees}\label{sec_main_split}
Using~\cref{lemma_divisionavoidingavertex} (and $|P_1|\geq n-r)$, take a tree decomposition of $P_1$ into $E_{1,1}$ and $E_{1,2}$ so that $\frac{n}{3}\leq |E_{1,1}|\leq \frac{2n}{3}$, and note that $\frac{n}{4}\leq |E_{1,2}|\leq \frac{3n}{4}$. In step~\ref{stepB}, we will embed $P_1$ as the union of these two ends. Let $t_{1,1}$ be the connector vertex between $E_{1,1}$ and $E_{1,2}$.

For each even $i\in [r]$ let $W_{i,0}=W_i\setminus W_{i+1}$ (recall that the sets $W_j$ are defined in~\eqref{eqn:widefn}, and that these sets $W_{i,0}$ will be covered in step~\ref{stepA}). Note that the sets $W_{i,0}$, for even $i\in [r]$, are disjoint and in $W$, so that
\begin{equation}\label{eq:Wsum}
\sum_{\text{even }i\in [r]}|W_{i,0}|\leq |W|=r.
\end{equation}
Using~\cref{corollary_decomp_three}, for each even $i\in [r]$, decompose $P_i$ into trees $E_{i,1}$, $E_{i,2}$, and $F_i$, where
\begin{equation}\label{eqn:Wi0}
|W_{i,0}|+1< |E_{i,1}|\leq 2|W_{i,0}|+2\;\;\text{ and }\;\;  2\varepsilon n< |E_{i,2}|\leq 4\varepsilon n,
\end{equation}
and $E_{i,1}$ and $E_{i,2}$ both contain at most one connector. We think of the sets $E_{i,1},E_{i,2}$ as the two ends, the first of which is used to cover $W_{i,0}$ in step~\ref{stepA}, and the second of which is used in step~\ref{stepC}.

Using~\cref{corollary_decomp_three} again, for each odd $i\in [r]$ with $i\geq 3$ decompose $P_i$ into trees $E_{i,1}$, $E_{i,2}$, and $F_i$, where \begin{equation}
    \label{eqn:Wi0_2}
2\varepsilon n< |E_{i,1}|,|E_{i,2}|\leq 4\varepsilon n
\end{equation}
and $E_{i,1}$ and $E_{i,2}$ both contain at most one connector vertex. Again, we think of $E_{i,1}, E_{i,2}$ as the two ends, but for odd $i\geq 3$ both of these ends will be used in step~\ref{stepC} only.

For each $2\leq i\leq r$, let $t_{i,1}$ be the connector vertex in $E_{i,1}$ and let $t_{i,2}$ be the connector vertex in $E_{i,2}$ (noting that we may have $t_{i,1}=t_{i,2}$).


\subsection{Step~\ref{stepA}: Embedding an end of $P_i$ for each even $i\in [r]$.}\label{sec_main_A}
For each even $i\in [r]$, we will embed $E_{i,1}$ into $K_n[W_{i,0}\cup X]$ with an embedding $\psi_i$, while covering $W_{i,0}$, as follows. For each even $i\in [r]$, pick $a_{i,1}\in A_i$ arbitrarily. Now, noting that
\[
|X\setminus A|=n-21r\geq 4r\overset{\eqref{eq:Wsum}}{\geq} 2\sum_{\text{even }i\in [r]}|W_{i,0}|+2r
\overset{\eqref{eqn:Wi0}}{\geq}\sum_{\text{even }i\in [r]}|E_{i,1}|,
\]
take disjoint sets $U_i$, for even $i\in [r]$, in $X\setminus A$, such that
\[
|U_i|=|E_{i,1}|-|W_{i,0}|-1\overset{\eqref{eqn:Wi0}}\geq 0,
\]
for each even $i\in [r]$. For each even $i\in [r]$, using that $E_{i,1}$ is a tree, 
order the vertices of $E_{i,1}$ so that each vertex has at most $1$ rightward neighbour in $E_{i,1}$ and so that $t_{i,1}$ is the rightmost vertex in this ordering.
Using this ordering of $V(E_{i,1})$ and the ordering of $V(K_n)$, embed $E_{i,1}$ into $K_n[W_{i,0}\cup U_i\cup \{a_{i,1}\}]$ in an order-preserving way, calling the resulting embedding $\psi_i:E_{i,1}\to K_n$. Recall that the vertices of $A$ were the $20r$ rightmost vertices of $V(K_n)$, and therefore we have set $\psi_i(t_{i,1})=a_{i,1}$. We record the following properties for convenience.



\medskip

\noindent\textbf{Embeddings $\psi_i$:} For each even $i\in [r]$, we have an embedding $\psi_i:E_{i,1}\to K_n$, such that these embeddings are vertex-disjoint and, letting $H_0$ be the union of $\psi_i(E_{i,1})$ across all even $i\in [r]$, the following hold.
\stepcounter{propcounter}
\begin{enumerate}[label = \textbf{{\Alph{propcounter}\arabic{enumi}}}]
\item Every vertex in $W=\{v_1,\ldots,v_r\}$ has at most one rightward neighbour in $H_0$.\label{prop_psi_1}
\item For each even $i\in[r]$, $W_{i,0}\subset \operatorname{Im}(\psi_i)\subset W_{i,0}\cup X$.\label{prop_psi_2}
\item For each even $i\in [r]$, $\operatorname{Im}(\psi_i)$ has exactly one vertex in $A$, which is $a_{i,1}=\psi_i(t_{i,1})$.\label{prop_psi_3}
\end{enumerate}


\subsection{Good sequences of embeddings}\label{sec_good}
As in the proof sketch, we will now embed $P_1$ (step \ref{stepB}), before embedding the rest of $P_i$ and the whole of $P_{i+1}$ together for each even $i\in [r]$ (step \ref{stepC}). We prove that the embedding can be done by induction, for which the following definition recording the successful embedding of $P_1,\ldots,P_\ell$ for some odd $\ell$ will be useful. Note that the properties below essentially correspond to rules~\ref{prop:rule1} and~\ref{prop:rule2}, with some additional mild assumptions.

\begin{restatable}{definition}{goodsequence}
	Let $\ell\in [r]$ be odd and let $\phi_1,\dots,\phi_\ell$ be embeddings of $P_1,\dots,P_\ell$, respectively, in $K_n$.  We say that $\phi_1,\dots,\phi_\ell$ form a \emph{good sequence of length $\ell$} if the following properties are satisfied, where $H$ is the subgraph of $K_n$ whose edges are the edges of the trees embedded by $\phi_1,\dots,\phi_\ell,\psi_{\ell+1},\psi_{\ell+3},\dots,\psi_{r-1}$.
  \stepcounter{propcounter}
  \begin{enumerate}[label = \textbf{{\Alph{propcounter}\arabic{enumi}}}]
		\item The trees embedded by $\phi_1,\dots,\phi_\ell,\psi_{\ell+1},\psi_{\ell+3},\dots,\psi_{r-1}$ are edge-disjoint.\label{property_disjoint}
    \item For each $j\in [\ell]$,  $\phi_j$ embeds $P_j$ into $W_j\cup X$.\label{property_image}
		\item For all even $j\leq \ell$, $\phi_j$ extends $\psi_j$.\label{property_extend}
    \item Each vertex in $W=\{v_1,\ldots,v_r\}$ has at most $1$ rightward neighbour in $\phi_1(P_1)$. \label{property_forwardedge_new1}
  		\item Each vertex in $W=\{v_1,\ldots,v_r\}$ has at most $2$ rightward neighbours in $\phi_j(P_j)\cup \phi_{j+1}(P_{j+1})$ for each even $j\in [\ell]$.\label{property_forwardedge_new2}
		\item Every vertex has at most $\frac{n}{100}+2\ell$ neighbours in $H$ in $X$, and each vertex in $\bigcup_{j\in[r]\setminus [\ell]}A_j$ has at most $\frac{n}{200}+2\ell$ neighbours in $H$ in $X$.\label{property_mindegree}
		\item For each $j\in [r]\setminus [\ell]$, there is no edge between $W$ and $A_j$, or within $A_j$, in $\bigcup_{i\in [\ell]}\phi(P_{i})$.\label{property_A}
	\end{enumerate}
\end{restatable}

If we are able to construct a good sequence of length $r$, then we will be able to easily embed the stars $S_1,\ldots,S_r$ disjointly from the paths $P_1,\ldots,P_r$. Indeed, suppose we have a good sequence $\phi_1,\ldots,\phi_r$ of length $r$ and let $G$ be the edges of $K_n$ not in $H:=\bigcup_{j\in [r]}\phi_j(P_j)$. Letting $i\in [r]$, we now show that $v_i$ has at least $|S_i|-1$ rightward neighbours in $G$.

 If $v_i\notin W_j$ for each $j\in [r]$, then $v_i$ is in no embedding $\phi_j(P_j)$, $j\in [r]$, by \ref{property_image}, and therefore $v_i$ has $n-i=|S_i|-1$ rightward neighbours in $G$. Otherwise, let $j\in [r]$ be the largest such $j$ with $v_i\in W_j$. If $j$ is odd, then, by \ref{property_forwardedge_new1} and \ref{property_forwardedge_new2}, $v_i$ has at least $n-i-j$ rightward neighbours in $G$, while if $j$ is even, then by \ref{property_forwardedge_new1}, \ref{property_forwardedge_new2}, \ref{prop_psi_1} and \ref{property_extend}, $v_i$ also has at least $n-i-j$ rightward neighbours in $G$. Then, as
\[
n-i-j\overset{\eqref{eqn:widefn}}=n-1-|\{i'\in [r]:|S_i|<|S_{i'}|\}|-|\{j'\in [r]:|S_i|<|P_{j'}|\}|=|S_i|-1,
\]
$v_i$ has at least $|S_i|-1$ rightward neighbours in $G$. Thus, for every $i\in [r]$, $v_i$ has at least $|S_i|-1$ rightward neighbours in $G$, so that the stars $S_1,\ldots,S_r$ can be embedded disjointly into $G$, which from the definition of $G$, \ref{property_disjoint} and \ref{property_image}, shows that $P_1,\ldots,P_r,S_1,\ldots,S_r$ can be embedded disjointly into $K_n$, where each star is embedded with the centre to the left of all its leaves. To complete the proof of Theorem~\ref{theorem_treepacking_starlikestars} (subject to the proof of the two key lemmas), we note that every vertex which is not the centre of an embedded star (i.e., every vertex in $X$) has, by \ref{property_mindegree}, degree at most $r+\frac{n}{100}+2r\leq \frac{n}{10}$ in the embedding, as required.

It is left then to show that a good sequence of length 1 exists, which we do in Section~\ref{sec_main_B}, and to inductively extend this good sequence by completing the next two embeddings, which we do in Section~\ref{sec_main_C}.


\subsection{Step~\ref{stepB}: Embedding $P_1$}\label{sec_main_B}
The following lemma shows that a good sequence of length 1 does exist.
\begin{lemma}\label{lemma_T1}
	There is an embedding $\phi_1$ of $P_1$ such that $\phi_1$ forms a good sequence of length $1$.
\end{lemma}
\begin{proof} Let $H_0$ be the embedded edges in $K_n$ so far, i.e., the union of the images of $\psi_i$ over all even $i\in [r]$.
Pick $a_{1,1}\in A_1$, and note that, by \ref{prop_psi_3}, $a_{1,1}$ has no neighbours in $H_0$. Note that $H_0$ is a forest with, by \eqref{eqn:Wi0} and \eqref{eq:Wsum}, at most $2|W|+2r=4r$ non-isolated vertices, and that $|E_{1,1}|+|E_{1,2}|=|P_1|+1\overset{\eqref{eqn:wix}}=|W_1\cup X|+1$ with $n/4\leq |E_{1,1}|,|E_{1,2}|\leq 3n/4$. Hence, we can partition $(W_1\cup X)\setminus \{a_{1,1}\}$ as $U_{1,1}\cup U_{1,2}$ so that $H_0[U_{1,1}]$ and $H_0[U_{1,2}]$ are empty, and $|U_{1,i}|=|E_{1,i}|-1$ for each $i\in [2]$. Indeed, as $H_0$ is bipartite, we can take a bipartition of the (at most $4r$) non-isolated vertices, and distribute the remaining vertices arbitrarily among the two vertex classes in such a way that the parts we obtain have sizes $|E_{1,1}|-1$ and $|E_{1,2}|-1$. Thus, we have that $H_0[U_{1,1}\cup\{a_{1,1}\}]$ and $H_0[U_{1,2}\cup\{a_{1,1}\}]$ are both empty graphs.

Arrange the vertices of both $E_{1,1}$ and $E_{1,2}$ from left to right so that each vertex has at most 1 rightward neighbour in $E_{1,1}$ and $E_{1,2}$, respectively, and $t_{1,1}$ is the rightmost vertex in each ordering. For each $j\in [2]$, let $Y_{1,j}$ be the $|W_1\cap U_{1,j}|$ leftmost vertices of $E_{1,j}$ in this ordering. For each $j\in[2]$,
using that $E_{1,j}$ is bipartite and $|E_{1,j}|\geq n/4$, take a set $Z_{1,j}\subset V(E_{1,j})\setminus (Y_{1,j}\cup\{t_{1,1}\})$ of $|A\cap U_{1,j}|-1$ vertices which form an independent set in $E_{1,j}$ and does not contain any neighbour of $t_{1,1}$ (using that $\Delta(E_{1,j})\leq \lambda n$) or any neighbour of $Y_{1,j}$ (using that each vertex has in $Y_{1,j}$ has at most one rightward neighbour in $E_{1,j}$).

For each $j\in [2]$, using the orderings of the vertices in $Y_{1,j}$ in $E_{1,j}$ and in $W_1\cap U_{1,j}$ in $K_n$, embed $E_{1,j}$ to $K_n[U_{1,j}\cup \{a_{1,1}\}]$ so that $Y_{1,j}$ is embedded to $W_1\cap U_{1,j}$ in an order-preserving way, $Z_{1,j}$ is embedded to
 $(A\cap U_{1,j})\setminus \{a_{1,1}\}$, and $t_{1,1}$ is embedded to $a_{1,1}$. Letting $\phi_1$ be the embedding of $P_1$, we check that $\phi_1$ gives a good sequence of length $1$. 
	Properties~\ref{property_disjoint},~\ref{property_image} and~\ref{property_forwardedge_new1} hold simply by construction, and Properties~\ref{property_extend} and \ref{property_forwardedge_new2} are vacuous for $\ell=1$.
As $H_0$ is a forest, each of whose components are trees with at most $\lambda n$ leaves, $\Delta(H_0)\leq \lambda n$, and therefore, as $\Delta(T_1)\leq \lambda n$ and $2\lambda n\leq n/200$, we have that \ref{property_mindegree} holds. For each $j\in [2]$, as $Y_{1,j}$ is embedded to $W_1\cap U_{1,j}$ and $Z_{i,j}$ is embedded to $(A\cap U_{1,j})\setminus \{a_{1,1}\}$, there are no edges of $\phi_1(E_{1,j})$ between $W_1\cap U_{1,j}$ and $(A\cap U_{1,j})\setminus \{a_{1,1}\}$, or inside $(A\cap U_{1,j})\setminus\{a_{1,1}\}$. Therefore, there are no edges of $\phi_1(P_1)$ between $W_1$ and $A\setminus \{a_{1,1}\}$, or inside $A\setminus\{a_{1,1}\}$, and thus \ref{property_A} holds.
\end{proof}


\subsection{Step~\ref{stepC}: Embedding $P_i$ and $P_{i+1}$}\label{sec_main_C}
In order to complete the proof of Theorem~\ref{theorem_treepacking_starlikestars}, subject only to the proof of the key lemmas, it is left to show that we can extend a good sequence. That is, we show the following lemma.

\begin{lemma}\label{lemma_treepairs}
	Let $i$ be even with $2\leq i\leq r-1$. Then any good sequence $\phi_1,\dots,\phi_{i-1}$ of $i-1$ embeddings extends to a good sequence $\phi_1,\dots,\phi_{i+1}$ of $i+1$ embeddings.
\end{lemma}
\begin{proof}
	Let $H$ be the subgraph of $K_n$ whose edges are used in the images of $\phi_1,\dots,\phi_{i-1}$, $\psi_i,\psi_{i+2},\dots,\psi_{r-1}$, and let $G$ be the complement of $H$.
Recall that $\psi_i$ is an embedding of $E_{i,1}$, and that the properties \ref{prop_psi_1}--\ref{prop_psi_3} hold, moreover, $W_i$ is the disjoint union of $W_{i,0}$ and $W_{i+1}$. By \ref{prop_psi_1}, \ref{property_forwardedge_new1} and \ref{property_forwardedge_new2}, each vertex in $W_{i+1}$ has at most $1+1+2(i-2)/2=i$ rightward neighbours in $H$, and hence $H[W_{i+1}]$ is $i$-degenerate and thus can be partitioned as $W_{i,1}\cup \ldots\cup W_{i,i+1}$ so that these sets are all independent sets in $H$. For each $j\in [i+1]$, let $w_{i,j,1}$ and $w_{i,j,2}$ be the leftmost two vertices in that order in $W_{i,j}$ (with the convention that if no such vertex or vertices exist then, for example, we interpret $\{w_{i,j,1}\}$ as an empty set).



\smallskip \noindent\textbf{Embedding the ends of $P_{i+1}$.}
We start by embedding the two ends of $P_{i+1}$ to cover $W_{i,i-1}\setminus \{w_{i,i-1,1}\}$ and  $W_{i,i}\setminus \{w_{i,i,1}\}$ (cf.\ Figure~\ref{fig:turnaround1}). We will do so while avoiding high-degree vertices in $H$, so for this let $X^{\mathrm{high}}\subset X$ be the set of vertices in $X$ with at least $\frac{n}{100}-2\lambda n$ neighbours in $X$ in $H$. As $e(H)\leq r\cdot n=\varepsilon n^2$,
\begin{equation}\label{eqn:high}
|X^{\mathrm{high}}|\leq \frac{2\eps n^2}{n/200}=400\eps n,
\end{equation}
and, by \ref{property_mindegree}, we have $X^{\mathrm{high}}\cap (A_i\cup A_{i+1})=\emptyset$.
Let $A_{i+1,1}\subseteq A_{i+1}$ be an arbitrary subset of size $2$ and pick $a_{i+1,1}\in A_{i+1,1}$, let $X_1^{\mathrm{forb}}=X^{\mathrm{high}}\cup (A\setminus A_{i+1,1})$, and
let $W_{i+1,i-1}^-:=W_{i,i-1}\setminus \{w_{i,i-1,1}\}$. Note that $|X_1^{\mathrm{forb}}|\leq n/100$,
and there are no edges in $H$ inside $W_{i+1,i-1}^-\cup A_{i+1,1}$  by \ref{property_A} and the definition of our sets $W_{i,j}$. As
\begin{equation}\label{eqn:Wsize}
|W_{i+1,i-1}^-|\leq |W|= r= \eps n\overset{\eqref{eqn:Wi0_2}}{<} |E_{i+1,1}|\overset{\eqref{eqn:Wi0_2}}{\leq} 4\eps n,
\end{equation}
then, we can apply Lemma~\ref{lemma_forends} to $W$, $X$, $W_{i+1,i-1}^-$, $A_{i+1,1}$, $X^{\mathrm{forb}}_1$, $E_{i+1,1}$ and $t_{i+1,1}\in V(E_{i+1,1})$ (recall that $t_{i+1,j}$ is the unique connector vertex of $E_{i+1,j}$). That is, we can take an embedding $\psi_{i+1}$ of $E_{i+1,1}$ in $G[W_{i+1,i-1}^-\cup A_{i+1,1}\cup (X\setminus (X^{\mathrm{high}}\cup A))]$ such that $\psi_{i+1}(t_{i+1,1})=a_{i+1,1}$, $W_{i+1,i-1}^-=W_{i,i-1}\setminus \{w_{i,i-1,1}\}\subseteq \psi_{i+1}(V(E_{i+1,1}))$, and each vertex in $W_{i+1,i-1}^-$ has at most 1 rightward neighbour in $\psi_{i+1}(E_{i+1,1})$.

Next, if $t_{i+1,2}=t_{i+1,1}$, let $A_{i+1,2}\subseteq A_{i+1}$ be a subset of size $2$ with $A_{i+1,1}\cap A_{i+1,2}=\{a_{i+1,1}\}$ and let $a_{i+1,2}=a_{i+1,1}$. If $t_{i+1,2}\neq t_{i+1,1}$, let $A_{i+1,2}\subseteq A_{i+1}\setminus A_{i+1,1}$ be a subset of size $2$
 and pick some $a_{i+1,2}\in A_{i+1,2}$. 
 If $t_{i+1,2}\neq t_{i+1,1}$, extend $\psi_{i+1}$ to embed $E_{i+1,1}\cup E_{i+1,2}[\{t_{i+1,2}\}]$ by setting $\psi_{i+1}(t_{i+1,2})=a_{i+1,2}$. (If $t_{i+1,1}t_{i+1,2}$ is an edge of $P_{i+1}$, then we will embed this later, in an application of Lemma~\ref{lemma_finalembedding}, using that, then, by \ref{property_A}, $\psi_{i+1}(t_{i+1,1})\psi_{i+1}(t_{i+1,2})=a_{i+1,1}a_{i+1,2}$ is an edge of $G$.)
 Let $X_2^{\mathrm{forb}}=(X^{\mathrm{high}}\cup A\cup \psi_{i+1}(V(E_{i+1,1})))\setminus A_{i+1,2}$,
 so that
 \[
|X_2^{\mathrm{forb}}|\leq |X^{\mathrm{high}}|+|A|+|E_{i+1,1}|\overset{\eqref{eqn:high},\eqref{eqn:Wi0_2}}{\leq} 400\eps n+20\eps n+4\eps n\leq \frac{n}{100},
 \]
and let $W_{i+1,i}^-:=W_{i,i}\setminus \{w_{i,i,1}\}$  (again, cf.\ Figure~\ref{fig:turnaround1}).

Then (using the same calculation as in~\eqref{eqn:Wsize}, as well as~\ref{property_A}), by~\cref{lemma_forends} again, applied to $W$, $X$, $W_{i+1,i}^-$, $A_{i+1,2}$, $X^{\mathrm{forb}}_2$, $E_{i+1,2}$ and $t_{i+1,2}\in V(E_{i+1,2})$, we can extend $\psi_{i+1}$  to an embedding $\psi_{i+1}'$ of $E_{i+1,1}\cup E_{i+1,2}$ in $G$ 
such that, recapping in part our properties from embedding $E_{i+1,1}$, we have the following.
\stepcounter{propcounter}
\begin{enumerate}[label = \textbf{\Alph{propcounter}\arabic{enumi}}]
\item $\operatorname{Im}(\psi_{i+1}')
\subseteq (W_{i,i-1}\setminus \{w_{i,i-1,1}\})\cup (W_{i,i}\setminus \{w_{i,i,1}\})\cup A_{i+1,1}\cup A_{i+1,2}\cup (X\setminus (X^{\mathrm{high}}\cup A))$.\label{prop_eil_0}
\item The connectors $t_{i+1,1}$ and $t_{i+1,2}$ map to $a_{i+1,1},a_{i+1,2}\in A_{i+1}$, respectively.\label{prop_ei1_1}
\item We have $(W_{i,i-1}\cup W_{i,i})\setminus \{w_{i,i-1,1},w_{i,i,1}\}\subseteq \operatorname{Im}(\psi_{i+1}')$.
\label{prop_ei1_2}
\item Each vertex in $(W_{i,i-1}\cup W_{i,i})\setminus \{w_{i,i-1,1},w_{i,i,1}\}$ has at most $1$ rightward neighbour in $\psi_{i+1}'(E_{i+1,1}\cup E_{i+1,2})$.\label{prop_ei1_3}
\end{enumerate}

\smallskip \noindent\textbf{Embedding the unembedded end of $P_{i}$.}
Next, we embed the unembedded end  $E_{i,2}$ of $P_i$ to cover $W_{i,i+1}^-:=W_{i,i+1,1}\setminus \{w_{i,i+1}\}$. If $t_{i,2}\neq t_{i,1}$, then select $a_{i,2}\in A_{i}\setminus \{a_{i,1}\}$ and let $A_{i,1}\subseteq A_{i}\setminus \{a_{i,1}\}$
be a subset of size $2$ containing $a_{i,2}$. Otherwise (if $t_{i,2}= t_{i,1}$), then let $a_{i,2}=a_{i,1}$ and let $A_{i,1}\subseteq A_{i}$ be a subset of size two with $a_{i,1}\in A_{i,1}$. If $t_{i,2}\neq t_{i,1}$, extend $\psi_{i}$ to embed $E_{i,1}\cup E_{i,2}[\{t_{i,2}\}]$ by setting $\psi_{i}(t_{i,2})=a_{i,2}$. (Again, if $t_{i,1}t_{i,2}$ is an edge of $P_i$, we do not yet embed it, but will do so later during an application of Lemma~\ref{lemma_finalembedding}, using that, then, by \ref{property_A}, $\psi_i(t_{i,1})\psi_i(t_{i,2})=a_{i,1}a_{i,2}$ is an edge of $G$.)

Let $X_3^{\mathrm{forb}}=(X^{\mathrm{high}}\cup A\cup \psi_{i}(V(E_{i,1})))\setminus A_{i,1}$,
so that, similarly to $X_2^{\mathrm{forb}}$, we have $|X_3^{\mathrm{forb}}|\leq \frac{n}{100}$. Let $W_{i,i+1}^-:=W_{i,i+1}\setminus \{w_{i,i+1,1}\}$.
Then, by~\cref{lemma_forends} again, applied to $W$, $X$, $W_{i,i+1}^-$, $A_{i,1}$, $X^{\mathrm{forb}}_3$, $E_{i,2}$ and $t_{i,2}\in V(E_{i,2})$, we can extend $\psi_{i}$ to an embedding $\psi_i'$ of $E_{i,1}\cup E_{i,2}$ in $G$ 
such that, recapping in part our properties from embedding $E_{i,1}$ (\ref{prop_psi_1}--\ref{prop_psi_3}), we have the following.
\stepcounter{propcounter}
\begin{enumerate}[label = \textbf{\Alph{propcounter}\arabic{enumi}}]
\item $\operatorname{Im}(\psi_i')\subseteq W_{i,0}\cup (W_{i,i+1}\setminus \{w_{i,i+1,1}\})\cup \{a_{i,1}\}\cup A_{i,1}\cup (X\setminus (X^{\mathrm{high}}\cup A))\cup(X^{\mathrm{high}}\cap \psi_i(V(E_{i,1})))$.\label{prop:groundset}
\item The connectors $t_{i,1}$ and $t_{i,2}$ map to $a_{i,1},a_{i,2}\in A_{i,1}\cup\{a_{i,1}\}$, respectively.\label{prop:connectortoa}
\item We have $W_{i,0}\cup (W_{i,i+1}\setminus \{w_{i,i+1,1}\})\subseteq \operatorname{Im}(\psi_i')$.
\label{H3}
\item Each vertex in $W_{i,0}\cup (W_{i,i+1}\setminus \{w_{i,i+1,1}\})$ has at most $1$ rightward neighbour in $\psi_{i}'(E_{i+1,1}\cup E_{i+1,2})$.\label{prop:forward}
\end{enumerate}


\smallskip \noindent\textbf{Embedding the rest of $P_{i}$.}
We will now embed the rest of $P_i$, which is the tree $F_i$, using \cref{lemma_finalembedding}. Let $G'$ be $G$ with any edges used in the embeddings $\psi_{i+1}'$ and $\psi_i'$ removed, and let $H'$ be the complement of $G'$ in $K_n$. 
For each $j\in [i-2]$, let $W_{i,j}^-=W_{i,j}\setminus \{w_{i,j,2}\}$ (cf.\ Figure~\ref{fig:turnaround1}). Let $W_{i,i-1}^-=W_{i,i-1}$ and $W_{i,i}^-=W_{i,i}$.
Let $W_i^-=\bigcup_{j\in [i]}W_{i,j}^-$. Let $X_{i,0}^-=(X\setminus \mathrm{Im}(\psi_i'))\cup \{a_{i,1},a_{i,2}\}$ and let $A^-_i\subset A_i\setminus (A_{i,1}\setminus\{a_{i,2}\})$ have size 16 with $a_{i,1},a_{i,2}\in A^-_i$. Recalling that, by our convention, if $|W_{i,j}|<2$ then $\{w_{i,j,2}\}=\emptyset$, and similarly, if $|W_{i,j}|<1$ then $\{w_{i,j,1}\}=\emptyset$, let $b_i=(i-1)-|\{w_{i,1,2},\dots,w_{i,i-2,2}\}\cup\{w_{i,i+1,1}\}|\geq 0$. Let $X_i^-$ be obtained from $X_{i,0}^-$ by removing $b_i$ arbitrary elements not belonging to $A$.

We will now check that the conditions of Lemma~\ref{lemma_finalembedding} hold for the graph $G'$ (with complement $H'$), tree $F_i$, the subset $W_i^-\subseteq W$ partitioned as  $W_{i,1}^-\cup\ldots\cup W_{i,i}^-$ (so that $k=i$), $X_i^-$ taking the role of $X^-$, and $A^-$ being given by $A_i^-$. (The sets $W$, $X$, $A$ for the application of the lemma are given by the sets with the same notation.)  

First, by \ref{property_mindegree}, we have
\begin{equation}\label{eqn:deltaGprime}
\delta(G')\geq n-|W|-\left(\max_{v\in V(G)}d_H(v,X)\right)-\Delta(P_i)-\Delta(P_{i+1})\geq n-r-\frac{n}{100}-2r-2\lambda n\geq \frac{49n}{50},
\end{equation}
as required for an application of Lemma~\ref{lemma_finalembedding}. Also, $e(H)\leq \eps n^2$ as $H$ is contained in the union of $r=\eps n$ trees.
Furthermore, by \eqref{eq:Wsum} and \eqref{eqn:Wi0}, we have $|F_i|\geq n-|E_{i,1}|-|E_{i,2}|\geq n-8\eps n\geq \frac{99}{100}n$ and, like $P_i$, $F_i$ has at most $\lambda n$ leaves.

Next,
\begin{align*}
|W_i^-\cup X_{i,0}^-|&\;\;\;=|W_i\cup X|-|\{w_{i,1,2},\ldots,w_{i,i-2,2}\}|-|W_{i,0}|-|W_{i,i+1}|-|X\cap \mathrm{Im}(\psi_i')|+|\{a_{i,1},a_{i,2}\}|\\
&\overset{\ref{prop:groundset},\ref{H3}}{=}|W_i\cup X|-|\{w_{i,1,2},\ldots,w_{i,i-2,2}\}|-| \mathrm{Im}(\psi_i')|-|\{w_{i,i+1,1}\}|+|\{a_{i,1},a_{i,2}\}|\\
&\;\;\;=|W_i\cup X|-(i-1-b_i)-| \mathrm{Im}(\psi_i')|+|\{a_{i,1},a_{i,2}\}|\\
&\;\;\,\overset{\eqref{eqn:wix}}=|P_i|+(i-1)-(i-1)-|\mathrm{Im}(\psi_i')|+|\{a_{i,1},a_{i,2}\}|+b_i\\
&\;\;\;=|P_i|-|V(E_{i,1})\cup V(E_{i,2})|+|\{t_{i,1},t_{i,2}\}|+b_i\\
&\;\;\;=|F_i|+b_i,
\end{align*}
so we have
\(|W_i^-\cup X_{i}^-|=|F_i|.\) Also, we clearly have $|A_i^-|=16$ and $|A|\leq 20r$ by definition. So we have checked all the conditions listed in the first two paragraphs of Lemma~\ref{lemma_finalembedding}, and it remains to check \itref{cond_finalembedlemma1}, \itref{cond_finalembedlemma2}, \itref{cond_finalembedlemma3} .

Next, note that $H'$ is the graph formed from $H$ by adding any edges from the image of $\psi_{i+1}'$ and $\psi_i'$ which are not already in $H$. 
Then, for each $j\in [i-2]$, $W_{i,j}^-$ is an independent set in $H$, so, by \ref{prop_eil_0} and \ref{prop:groundset}, $W_{i,j}^-$ is an independent set in $H'$. Furthermore, by \ref{prop:groundset}, for each $j\in \{i-1,i\}$, the only edges in $H'[W_{i,j}^-]$ are from the image of $\psi_{i+1}'$, so as this image is 2-colourable, $H'[W_{i,j}^-]$ is $2$-colourable, and, in combination with \ref{prop_eil_0}, $w_{i,j,1}$ is an isolated vertex in $H'[W_{i,j}^-]$.
Thus, we will have that \itref{cond_finalembedlemma2} holds for an application of Lemma~\ref{lemma_finalembedding}. Furthermore, there are no edges between $W^-_i$ and $A^-_i$ or inside  $A_i^-$ in $H'$ by \ref{property_A} (and \ref{prop_eil_0} and \ref{prop:groundset}), so that \itref{cond_finalembedlemma1} holds for an application of Lemma~\ref{lemma_finalembedding}. Finally, each vertex in $W_i^-$ has at most $|W|=r=\eps n$ leftward neighbours in $H'$ and at most $r=\eps n$ rightward neighbours in $H'$ by \ref{property_forwardedge_new1}, \ref{property_forwardedge_new2}, \ref{prop_eil_0}, \ref{prop_ei1_3}, \ref{prop:groundset} and \ref{prop:forward}. Thus, each vertex in $W_i^-$ has at most $2\eps n\leq 2\lambda n$ neighbouring edges in $H'$, so that \itref{cond_finalembedlemma3} holds for an application of Lemma~\ref{lemma_finalembedding}.

Therefore, we have checked all the conditions to apply Lemma~\ref{lemma_finalembedding}. Taking $U=\{t_{i,1},t_{i,2}\}$ embedded to $a_{i,1},a_{i,2}$ respectively (recall \ref{prop:connectortoa} and also \itref{cond_finalembedlemma1} checked above to note this is a valid embedding), we get that
$\psi_i'$ extends to an embedding $\phi_i$ of $P_i$ in $G'\cup \psi_i'(E_{i,1}\cup E_{i,2})$ such that the following properties hold. (Note that $\phi_i$ is obtained by combining the embedding $\psi_i'$ with the embedding guaranteed by Lemma~\ref{lemma_finalembedding}.)
\stepcounter{propcounter}
\begin{enumerate}[label = \textbf{\Alph{propcounter}\arabic{enumi}}]
\item $\operatorname{Im}(\phi_i)\subseteq (W_i\cup X)\setminus(\{w_{i,j,2}:j\in[i-2]\}\cup\{w_{i,i+1,1}\}).$\label{prop:newS0}
\item Each vertex in $\operatorname{Im}(\phi_i)\cap W_i$ has at most 1 rightward neighbour in $\phi_i(P_i)$, except for $w_{i,j,1}$, $j\in [i]$, which have at most 2 rightward neighbours.\label{prop:newS1}
\item Every vertex in $V(G)\setminus (A_i^-\cup\operatorname{Im}(\psi_i'))$ with more than $\frac{n}{200}$ neighbours in $X$ in $H'$, as well as every vertex in $A\setminus A_i$, has at most 2 neighbours in $X$ in $\phi_i(P_i)$.\label{prop:newS2}
\item $\phi_i(P_i)$ has no edges between $W_i$ and $A\setminus A_i$, or inside $A\setminus A_i$.\label{prop:newS3}
	\end{enumerate}
Indeed, \ref{prop:newS0} holds using \ref{prop:groundset} (and the definition of $W_i^-$, $X_i^-$); \ref{prop:newS1} holds by \ref{prop:forward} and \itref{prop:S1}; \ref{prop:newS2} holds by \itref{prop:S2} (and \ref{prop:groundset}, as $(A\setminus A_i)\cap \operatorname{Im}(\psi_i')=\emptyset$ and $A_i^-\subseteq A_i$); and \ref{prop:newS3} holds by \itref{prop:S3}  and \ref{prop:groundset}.

\smallskip \noindent\textbf{Embedding the rest of $P_{i+1}$.}
  Finally, we will embed the rest of $P_{i+1}$, which is the tree $F_{i+1}$. Let $G''$ be obtained from $G'$ by removing any edges in $\phi_i(P_i)$, and let $H''$ be the complement of $G''$ in $K_n$. For each $j\in [i-2]$, let $W_{i+1,j}^-=W_{i,j}\setminus \{w_{i,j,1}\}$ (cf. Figure~\ref{fig:turnaround1}). Let $W_{i+1,i+1}^-=W_{i,i+1}$.
Let $W_{i+1}^-=\left(\bigcup_{j\in [i-2]}W_{i+1,j}^-\right)\cup W_{i+1,i+1}^-$. Let $X_{i+1,0}^-=(X\setminus \operatorname{Im}(\psi_{i+1}'))\cup \{a_{i+1,1},a_{i+1,2}\}$ and let $A^-_{i+1}\subset (A_{i+1}\setminus (A_{i+1,1}\cup A_{i+1,2}))\cup\{a_{i+1,1},a_{i+1,2}\}$ have size 16 with $a_{i+1,1},a_{i+1,2}\in A^-_{i+1}$. Similarly to before, let $b_{i+1}'=i-|\{w_{i,1,1},\dots,w_{i,i,1}\}|\geq 0$, and let $X_{i+1}^-$ be obtained from $X_{i+1,0}^-$ by removing $b_{i+1}'$ arbitrary elements not belonging to $A$.

We will now check that we can apply Lemma~\ref{lemma_finalembedding} for the graph $G''$ (with complement $H''$), tree $F_{i+1}$, the subset $W_{i+1}^-\subseteq W$ partitioned as  $W_{i+1,1}^-\cup\ldots\cup W_{i+1,i-2}^-\cup W_{i+1,i+1}^-$ (so that $k=i-1$), $X_{i+1}^-$ taking the role of $X^-$, and $A^-$ being given by $A_{i+1}^-$. (Again, the sets $W$, $X$, $A$ for the application of the lemma are given by the sets with the same notation.)  

We have, using \eqref{eqn:deltaGprime}, that $\delta(G'')\geq \delta(G')-\lambda n\geq  \frac{49}{50}n$, and, similarly to $F_{i}$, that $|F_{i+1}|\geq \frac{99}{100}n$ and $F_{i+1}$ has at most $\lambda n$ leaves. 
We have $e(H'')\leq \eps n^2$ as before, and
\begin{align*}
|W_{i+1}^-&\cup X_{i+1,0}^-|\\
&\;\;\,=|W_{i+1}\cup X|-|\{w_{i,1,1},\ldots,w_{i,i-2,1}\}|-|W_{i,i-1}|-|W_{i,i}|-|X\cap \mathrm{Im}(\psi_{i+1}')|+|\{a_{i+1,1},a_{i+1,2}\}|\\
&\overset{\ref{prop_eil_0},\ref{prop_ei1_2}}{=}|W_{i+1}\cup X|-|\{w_{i,1,1},\ldots,w_{i,i-2,1}\}|-|\mathrm{Im}(\psi_{i+1}')|-|\{w_{i,i-1,1},w_{i,i,1}\}|+|\{a_{i+1,1},a_{i+1,2}\}|\\
&\;\;\;=|W_{i+1}\cup X|-(i-b_{i+1}')-|\mathrm{Im}(\psi_{i+1}')|+|\{a_{i+1,1},a_{i+1,2}\}|\\
&\;\;\,\overset{\eqref{eqn:wix}}=|P_{i+1}|+i-i-|\mathrm{Im}(\psi_{i+1}')|+|\{a_{i+1,1},a_{i+1,2}\}|+b_{i+1}'\\
&\;\;\;=|P_{i+1}|-|E_{i+1,1}\cup E_{i+1,2}|+|\{t_{i+1,1},t_{i+1,2}\}|+b_{i+1}'\\
&\;\;\;=|F_{i+1}|+b_{i+1}',
\end{align*}
so we have
\(|W_{i+1}^-\cup X_{i+1}^-|=|F_{i+1}|.\) Thus, we have shown that the conditions in the first two paragraphs of Lemma~\ref{lemma_finalembedding} hold, and it remains to check \itref{cond_finalembedlemma1}, \itref{cond_finalembedlemma2}, \itref{cond_finalembedlemma3}.

For each $j\in[i-2]\cup\{i+1\}$, the only edges appearing in $H'[W_{i+1,j}^-]$ are those coming from $\phi_i(P_i)$, thus, each such $H'[W_{i+1,j}^-]$ is $2$-colourable. Moreover, for all $j\in[i-2]$, $w_{i,j,2}$ is not in the image of $\phi_i$ by \ref{prop:newS0}, and hence the first element of $W_{i+1,j}^-$ (namely $w_{i,j,2}$) is an isolated vertex in $H''[W_{i+1,j}^-]$. Similarly, again  by \ref{prop:newS0}, the first element of $W_{i+1,i+1}^-$ (namely $w_{i,i+1,1}$) is an isolated vertex in $H''[W_{i+1,i+1}^-]$. Thus, \itref{cond_finalembedlemma2} holds for an application of Lemma~\ref{lemma_finalembedding}.
Furthermore, in $H''$, there are no edges between $W_{i+1}^-$ and $A^-_{i+1}$, or inside $A_{i+1}^-$, by \ref{property_A}, \ref{prop_eil_0} and \ref{prop:newS3}.
Thus, we will have that \itref{cond_finalembedlemma1} holds for an application of Lemma~\ref{lemma_finalembedding}. Moreover, similarly to before, each vertex in $W_{i+1}^-$ has at most $|W|=r$ leftward edges, and at most $r$ rightward edges  by \ref{property_forwardedge_new1}, \ref{property_forwardedge_new2}, \ref{prop_eil_0}, \ref{prop_ei1_3}, \ref{prop:groundset}, \ref{prop:forward} and \ref{prop:newS1} (as $w_{i,j,1}\notin W_{i+1}^-$ for each $j\in [i]$), for at most $2r\leq \lambda n$ neighbouring edges in total, so \itref{cond_finalembedlemma3} holds as well for an application of Lemma~\ref{lemma_finalembedding}.

Therefore, we have checked all the conditions to apply Lemma~\ref{lemma_finalembedding}. Taking $U=\{t_{i+1,1},t_{i+1,2}\}$ embedded to $a_{i+1,1},a_{i+1,2}$ respectively (recall \ref{prop_ei1_1} and also \itref{cond_finalembedlemma1} checked above to note this is a valid embedding), we get that
$\psi_{i+1}'$ extends to an embedding $\phi_{i+1}$ of $P_{i+1}$ in $G''\cup \psi_{i+1}'(E_{i+1,1}\cup E_{i+1,2})$ such that the following properties hold. 
\stepcounter{propcounter}
\begin{enumerate}[label = \textbf{\Alph{propcounter}\arabic{enumi}}]
\item $\operatorname{Im}(\phi_{i+1})\subseteq (W_{i+1}\cup X)\setminus\{w_{i,j,1}:j\in[i]\}$.\label{prop:newnewS0}
\item Each vertex in $\operatorname{Im}(\phi_{i+1})\cap W_{i+1}$ has at most 1 rightward neighbour in $\phi_{i+1}(P_{i+1})$, except for $w_{i,j,2}$, $j\in [i-2]$, and $w_{i,i+1,1}$, which have at most 2 rightward neighbours.\label{prop:newnewS1}
\item Every vertex in $V(G)\setminus (A_{i+1}^-\cup\operatorname{Im}(\psi_{i+1}'))$ with more than $\frac{n}{200}$ neighbours in $X$ in $H''$, as well as every vertex in $A\setminus A_{i+1}$, has at most 2 neighbours in $X$ in $\phi_{i+1}(P_{i+1})$.\label{prop:newnewS2}
\item $\phi_{i+1}(P_{i+1})$ has no edges between $W_{i+1}$ and $A\setminus A_{i+1}$, or inside $A\setminus A_{i+1}$.\label{prop:newnewS3}
	\end{enumerate}
Indeed, as before, \ref{prop:newnewS0} holds using \ref{prop_eil_0} (and the definition of $W_{i+1}^-$, $X_{i+1}^-$); \ref{prop:newnewS1} holds by \ref{prop_ei1_3} and \itref{prop:S1}; \ref{prop:newnewS2} holds by \itref{prop:S2} (and \ref{prop_eil_0}); and \ref{prop:newnewS0} holds by \itref{prop:S3}  and \ref{prop_eil_0}.

  \smallskip \noindent\textbf{Checking \ref{property_disjoint}--\ref{property_A}.}
	We check that all the conditions are satisfied for the extension of our good sequence by $\phi_i$ and $\phi_{i+1}$. Let $H_{i+1}$ be the subgraph of $K_n$ whose edges are those edges in the image of $\phi_1,\dots,\phi_{i+1},\psi_{i+2},\psi_{i+4},\dots,\psi_{r-1}$, and let $G_{i+1}$ be the complement of $H_{i+1}$ in $K_n$.


	Properties~\ref{property_disjoint}, \ref{property_image} and~\ref{property_extend} hold for $\phi_i$ and $\phi_{i+1}$ by our construction (and \ref{prop:newS0}, \ref{prop:newnewS0} for \ref{property_image}), while there is nothing additional to show for Property~\ref{property_forwardedge_new1}. For Property~\ref{property_forwardedge_new2}, it is sufficient to check that each vertex in $W_{i}$ has at most 2 rightward neighbours in $\phi_i(P_{i})\cup \phi_{i+1}(P_{i+1})$ (where Figure~\ref{fig:turnaround1} is a helpful reference). This certainly holds for the vertices in $W_{i,0}=W_{i+1}\setminus W_i$, by the properties of $\psi_i$ (\ref{prop_psi_1}--\ref{prop_psi_3}) together with  properties \ref{property_image} and \ref{property_extend} verified above, so it remains to check the vertices in $W_{i+1}$. From our embeddings, and in particular from \ref{prop:newS1} and \ref{prop:newnewS1}, every vertex in 
$W_{i+1}\setminus(\bigcup_{j\in [i-2]}\{w_{i,j,1},w_{i,j,2}\}\cup\{w_{i,i-1,1},w_{i,i,1},w_{i,i+1,1}\})$ 
has at most 1 rightward neighbour in  $\phi_i(P_{i})$ and at most 1 rightward neighbour in $\phi_{i+1}(P_{i+1})$.
  For each $j\in [i]$, by \ref{prop:newS1} and \ref{prop:newnewS0}, $w_{i,j,1}$ has at most 2 rightward neighbours in $\phi_{i}(P_{i})$ and no rightward neighbours in $\phi_{i+1}(P_{i+1})$.
 For each $j\in [i-2]$, by \ref{prop:newnewS1} and \ref{prop:newS0}, $w_{i,j,2}$ has at most 2 rightward neighbours in $\phi_{i+1}(P_{i+1})$ and no rightward neighbours in $\phi_{i}(P_{i})$, and this is also true for $w_{i,i+1,1}$. 
Thus, in total, we have that property \ref{property_forwardedge_new2} holds.

For Property~\ref{property_mindegree}, note that we have just checked that every vertex in $W$ has at most 2 (rightward) neighbours in $X$ in $\phi_i(P_i)\cup \phi_{i+1}(P_{i+1})$.
Furthermore, by \ref{prop_eil_0}, \ref{prop:groundset}, \ref{prop:newS2} and \ref{prop:newnewS2}, every vertex in $X\setminus (A_i\cup A_{i+1}\cup (X^{\text{high}}\cap \psi_i(V(E_{i,1}))))$ with at least $\frac{n}{100}-2\lambda n$ neighbours in $X$ in $H$ has at most $2$ neighbours in $X$ in each of $\phi_i(P_i)$ and $\phi_{i+1}(P_{i+1})$, and hence at most $\frac{n}{100}+2(i+1)$ neighbours in $X$ in $H_{i+1}$, using~\ref{property_mindegree} for the good sequence $\phi_1,\dots,\phi_{i-1}$. Moreover, each vertex in $(X^{\text{high}}\cap \psi_i(V(E_{i,1})))\setminus A_i$ has the same neighbourhood in $\phi_i(P_i)$ as in $\psi_i(E_{i,1})$, and has at most $2$ neighbours in $X$ in $\phi_{i+1}(P_{i+1})$ by \ref{prop_eil_0} and \ref{prop:newnewS2}, showing that the number of neighbours in $X$ in $H_{i+1}$ of such vertices is at most $\frac{n}{100}+2i$, once again by~\ref{property_mindegree} for the good sequence $\phi_1,\dots,\phi_{i-1}$. Furthermore, each vertex with fewer than $\frac{n}{100}-2\lambda n$ neighbours in $X$ in $H$ clearly has fewer than $\frac{n}{100}$ neighbours in $X$ in $H_{i+1}$ -- in particular, this holds for all vertices in $A_i\cup A_{i+1}$ by \ref{property_mindegree} for the good sequence $\phi_1,\ldots,\phi_{i-1}$.  
So we have proved that each vertex has at most $\frac{n}{100}+2(i+1)$ neighbours in $X$ in $H_{i+1}$.
Finally, by \ref{prop_eil_0}, \ref{prop:groundset}, \ref{prop:newS2} and \ref{prop:newnewS2}, each vertex in $\bigcup_{j\in[r]\setminus[i+1]}A_j$ has at most $2$ neighbours in $X$ in each of $\phi_{i}(P_i)$ and $\phi_{i+1}(P_{i+1})$, and hence, using~\ref{property_mindegree} for the good sequence $\phi_1,\ldots,\phi_{i-1}$, all such vertices have at most $\frac{n}{200}+2(i+1)$ neighbours in $X$ in $H_{i+1}$.
Thus, altogether, we have that Property~\ref{property_mindegree} holds.

Finally, 
\ref{prop:newS0}, \ref{prop:newnewS0}, \ref{prop:newS3} and \ref{prop:newnewS3}, together with~\ref{property_A} for the good sequence $\phi_1,\ldots,\phi_{i-1}$, show that Property~\ref{property_A} holds. Therefore, $\phi_1,\ldots,\phi_{i+1}$ is a good sequence, as required.
  \end{proof}


\section{Replacing star-like trees with stars}\label{section_starlike}
In this section we prove~\cref{theorem_treepacking} from Theorem~\ref{theorem_treepacking_starlikestars} by packing the trees with the `star-like' trees replaced by stars, and then replacing the embedded stars iteratively with embeddings of our actual star-like trees. To embed the star-like trees we use a method inspired by Havet, Reed, Stein and Wood~\cite{havet2020variant}.
This technique is very effective for embedding an $n$-vertex tree $T$ with many leaves into an $n$-vertex dense\footnote{For us, here, missing at most $2\eps n^2$  edges for some small $\eps>0$ (see Lemma~\ref{lemma_onestarlike}).} graph with one `universal' vertex, $v$ say, that is a neighbour of every other vertex. To do this (in a slightly simplified form), let $t\in V(T)$ have maximally many neighbouring leaves in $T$,
and let $L$ be a linear set of leaves containing the neighbouring leaves of $t$, where $t$ has, say, $d$ neighbours in $L$.
Next, we map $t$ to the universal vertex $v$, randomly embed the rest of $T-L$ 
vertex by vertex into $G$, and then attempt to attach the unused vertices in $G$, those in the set $B$ say, as leaves in $G$ of the embedded tree in the correct manner to get a copy of $T$. As $v$ is a universal vertex, every vertex in $B$ could be added in $G$ as a leaf next to the embedding of $t$ (which is $v$), and we have $d$ such leaves to add. We effectively add these leaves last, after adding leaves next to the other vertices that need leaves added, where we thus have $d$ spare vertices in $B$ while doing this. However, we can show that, with positive probability, in the embedding of $T-L$, the correct form of Hall's generalised matching criterion (see \cref{lem:hall}) holds for this to be done. The intuition behind this being true is the following argument. As the number of vertices needing leaves added is at least $|L|/d$, the smaller $d$ is, the more likely it is (or, more precisely, the better our bound on the probability is) for each vertex in $B$ to have many neighbours which need leaves attached, but the fewer spare vertices we have in $B$. On the other hand, the larger $d$, the less likely each vertex in $B$ has many neighbours which need leaves attached, but the more spare vertices we will have. All we need is for the balance between these competing forces to work in our favour in all cases, which we are able to show in our set-up (as we later do for Claim~\ref{clm:new}).

In order to use such an embedding result to iteratively replace our stars, we need to have a little more control in order to embed some leaves onto the lower degree vertices, but this is not a substantial complication. The result we need is the following lemma, after which we use this to prove Theorem~\ref{theorem_treepacking} from Theorem~\ref{theorem_treepacking_starlikestars}.

\begin{restatable}{lemma}{lemmaonestarlike}\label{lemma_onestarlike}
Let $1/n\ll \eps\ll \lambda \leq 1$. Let $T$ be a tree with $n$ vertices and at least $\lambda n$ leaves. Let $G$ be an $n$-vertex graph with $\delta(G)\geq n/10$ and at least $\binom{n}{2}-2\eps n^2$ edges. Let $v\in V(G)$ have degree $n-1$ in $G$.

Then, there is an embedding $\psi:T\to G$ for which the following hold.

\stepcounter{propcounter}
\begin{enumerate}[label = \textbf{\emph{\Alph{propcounter}\arabic{enumi}}}]
		\item Every vertex with degree at most $99n/100$ in $G$ is a leaf of $\psi(T)$.\label{prop_onestar1}
    \item If $\Delta(T)\geq 2n/3$, then the vertex of $T$ with maximum degree is mapped to $v$.\label{prop_onestar2}
	\end{enumerate}
\end{restatable}
\begin{proof}
Let $\mu$ satisfy $\eps\ll \mu \ll \lambda$. Let $t$ be a vertex of $T$ which has the maximal number of neighbours which are leaves in $T$, and let this number of neighbouring leaves be $d$. Note that if $\Delta(T)\geq 2n/3$, then $t$ is the (unique) vertex in $T$ with maximal degree. If $d\geq \mu n$ then let $L$ be the set of neighbouring leaves of $t$ in $T$, and otherwise let $L$ be a set of $\lambda n$ leaves of $T$ which includes all of the neighbouring leaves of $t$.

Let $T'=T-L$ and $m=|T'|$.
Label the vertices of $T'$ as $t_1,\ldots,t_{m}$ so that $t_1=t$ and each vertex (apart from $t_1$) has exactly $1$ leftward neighbour in $T'$ in this ordering. For each $i\in [m]$, let $d_i$ be the number of neighbours of $t_i$ in $L$ in $T$.

Let $U^{\textrm{low}}$ be the set of vertices in $G$ with degree less than $n-\mu n/10$ in $G$, so that $|U^{\textrm{low}}|\cdot \mu n/10\leq 4\eps n^2$, and, hence, $|U^{\textrm{low}}|\leq 40\eps  n/\mu\leq \mu n/10$.
Starting with $\psi(t_1)=v$, for each $2\leq i\leq m$ in turn, embed $t_i$ into $G$ (greedily) at random by letting $s_i$ be the unique neighbour of $t_i$ in $T[\{t_1,\ldots,t_i\}]$ and  selecting $\psi(t_i)$ uniformly at random from $N_G(\psi(s_i))\setminus (U^{\textrm{low}}\cup \psi(\{t_1,\ldots,t_{i-1}\}))$. Note that this is always possible, as $\psi(s_i)\in V(G)\setminus U^{\mathrm{low}}$ so that
\[
|N_G(\psi(s_i))\setminus (U^{\textrm{low}}\cup \psi(\{t_1,\ldots,t_{i-1}\}))|\geq \left(n-\frac{\mu n}{10}\right)-\frac{\mu n}{10}-|T'|= |L|-\frac{\mu n}{5}\geq \frac{4\mu n}{5}.
\]

Now, if $d\geq \mu n$, then we can complete the embedding of $T$ by attaching all of the $d$ vertices in $V(G)\setminus \psi(V(T'))$ as neighbours of $v$ (which has degree $n-1$ in $G$), and we have that both \emph{\ref{prop_onestar1}} and \emph{\ref{prop_onestar2}} hold. Therefore, we can assume that $d< \mu n$, and hence we have $|L|=\lambda n$. We will show the following claim.

\begin{claim} With probability at least $1/2$, there are at least $n-d$ vertices $u\in V(G)$ with
\begin{equation}\label{eqn:dis}
\sum_{i\in [m]:\psi(t_i)\in N_G(u)}d_i\geq  \frac{\lambda n}{100}.
\end{equation}\label{clm:new}
\end{claim}

Given this claim, we can complete the proof, as follows. First, there must by the claim be some embedding $\psi:T'\to G$ such that \eqref{eqn:dis} holds for all but at most $d$ vertices in $V(G)$.
Let $B=V(G)\setminus \operatorname{Im}(\psi)$, so that $|B|=n-(|T|-|L|)=|L|=\lambda n$.
Let $J=\{i\in [m]:d_i>0\}$. We will show that we can apply Hall's generalised matching criterion (see \cref{lem:hall}) to match up the remaining vertices $B$. Thus, consider some non-empty $I\subset J$. If $\sum_{i\in I}d_i\leq \frac{99\lambda n}{100}$, then, picking some arbitrary $j\in I$, as $\psi(t_j)\notin U^{\mathrm{low}}$, we have
\[
\bigg|\bigcup_{i\in I}N_G(\psi(t_i),B)\bigg|\geq |N_G(\psi(t_j),B)| \geq |B|-\frac{\mu n}{10}\geq \frac{99\lambda n}{100}\geq  \sum_{i\in I}d_i.
\]
If $\frac{99\lambda n}{100}< \sum_{i\in I}d_i\leq \lambda n-d$, then,
by the property from the claim, for the $\psi$ as above we have that all but at most $d$ vertices are in $\bigcup_{i\in I}N(\psi(t_i))$. Using this observation and then the assumption $\sum_{i\in I}d_i\leq \lambda n-d$, we get
\[
\bigg|\bigcup_{i\in I}N_G(\psi(t_i),B)\bigg|\geq |B|-d=\lambda n-d\geq \sum_{i\in I}d_i.
\]
Finally, if $\sum_{i\in I}d_i> \lambda n-d$, then as $d_1=d$ we have $1\in I$ so that, as $\psi(t_1)=v$ has degree $n-1$ in $G$,
\[
\bigg|\bigcup_{i\in I}N_G(\psi(t_i),B)\bigg|=|B|=\lambda n\geq \sum_{i\in I}d_i.
\]
Thus, by Hall's generalised matching criterion (see \cref{lem:hall}), we can extend $\psi$ to an embedding of $T$ in $G$. As $U^{\mathrm{low}}\subset \psi(L)$, we have that \emph{\ref{prop_onestar1}} holds, and \emph{\ref{prop_onestar2}} is vacuous as $d\leq \mu n$ implies $\Delta(T)<2n/3$. It is left then only to prove Claim~\ref{clm:new}.

\begin{proof}[Proof of Claim~\ref{clm:new}]
Let $u\in V(G)$. We will show that \eqref{eqn:dis} fails for $u$ with probability at most ${d}/(2n)$, so that the expected number of vertices for which \eqref{eqn:dis} fails is at most $d/2$ and hence the claim follows by Markov's inequality. We may assume that $u\not =v$, as~\eqref{eqn:dis} is immediate for $u=v$ as $v$ will always be a neighbour of each $\psi(t_i)$, $2\leq i\leq m$, and $d_1=d\leq\mu n$.

Thus, fix some $u\in V(G)\setminus\{v\}$, and, for each $i\in [m]$, let
\begin{equation}\label{eqn:Xui}
X_{i}=\left\{\begin{array}{ll}d_i & \text{if $\psi(t_i)\in N_G(u)$\; or \;$|N_G(u)\setminus (U^{\mathrm{low}}\cup \psi(\{t_1,\ldots,t_{i-1}\}))|\leq \frac{n-i+1}{25}$}\\
0 & \text{otherwise}.\end{array}\right.
\end{equation}
As $\psi(t_1)=v\in N_G(u)$, we always have $X_1=d_1=d$. Observe too that, for any $2\leq i\leq m$, as $n-i+1\geq |L|=\lambda n$ and $\lambda\gg \mu$,
\begin{align*}
\PP\left(X_{i}=d_i\;\Big|\; |N_G(u)\setminus (U^{\mathrm{low}}\cup \psi(\{t_1,\ldots,t_{i-1}\}))|\geq \frac{n-i+1}{25}\right)\geq \frac{\frac{n-i+1}{25}-\mu n/10}{n-i+1}\geq \frac{1}{50}
\end{align*}
so that $\PP(X_{i}=d_i \mid \psi(t_1),\ldots,\psi(t_{i-1}))\geq \frac{1}{50}$.
Therefore, as $\sum_{i\in [m]}d_i=|L|=\lambda n$, Azuma's inequality (\cref{{lem:azuma}}) for the submartingale $Z_{j}=\sum_{i\in[j]} (X_{i}-\frac{d_i}{50})$ gives (using $t=\lambda n/100$ and that $Z_m=\sum_{i\in [m]}X_i-\lambda n/50$)
\begin{equation}\label{eqn:sumXi}
\PP\bigg(\sum_{i\in [m]} X_{i}\leq \frac{\lambda n}{100}\bigg)\leq \exp\bigg(-\frac{(\lambda n/100)^2}{2\sum_{i\in [m]}d_i^2}\bigg)\leq\exp\bigg(-\frac{(\lambda n/100)^2}{2d\cdot \sum_{i\in [m]}d_i}\bigg)=  \exp\bigg(-\frac{\lambda n}{20000d}\bigg)\leq \frac{d}{4n},
\end{equation}
where this last inequality follows as $d/n\leq \mu \ll \lambda$.

Now, for each 
$i\in[m]$, let
\begin{equation}\label{eqn:Yi}
Y_{i}=|N_G(u)\setminus \psi(\{t_1,\ldots,t_{i-1}\})|-\frac{n-i+1}{10}.
\end{equation}
Note that, as $\delta(G)\geq n/10$, we have $Y_{1}\geq 0$. Moreover, $Y_{i+1}-Y_{i}\in\{-\frac{9}{10},\frac{1}{10}\}$ for all $i\in[m-1]$, so, in particular, $|Y_{i+1}-Y_{i}|\leq 1$.
We will now show that the probability that $Y_{i}\leq -2\mu n$ for some $i\in[m]$ is at most $\frac{1}{4n}$.

For all  $i,j\in[m]$ with $i<j$, let $\mathcal{B}_{i,j}$ be the event that $Y_{i},Y_{i+1},\dots,Y_{j}\leq -\mu n$, $Y_{i}\geq -\mu n-1$ and $Y_{j}\leq -2\mu n$. Note that if $Y_{j}\leq -2\mu n$, then $\mathcal{B}_{i,j}$ holds for some $i\in[j-1]$. Thus, to prove that there is some $i\in[m]$ with $Y_i\leq -2\mu n$ with probability at most $\frac{1}{4n}$, it is enough to show that $\mathbb{P}(\mathcal{B}_{i,j})\leq \frac{1}{4n^3}$ for all $i,j$. Observe that whenever $i\in[m-1]$ with $Y_i\leq -\mu n$, then (recalling that $s_i$ is the unique neighbour of $t_i$ in $T[\{t_1,\ldots,t_i\}]$) the probability that $\psi(t_i)\in N_G(u)$ is at most
\[
\frac{|N_G(u)\setminus \psi(\{t_1,\ldots,t_{i-1}\})|}
{n-(i-1)-|U^{\mathrm{low}}|-|V(G)\setminus N_G(\psi(s_i))|}\leq \frac{\frac{n-i+1}{10}-\mu n
}{n-(i-1)-\mu n/5}\leq \frac{1}{10}.
\]
Thus, we have 
\begin{equation}\label{eqn_Ymartingale}
\EE(Y_{i+1}\mid \psi(t_1),\dots,\psi(t_{i-1}))\geq Y_i\;\text{ if\; $Y_i\leq -\mu n$}.
\end{equation}
For each $i,k\in [m]$ with $i\leq k$, define $R_{i,k}$ by setting
\begin{equation*}
R_{i,k}=\left\{\begin{array}{ll} Y_k & \text{if $Y_i,Y_{i+1},\dots,Y_k\leq -\mu n$}\\
\max\{Y_i,\dots,Y_k\} & \text{otherwise}.\end{array}\right.
\end{equation*}
Note that we have $\EE(R_{i,k+1}\mid \psi(t_1),\dots,\psi(t_{k-1}))\geq R_{i,k}$ (and hence $\EE(R_{i,k+1}\mid R_{i,i},\dots,R_{i,k})\geq R_{i,k}$) for all $i,k\in[m-1]$ with $i\leq k$. 
Indeed, if $R_{i,k}>-\mu n$ then trivially $R_{i,k+1}\geq R_{i,k}$, whereas if $R_{i,k}\leq -\mu n$ then 
\[\EE(R_{i,k+1}\mid \psi(t_1),\dots,\psi(t_{k-1}))\geq \EE(Y_{k+1}\mid \psi(t_1),\dots,\psi(t_{k-1}))\overset{\eqref{eqn_Ymartingale}}\geq Y_{k}=R_{i,k}.\]
By Azuma's inequality (\cref{{lem:azuma}}) for the submartingale $(R_{i,k})_{k\geq i}$, since $|R_{i,k+1}-R_{i,k}|\leq 1$ for all $k$, we get that, for all $j\in[m]$, with $j>i$, the probability that $R_{i,j}-R_{i,i}\leq -\frac{\mu n}{2} $ is at most $\exp(-(\frac{\mu n}{2})^2/(2m))\leq \frac{1}{4n^3}$ as $m\leq n$. But this immediately implies that $\mathbb{P}(\mathcal{B}_{i,j})\leq \frac{1}{4n^3}$. This finishes the proof that $Y_i\leq-2\mu n$ for some $i\in [m]$ with probability at most $\frac{1}{4n}$.

Therefore, in combination with \eqref{eqn:sumXi}, with probability at least $1-\frac{d}{2n}$, we have that  $|Y_i|\geq -2\mu n$ for each $i\in [m]$ and $\sum_{i\in [m]} X_{i}\geq\frac{\lambda n}{100}$. For each $i\in [m]$, then, by \eqref{eqn:Yi},
we have $|N_G(u)\setminus \psi(\{t_1,\ldots,t_{i-1}\})|-\frac{n-i+1}{10}\geq -2\mu n$, so that
\[
|N_G(u)\setminus (U^{\mathrm{low}}\cup \psi(\{t_1,\ldots,t_{i-1}\}))|\geq \frac{n-i+1}{10}-|U^{\mathrm{low}}|-2\mu n>\frac{n-i+1}{25},
\]
as $|U^\mathrm{low}|\leq \mu n/10$ and $n-i+1\geq \lambda n$. Thus, by \eqref{eqn:Xui}, if $X_{i}=d_i$, then $\psi(t_i)\in N_{G}(u)$. Therefore,
\[
\sum_{i\in [m]:\psi(t_i)\in N_G(u)}d_i= \sum_{i\in [m]}X_{i}\geq  \frac{\lambda n}{100},
\]
so~\eqref{eqn:dis} holds for $u$. Hence, the probability that \eqref{eqn:dis} fails for $u$ is at most $\frac{d}{2n}$, as required.
\renewcommand{\qedsymbol}{$\boxdot$}
\qedhere
\renewcommand{\qedsymbol}{$\square$}
\qedsymbol
\end{proof}
\renewcommand{\qedsymbol}{}
\end{proof}
\renewcommand{\qedsymbol}{$\square$}

Using Lemma~\ref{lemma_onestarlike}, we can now prove \cref{theorem_treepacking} from \cref{theorem_treepacking_starlikestars}.

\begin{proof}[Proof of \cref{theorem_treepacking} from \cref{theorem_treepacking_starlikestars}]
Let $1/n\ll \eps\ll \lambda \ll 1$ and $r=\eps n$. Let $T_1,\ldots,T_r$ be trees with $|T_i|=n-i+1$ for each $i\in [r]$. Let $s$ be the number of these trees with at least $\lambda n$ leaves, and let these trees be $S_1,\ldots,S_s$ (in any order). For each $i\in [s]$, let $S_i'$ be a star with $|S_i|$ vertices. Let $T_1',\ldots,T_r'$ be the sequence of trees $T_1,\ldots,T_r$ with, for each $i\in [s]$, $S_i$ replaced by $S_i'$.
 Then, by Theorem~\ref{theorem_treepacking_starlikestars}, for some ordering $v_1,\ldots,v_n$ of $V(K_n)$ the sequence $T_1',\ldots,T_r'$ packs into $K_n$ so that every vertex which is not the centre of an embedded star has degree at most $n/10$ in the edges used in the embeddings, and the stars are embedded so that the centre is to the left of all its leaves. Relabelling if necessary (using that we did not label the stars in order of size), assume that, for each $i\in [s]$, the centre of the embedding of $S_i$ is $v_i$.
Picking such embeddings, let $H$ be the edges used by the embeddings of all the trees with fewer than $\lambda n$ leaves, and let $G$ be the complement of $H$ in $K_n$. As every vertex in $G$ is either not the centre of an embedded star, or has at least $n-r$ neighbours in $G$ from the embedding of a star, we have $\delta(G)\geq 9n/10$.

Then, to complete the proof it is sufficient to find disjointly in $G$ a copy of $S_i$ for each $i\in [s]$.
Therefore, Theorem~\ref{theorem_treepacking} follows directly from the following claim with $\ell=s$.

\begin{claim} For each $0\leq \ell\leq s$, $S_1,\ldots,S_\ell,S'_{\ell+1},\ldots,S'_{s}$ can be embedded into $G$ so that, letting $G_\ell$ be $G$ with the edges of the embedded copies of $S_1,\ldots,S_\ell$ removed,
\label{claim}
\stepcounter{propcounter}
\begin{enumerate}[label = \textbf{{\Alph{propcounter}\arabic{enumi}}}]
\item for each $\ell<i\leq s$, $v_i$ is the centre of the embedding of $S'_i$, which contains no edges $v_iv_j$ with $\ell<j<i$.  \label{prop_replace1}
\item for each $\ell<i\leq s$, any vertex in the embedding of $S'_i$ has degree at least $n/5-\ell$ in $G_\ell$.\label{prop_replace2}
\end{enumerate}
\end{claim}

\begin{proof}[Proof of Claim~\ref{claim}] We prove this by induction on $\ell$. When $\ell=0$, we have the embedding of the stars required, so suppose that $\ell\in [s]$ and we have embeddings of $S_1,\ldots,S_{\ell-1},S'_{\ell},\ldots,S'_s$ into $G$ so that, letting $G_{\ell-1}$ be $G$ with the edges of the embedded copies of $S_1,\ldots,S_{\ell-1}$ removed, any vertex in the embedding of $S'_i$ has degree at least $n/5-(\ell-1)$ in $G_{\ell-1}$ for each $\ell-1<i\leq s$. Suppose furthermore that $v_\ell,\ldots,v_{s}$ are the centres and $L_\ell,\ldots,L_s$ are the sets of leaves of the embedded copies of $S'_{\ell},\ldots,S'_s$ respectively, and no $S_i'$ contains an edge of the form $v_iv_j$ with $\ell-1<j<i$. We consider two cases according to the value of $\Delta(S_\ell)$.\medskip


	\textbf{Case 1: $\Delta(S_\ell)\geq 3n/4$.}
  Let $m=|L_\ell\cap \{v_{\ell+1},\ldots,v_{s}\}|\leq r$ and $L_\ell^-=L_\ell\setminus \{v_{\ell+1},\ldots,v_s\}$. Let $t_\ell$ be the vertex of $S_\ell$ with maximum degree, and let $S_\ell^-$ be a copy of $S_\ell$ with $m$ leaves removed next to $t_\ell$. Note that $|L_\ell^-\cup \{v_\ell\}|=|S_\ell^-|$ and that the embeddings of $S_{\ell+1}',\ldots,S_s'$ are disjoint from $G_{\ell-1}[L_\ell^-\cup \{v_\ell\}]$. Furthermore, as every vertex in $G_{\ell-1}[L_\ell^-\cup \{v_\ell\}]$ is in the embedding of $S_\ell'$, we have (by the version of \ref{prop_replace2}) that $\delta(G_{\ell-1}[L_\ell^-\cup \{v_\ell\}])\geq n/5-(\ell-1)-(n-|S_\ell|)-r\geq n/10$, while, as the embedding of $S_\ell'$ is in $G_{\ell-1}$, we have that $v_\ell$ is a neighbour of all of the other vertices in $G_{\ell-1}[L_\ell^-\cup \{v_\ell\}]$. Moreover, $G_{\ell-1}[L_\ell^-\cup \{v_\ell\}]$ is a graph with $n':=|S_\ell^-|\geq n-(n-|S_\ell|)-r\geq n-2r>99n/100$ vertices and at least $\binom{n'}{2}-rn\geq \binom{n'}{2}-2\eps n'^2$ edges.

  By Lemma~\ref{lemma_onestarlike}, then, there is an embedding of $S_{\ell}^-$ in $G[L_\ell^-\cup \{v_\ell\}]$ in which $t_\ell$ is embedded to $v_\ell$ and in which every vertex with degree at most $49n/50$ in $G_{\ell-1}[L_\ell^-\cup \{v_\ell\}]$ is a leaf of the embedded tree.
  Extend this to an embedding of $S_\ell$ by adding the $m$ vertices in $L_\ell\cap \{v_{\ell+1},\ldots,v_{s}\}$ as leaves of $t_\ell$. Let $G_{\ell}$ be $G$ with the edges of the embedded copies of $S_1,\ldots,S_\ell$ removed, so that it is $G_{\ell-1}$ with the edges of the embedded copy of $S_\ell$ removed.

Now, no embedded star $S_i'$ with $i>\ell$ contains $v_\ell$ (as the embedded $S_i'$ is centred at $v_i$ and contains no edge of the form $v_iv_j$ with $\ell-1<j<i$), so for each $i>\ell$, every vertex in the embedding of $S_\ell$ which is in the embedded $S'_i$ has degree at most $n-\Delta(S_\ell)\leq n/4$ in the embedding of $S_\ell$. As every vertex with degree at most $49n/50$ in $G_{\ell-1}[L_\ell^-\cup \{v_\ell\}]$ is a leaf of the embedding in $S_\ell$, we therefore have that, whenever $v$ is in the embedded $S_\ell$ and also a non-centre vertex of $S_i'$ for some $i>\ell$,
$v$ has degree in $G_{\ell}[L_\ell^-\cup \{v_\ell\}]$ at least $\min \{d_{G_{\ell-1}[L_\ell^-\cup \{v_\ell\}]}(v)-1,49n/50-n/4\}$. Therefore, for all $i>\ell$, every vertex of $S_i'$ has degree in $G_\ell$ at least
\[
\min \{d_{G_{\ell-1}}(v)-1,24n/50\}\geq n/5-\ell,
\]
so that \ref{prop_replace2} holds. Moreover, \ref{prop_replace1} holds as the embeddings of $S'_i$, $\ell<i\leq s$, are unchanged.

\medskip

\textbf{Case 2: $\Delta(S_\ell)<3n/4$.} Let $G_{\ell-1}^-$ be $G_{\ell-1}$ with $n/10$ rightward edges arbitrarily removed next to each $v_i$, $\ell<i\leq s$ (to force a leaf onto each of these vertices in the application of Lemma~\ref{lemma_onestarlike} which will follow).
Writing $L_\ell^-= L_\ell\setminus \{v_{\ell+1},\ldots,v_s\}$, for each vertex $v\in L_\ell^-$ we have $d_{G_{\ell-1}^-}(v,L_\ell^-\cup \{v_\ell\})\geq \frac{n}{5}-\ell-(s-\ell)-r\geq n/10$. Furthermore, for each $\ell<i\leq s$ we have $d_{G_{\ell-1}^-}(v_i,L_\ell^-\cup \{v_\ell\})\geq |L_\ell^-\cap L_i|-n/10\geq n-3r-n/10\geq n/10$. Finally, again we have that, as the embedding of $S_\ell'$ is in $G_{\ell-1}$, $v_\ell$ is a neighbour of every other vertex in $G^-_{\ell-1}[L_\ell\cup \{v_\ell\}]$, and that $G^-_{\ell-1}[L_\ell\cup \{v_\ell\}]$ is a graph on $n'\geq 99n/100$ vertices with at least $\binom{n'}{2}-rn-rn/10\geq \binom{n'}{2}-2\eps n'^2$ edges.

Thus, by Lemma~\ref{lemma_onestarlike}, there is an embedding of $S_{\ell}$ in $G^-_{\ell-1}[L_\ell\cup \{v_\ell\}]$ in which every vertex with degree at most $49n/50$ in $G_{\ell-1}^-[L_\ell\cup \{v_\ell\}]$ is a leaf of the embedded tree. Note that, as they have degree at most $9n/10$ in $G_{\ell-1}^-[L_\ell\cup \{v_\ell\}]$, whenever $\ell<i\leq s$ and $v_i\in L_\ell$, then $v_i$ is a leaf in the embedding of $S_\ell$, and so $v_i$ appears in a unique edge $f_{\ell,i}$ of the embedded $S_\ell$.
When this happens, and when we also have $f_{\ell,i}$ appearing in our embedding of $S_i'$, take the embedding of $S_{i}'$ and replace the edge $f_{\ell,i}$ by $v_\ell v_i$ (which is in $G_{\ell-1}$ but is not present in the embedding of $S_i'$ by our assumptions, so not equal to $f_{\ell,i}$, and hence not in the embedding of $S_\ell$ as $f_{\ell,i}$ is the unique such edge containing $v_i$). In other cases (i.e., if $\ell<i\leq s$ and either $v_i\not \in L_\ell$ or $f_{\ell,i}$ does not appear in the embedding of $S_i'$), we leave $S_i'$ unchanged.
Thus, we get disjoint embeddings of $S_1,\ldots,S_\ell,S'_{\ell+1},\ldots,S'_{s}$. Furthermore, we have that \ref{prop_replace1} holds as we only added to the embeddings of $S'_{\ell+1},\ldots,S'_{s}$ some edges containing $v_\ell$.

Let $G_{\ell}$ be $G$ with the edges of the embedded copies of $S_1,\ldots,S_\ell$ removed, so that it is $G_{\ell-1}$ with the edges of the embedded copy of $S_\ell$ removed. Similarly to Case 1, for each $\ell<i\leq s$, any vertex in the embedding of $S'_i$ has degree at least $n/5-\ell$ in $G_\ell$, except for (in this case only) possibly $v_\ell$ which may now be in the embedding of $S'_i$. However, $v_\ell$ had at least $|S_\ell|-1\geq n-r-1$ neighbours in $S_{\ell-1}$, and so it has at least $n-r-1-\Delta(S_\ell)\geq n/5$ neighbours in $G_\ell$. Thus, \ref{prop_replace2} holds. This completes the proof of the claim, and hence the lemma.
\renewcommand{\qedsymbol}{$\boxdot$}
\qedhere
\renewcommand{\qedsymbol}{$\square$}
\qedsymbol
\end{proof}
\renewcommand{\qedsymbol}{}
\end{proof}
\renewcommand{\qedsymbol}{$\square$}


\section{Embedding the path-like trees: proof of the key lemmas}\label{section_ends}
In this section, we prove our two key lemmas,~\cref{lemma_forends} and~\cref{lemma_finalembedding}. We begin by formally defining an `end'. For convenience, and in a slight change from the sketch, essentially an end of a tree $T$ is a subtree $T'$ with at most 2 connector vertices in $T$ (vertices $v$ with $d_{T'}(v)<d_T(v)$). Then, using the notation of Sections~\ref{sec:sketch} and~\ref{section_mainembeddingscheme}, $T'$ is moreover a `proper' end if it can be embedded in $K_n$ with connector vertices in $X$ so that plenty of its vertices are used to cover vertices in $W$ without using a `turnaround' vertex. Where a turnaround vertex is needed (that is, when too many vertices of $T'$ lie on the path between connector vertices in $T'$), $T'$ is then an `artificial' end. Within the definition we additionally record two parameters: $k$ will be number of vertices we can easily cover in $W$ while embedding $T'$ in such a fashion, while $\ell$ is the number of non-connector leaves of $T'$. We now make these definitions precisely.

\begin{definition}\label{defn:ends}
	Let $T$ be a tree, and let $T'$ be a (proper) subtree of $T$ with two not necessarily distinct connector vertices $u$ and $v$, and no other connector vertices other than $u$ and $v$. If $k\geq 0$ satisfies $|T'|\geq 20000k$ and $\ell$ is the number of leaves of $T'$ distinct from both $u$ and $v$, then \emph{$T'$ is a $(k,\ell)$-end of $T$ (with connector vertices $u,v$)}.
	Moreover, letting $P_{u,v}$ be the unique path in $T'$ between $u$ and $v$, if $|P_{u,v}|\leq \frac{99}{100}|T'|$, then $T'$ is a \emph{proper $(k,\ell)$-end}, while otherwise it is an \emph{artificial $(k,\ell)$-end}. Note that if $|T'|\geq 2$ and $T'$ is an artificial end, then, necessarily, $u\neq v$.
\end{definition}

In Section~\ref{section_genuineend}, we will prove the first of the key lemmas (\cref{lemma_forends}), before building to the proof of the second key lemma (\cref{lemma_finalembedding}) through the rest of the section, using several intermediate results. To do this, in Section~\ref{section_key2}, we give a result embedding a sufficiently large end to cover `most' vertices in one of the sets  $W_i^-$ in the statement of~\cref{lemma_finalembedding}, leaving only a few gaps and making sure that~\itref{prop:S1} is satisfied for the vertices covered. Then, in Section~\ref{section_key5}, we show that we can embed a subtree while using its leaves to cover the gaps left by the previous step. In Section~\ref{section_key6},
we give a result which we will use to embed the final piece of our path-like tree, making sure that we cover a certain subset with well-separated degree $2$ vertices (so that~\itref{prop:S2} and~\itref{prop:S3} can be satisfied). Finally, we put this all together in Section~\ref{section_key7} to prove~\cref{lemma_finalembedding}.


\subsection{Proof of the first key lemma}\label{section_genuineend}
In the proof of the first key lemma, and other results in this section, we often wish to embed part of a tree greedily. It is convenient, then, to record when this can be done, which we do in the following result.
\begin{lemma}\label{lemma_greedy}
    Let $G$ be a graph with complement $H$, and let $T$ be a tree. Let $U_1\subset U_2\subset V(T)$, and let $\kappa$ be a positive integer such that $T[U_1]$ has at most $\kappa$ connected components. Assume furthermore that $|U_2|+\kappa\cdot\Delta(H)\leq |G|$. Then, every embedding of $T[U_1]$ in $G$ extends to an embedding of $T[U_2]$ in $G$.
\end{lemma}
\begin{proof}
    List the vertices of $V(T)\setminus U_1$ as $t_1,\dots,t_m$ in such a way that, for all $i\in[m]$, we have $t_i\in N_T(U_1\cup\{t_1,\dots,t_{i-1}\})$ (if $U_1=\emptyset$, then we only require this to hold for all $i\in [m]\setminus\{1\}$). Then, for all $i\in[m]$, $T[U_1\cup\{t_1,\dots,t_i\}]$ has at most $\kappa$ connected components. Hence, for all $i\in[m]$, $t_i$ has at most $\kappa$ neighbours in $T$ in $U_1\cup\{t_1,\dots,t_{i-1}\}$.
    
    Thus, by taking a subsequence of $t_1,\dots,t_m$, we can list the vertices in $U_2\setminus U_1$ as $s_1,\dots,s_q$ such that, for all $i\in[q]$, $s_i$ has at most $\kappa$ neighbours in $U_1\cup\{s_1,\dots,s_{i-1}\}$.
    Then, we can greedily embed the vertices of $U_2\setminus U_1$ in this order. Indeed, if we have an embedding $\phi$ of $T[U_1\cup\{s_1,\dots,s_{i-1}\}]$ in $G$, then, writing $Y=N_T(s_i)\cap (U_1\cup\{s_1,\dots,s_{i-1}\})$, we can pick $\phi(s_i)$ to be any element of the set $V(G)\setminus(\bigcup_{y\in Y}N_H(\phi(y))\cup \phi(U_1\cup\{s_1,\dots,s_{i-1}\}))$, which has size at least $|G|-\kappa\cdot\Delta(H)-(|U_2|-1)>0$.    
\end{proof}

We can now prove~\cref{lemma_forends}, which is used in our main proof to embed the `ends' $E_{i,1},E_{i,2}$ -- which are subtrees of $T_i$ with only one connector vertex in $T_i$. First, we restate the lemma, for convenience.


\lemmaforends*

\begin{proof}
	Let $a\in A^-$ be given. If $W^-=\emptyset$, then a greedy algorithm (i.e., \cref{lemma_greedy} with $U_1=\{t\}$, $U_2=V(T)$ and $\kappa=1$, as $|X\setminus X^\text{forb}|\geq n-r-n/100$, $\Delta(H[X\setminus X^\text{forb}])\leq n/50$ and $|T|\leq 4\eps n$) shows that $T$ embeds in $G[X\setminus X^\text{forb}]$ with $t$ mapping to $a$. Thus, from now on, we may assume that $W^{-}\not =\emptyset$.
 
 By~\cref{lemma_genuineenddivision}, as $|T|>|W^-|$, there is an integer $k\geq 2$ and a tree decomposition $T_1,\dots,T_k$ of $T$ such that $t\in V(T_k)$, $|T_k|=|T|-|W^-|$, $|T_1|-1\geq |W^-|/2$, and, for each $i\in[k-1]$, $T_i$ contains exactly one connector. 
By merging $T_i$ and $T_j$ for all $i,j\in [k-1]$ if they share a vertex, we may assume that $T_1,\dots,T_{k-1}$ are vertex-disjoint, and note that we can assume each of these trees has at least one edge. We will essentially show that if we replace each $T_i$ ($i\in [k-1]$) with a star of the same size, centred at its connector vertex, then we can embed this new tree, $T'$ say, in $G$ such that $W^-$ is covered by the leaves of these stars, and then use this observation to deduce the statement of the lemma by putting back the trees $T_i$ in place of the stars.
 Thus, for each $i\in[k-1]$, let $s_i$ be the unique connector of $T_i$, and 
 note that $s_i\in V(T_k)$. 
 For each $i\in [k-1]$, let $d_{i}=|T_i|-1$. Note that 
 $\sum_{i\in [k-1]} d_i=|W^-|$ and $d_1\geq |W^-|/2$.

	Let $X_0=\{v\in X\setminus X^{\text{forb}}:d_G(v,W^-)\geq |W^-|/2\}=\{v\in X\setminus X^{\text{forb}}:d_H(v,W^-)\leq |W^-|/2\}$. Note that $e_H(W^-,X\setminus X^{\text{forb}})\leq |W^-|\cdot n/50$ as $\Delta(H)\leq n/50$, so $|(X\setminus X^{\text{forb}})\setminus X_0|\leq \frac{|W^-|\cdot n/50}{|W^-|/2}=n/25$, and hence  $|X_0|\geq |X|-|X^{\text{forb}}|-n/25\geq 23n/25$. Note that $A^-\subset X_0$. Take any embedding $\phi_0$ of $T_k[\{t,s_1\}]$ in $G[A^-]$ with $\phi_0(t)=a$. We can extend this embedding greedily (by~\cref{lemma_greedy} with $U_1=\{t,s_1\}$ and $U_2=V(T_k)$, using that $\Delta(H[X_0])\leq n/50$, $|X_0|\geq 23n/25$ and $|U_2|\leq 4\eps n$) to get an embedding $\phi_0$ of $T_k$ in $G[X_0]$ such that $\phi_0(t)=a$ and $\phi_0(s_1)\in A^-$.

	We claim that we can find a partition $W^-=\cup_{i\in [k-1]}W_i^-$ such that, for all $i\in [k-1]$, $W_i^-\subseteq N_G(\phi_0(s_i))$ and $|W_i^-|=d_i$. We show this using Hall's generalised matching criterion (see \cref{lem:hall}). For this, take any non-empty $I\subseteq [k-1]$. If $\sum_{i\in I}d_i\leq |W^-|/2$, then the definition of $X_0$ immediately gives that $|N_G(\{\phi_0(s_i):i\in I\},W^-)|\geq |W^-|/2\geq \sum_{i\in I}d_i$. On the other hand, if $\sum_{i\in I}d_i>|W^-|/2$, then necessarily $1\in I$, and hence, using $\phi_0(s_1)\in A^-$, we have $N_G(\{\phi_0(s_i):i\in I\},W^-)=W^-$. So, by~\cref{lem:hall}, such a partition does indeed exist.

    For each $i\in[k-1]$, order the vertices of $T_i$ as $t_{i,1},\dots,t_{i,d_i+1}$ such that $t_{i,d_i+1}=s_i$ and each vertex has at most $1$ rightward neighbour. Using that $W_i^-\cup\{\phi_0(s_i)\}$ induces a complete subgraph of $G$ as $W_i^-\subset N_G(\phi_0(s_i))$, we can extend $\phi_0$ to an embedding $\phi$ of $T$ in $G[W^-\cup X_0]$ by mapping, for each $i\in[k-1]$, $t_{i,1},\dots,t_{i,d_i}$ to $W_i^-$ in an order-preserving way. This embedding $\phi$ then satisfies the conditions.
\end{proof}


\subsection{Embedding ends}\label{section_key2}
It remains to prove the other key lemma,~\cref{lemma_finalembedding}. In this section, we will prove the following result, which states that we can use a sufficiently large end to appropriately cover most of the vertices in one of the sets $W_i^-$ in~\cref{lemma_finalembedding} (while, as required by~\itref{prop:S1}, having at most one rightward edge at all points of $W_i^-$, except at the leftmost vertex). The number of gaps left by the embedding will be a small fraction of the number of (non-connector) leaves of the end that we embed.

\begin{restatable}{lemma}{lemmaeitherend}\label{lemma_eitherend} Let $1/n\ll \eps\ll \lambda \ll 1$ and $r=\eps n$.
Let $v_1,\ldots,v_n$ be an ordering of the vertices of $K_n$ and let $W=\{v_1,\ldots,v_r\}$ and $X=\{v_{r+1},\ldots,v_{n}\}$. Let $G\subset K_n$ have $\delta(G)\geq 49n/50$, and let $H$ be the complement of $G$ in $K_n$. Suppose that $e(H)\leq \eps n^2$.

Let $W^-\subset W$, and let $X^-\subset X$ contain at least $5n/6$ vertices. Let $T$ be a tree and let $S$ be a $(36|W^-|,\ell)$-end of $T$ for some $\ell$, with $|S|\leq n/50$. Suppose that 
each vertex in $W^-$ has at most $2\lambda n$ edges in $H$,
$H[W^-]$ is $2$-colourable, and the leftmost vertex of $W^-$ is an isolated vertex in $H[W^-]$ (or $W^- =\emptyset$).

  Then, if $U\subset V(S)$ is a subset of the (at most two) connector vertices of $S$ in $T$, then any embedding of $S[U]$ in $G[X^-]$ extends to an embedding of $S$ in $G[W^-\cup X^-]$, such that, if $S'$ is the embedded copy of $S$, then the following conditions are satisfied.
  \stepcounter{propcounter}
  \begin{enumerate}[label = \textbf{\emph{\Alph{propcounter}\arabic{enumi}}}]
  \item Each vertex of $S'$ in $W^-$ has at most $1$ rightward neighbour in $S'$, unless it is the leftmost vertex in $W^-$, when it has at most $2$ rightward neighbours in $S'$.\label{prop:Snewnew1}
		\item There are at most $\frac{\ell}{100}$ vertices in $W^-$ which are not in $S'$.\label{prop:Snewnew2}
    \end{enumerate} 
\end{restatable}

As a first step, we will consider proper ends, and show in the following lemma that if $W^-\subset W$ induces a complete subgraph of $G$, then a proper $(|W^-|,\ell)$-end can be used to cover $W^-$ with few (at most $\frac{\ell-1}{100}$) gaps while having at most one rightward neighbour at each point of $W^-$. This lemma will then be used to deduce~\cref{lemma_eitherend} for artificial ends (by viewing the artificial end as two proper ends glued together, see \cref{lemma_artificialend}), and finally for both end types (again, by using two proper sub-ends to cover the two colour classes of $H[W^-]$).

\begin{restatable}{lemma}{lemmaproperend}\label{lemma_properend} Let $1/n\ll \eps\ll \lambda \ll 1$ and $r=\eps n$.
Let $v_1,\ldots,v_n$ be an ordering of the vertices of $K_n$ and let $W=\{v_1,\ldots,v_r\}$ and $X=\{v_{r+1},\ldots,v_{n}\}$. Let $G\subset K_n$ have $\delta(G)\geq 49n/50$, and let $H$ be the complement of $G$ in $K_n$. Suppose that $e(H)\leq \eps n^2$.

Let $W^-\subset W$, and let $X^-\subset X$ contain at least $3n/4$ vertices. Let $T$ be a tree, and let $S$ be a proper $(|W^-|,\ell)$-end of $T$ for some $\ell$, with $|S|\leq n/2$. Suppose that 
each vertex in $W^-$ has at most $2\lambda n$ edges in $H$, and $H[W^-]$ has no edges.


  Then, if $U\subset V(S)$ is a subset of the (at most two) connector vertices of $S$ in $T$, then any embedding of $S[U]$ in $G[X^-]$ extends to an embedding of $S$ in $G[W^-\cup X^-]$, so that, if $S'$ is the embedded copy of $S$, then the following conditions are satisfied.
  \stepcounter{propcounter}
  \begin{enumerate}[label = \textbf{\emph{\Alph{propcounter}\arabic{enumi}}}]
  \item Each vertex of $S'$ in $W^-$ has at most 1 rightward neighbour in $S'$.\label{prop:Snew1}
		\item There are at most $\frac{\ell-1}{100}$ vertices in $W^-$ which are not in $V(S')$.\label{prop:Snew2}
    \end{enumerate}
		Moreover, if $W^-\not =\emptyset$, and $t$ is a leaf of $S$ which is not a connector vertex, or adjacent to one, then such an embedding exists in which $t$ is embedded to the leftmost vertex of $W^-$.
\end{restatable}
\begin{proof}
First, note that if $|W^-|-1\leq \lfloor(\ell-1)/100\rfloor$, then it is sufficient to embed $S$ into $G[X^-\cup \{w\}]$, where $w$ is the leftmost vertex of $W^-$, with $t$ embedded to $w$ if necessary, and \eref{prop:Snew1} satisfied, as \eref{prop:Snew2} will then hold. Such an embedding is easy to construct greedily (by~\cref{lemma_greedy}) if there is a leaf of $S$ which is not a connector vertex or adjacent to a connector vertex in $S$. On the other hand, suppose every leaf of $S$ is a connector vertex or adjacent to a connector vertex in $S$ (and note, in particular, that therefore no $t$ exists to be embedded specifically to $w$). Then, as $S$ is a proper end, it has at least $|S|/100$ vertices not on the path between its connector vertices, and each of these vertices must then be a leaf adjacent to a connector vertex in $S$. Thus, as $S$ is a $(|W^-|,\ell)$-end, $\ell\geq |S|/100\geq 200|W^-|$. Then (considering the cases $|W^-|=0$ and $|W^-|>0$ separately) \eref{prop:Snew2} is trivial, and we can embed $S$ greedily (by~\cref{lemma_greedy}) into $G[X^-]$,  whereupon \eref{prop:Snew1} and \eref{prop:Snew2} will hold.

Suppose, therefore, that $|W^-|>\lfloor(\ell-1)/100\rfloor+1$. Then, as $S$ is a (proper) $(|W^-|,\ell)$-end, $\ell\leq \frac{|S|}{200}$. Let $S^-\subset S$ be the minimal subtree of $S$ containing every connector vertex in $S$ and their neighbours in $S$. Note that, as $S$ is a proper end, $|S^-|\leq \frac{99}{100}|S|+\ell\leq \frac{199}{200}|S|$.
In particular, if no $t$ has been specified, we can choose $t$ to be a leaf of $S$ which is not in $S^-$, where it is then not a connector vertex or adjacent to one, and thus we must also have that $\ell\geq 1$. Furthermore, as $\ell\leq 100|W^-|$, we have $|W^-|\geq 1$, so we can let $w$ be the leftmost vertex of $W^-$.

Note that $|S|-|S^-|\geq |S|/200> |W'|\geq m:=|W'|-1-\lfloor(\ell-1)/100\rfloor\geq 1$. Thus, by applying Lemma~\ref{lemma_genuineenddivision} to the tree $S-t$ with $S^-$ contracted to a single vertex (applying the lemma with $m$ and the contracted vertex), 
we can find, for some $k$, a tree decomposition of $S-t$ into $S_1,\ldots,S_{k}$ so that $S^-\subset S_k$, 
\begin{equation}\label{eqn:Sksize}
|S_k|=|S-t|-m=|S-t|-|W^-|+1+\bigg\lfloor\frac{\ell-1}{100}\bigg\rfloor,
\end{equation}
and each $S_i$ ($i\in[k-1]$) contains exactly one connector.  Furthermore, we may assume that for all $i\in [k-1]$ we have $|S_i|\geq 2$. Note that, then, as $S^-$ contains every connector vertex in $S$ and their neighbours in $S$, each $S_i$ ($i\in [k-1]$) does not contain any connectors of $S$. For each $i\in [k-1]$, let $m_i$ be the number of neighbours in $S_i$ of the connector of $S_i$, and note that we have $\sum_{i\in [k-1]}m_i\leq \ell$ and $\sum_{i\in [k-1]}m_i\leq |S-t|-|S_k|=|W^-|-1-\lfloor\frac{\ell-1}{100}\rfloor$. Let $\ell_0=\sum_{i\in [k-1]}m_i$, so $\ell_0\leq \ell$ and $|W^-|\geq\ell_0+\lfloor\frac{\ell-1}{100}\rfloor+1$.

Let $Z^-$ denote the set of $\ell_0+\lfloor\frac{\ell-1}{100}\rfloor$ rightmost vertices of $W^-$, and let $X_0^-$ be the set of vertices $v\in X^-$ which are neighbours of $w$ (the leftmost vertex of $W^-$) in $G$ and which have at least $\ell_0$ neighbours in $Z^-$ in $G$. Then, as each vertex in $W^-$ has at most $2\lambda n$ edges in $H$,
\[
|X_0^-|\geq |X^-|-2\lambda n-\frac{(\ell_0+\lfloor\frac{\ell-1}{100}\rfloor)\cdot 2\lambda n}{\lfloor\frac{\ell-1}{100}\rfloor+1}\geq |X^-|-\bigg(1+\frac{\ell+\lfloor\frac{\ell-1}{100}\rfloor}{\lfloor \frac{\ell-1}{100}\rfloor+1}\bigg)\cdot 2\lambda n\geq 2n/3,
\]
using that $|X^-|\geq 3n/4$, $\ell\geq 1$ and $\lambda\ll 1$.
Let $\phi:S[U]\to G[X^-]$ be an embedding of a set $U$ of connector vertices of $S$. As before, we can (by~\cref{lemma_greedy}) greedily extend this to an embedding of $S_k$ in $G[X^-]$ with all the vertices outside of $U$ embedded to $X_0^-$. For each $i\in[k-1]$, let $x_i$ be the image of the connector of $S_i$, so that, as $S_i$ contains no connectors of $S$ (and hence no vertices in $U$), we have $x_i\in X_0^-$.

Using the definition of $X_0^-$, we can greedily find disjoint subsets $Z_1^-,\dots, Z_{k-1}^-$ of $Z^-$ such that, for all $i\in [k-1]$, $|Z_i^-|=m_i$ and  $Z_i^-\subset N_G(x_i)$. Moreover, since
$$\sum_{i\in[k-1]}(|S_i|-m_i-1)=|S-t|-|S_k|-\sum_{i\in[k-1]}m_i=|W^-|-1-\left\lfloor\frac{\ell-1}{100}\right\rfloor-\ell_0=|W^-|-1-|Z^-|,$$ we can easily pick disjoint $Y_1^-,\dots, Y_{k-1}^-\subset W^-\setminus(\{w\}\cup Z^-)$ such that $|Y_i^-|=|S_i|-m_i-1$ for all $i\in[k-1]$. Thus, by setting $W_i^-=Y_i^-\cup Z_i^-$ for all $i\in [k-1]$, we can find $W_1^-,\ldots,W_{k-1}^-$ in $W^-\setminus \{w\}$ such that, for each $i\in[k-1]$, $|W_i^-|=|S_i|-1$, and the $m_i$ rightmost points of $W_i^-$ are neighbours of $x_i$. For each $i\in[k-1]$, we can embed $S_i$ in $G[W_i^-\cup\{x_i\}]$ so that the connector is mapped to $x_i$ and each vertex in $W_i^-$ has at most $1$ rightward neighbour in this embedding. Indeed, we can order the vertices of $S_i$ as $s_{i,1},\dots,s_{i,|S_i|}$ so that $s_{i,|S_i|}$ is the connector, $s_{i,|S_i|-1},\dots,s_{i,|S_i|-m_i}$ are its neighbours, and every vertex has at most $1$ rightward edge, and then we can simply map $V(S_i)$ into $W_i^-\cup\{x_i\}$ in an order-preserving way. Putting together these embeddings for each $S_i$, $i\in[k-1]$, as well as our embedding $\phi$, we get an extension of $\phi$ to $S-t$. Finally, we can extend $\phi$ to $S$ by embedding $t$ to $w$, since $\phi$ maps the unique neighbour of $t$ to a point in $(W^-\setminus\{w\})\cup X_0^-\subset N_G(w)$. Then, we have $\phi(t)=w$, and, writing $S'$ for the embedded copy of $S$, each point in $V(S')\cap W$ has at most one rightward neighbour, and 
\[
|W^-\setminus V(S')|=|W^-|-\sum_{i\in[k-1]}(|S_i|-1)-1=|W^-|-(|S-t|-|S_k|)-1\overset{\eqref{eqn:Sksize}}{=}\left\lfloor\frac{\ell-1}{100}\right\rfloor,
\]
as desired.
\end{proof}



Next, we use \cref{lemma_properend} to deduce the following form of~\cref{lemma_eitherend} for artificial ends.

\begin{restatable}{lemma}{lemmaartificialend}\label{lemma_artificialend} Let $1/n\ll \eps\ll \lambda \ll 1$ and $r=\eps n$.
Let $v_1,\ldots,v_n$ be an ordering of the vertices of $K_n$ and let $W=\{v_1,\ldots,v_r\}$ and $X=\{v_{r+1},\ldots,v_{n}\}$. Let $G\subset K_n$ have $\delta(G)\geq 49n/50$, and let $H$ be the complement of $G$ in $K_n$. Suppose that $e(H)\leq \eps n^2$.

Let $W^-\subset W$, and let $X^-\subset X$ contain at least $4n/5$ vertices. Let $T$ be a tree and let $S$ be an artificial $(4|W^-|,\ell)$-end of $T$ for some $\ell$, with $|S|\leq n/50$. Suppose that 
each vertex in $W^-$ has at most $2\lambda n$ edges in $H$, $H[W^-]$ is $2$-colourable, and the leftmost vertex of $W^-$ is an isolated vertex in $H[W^-]$ (or $W^-=\emptyset$).


  Then, if $U\subset V(S)$ is a subset of the (two) connector vertices of $S$ in $T$, then any embedding of $S[U]$ in $G[X^-]$ extends to an embedding of $S$ in $G[W^-\cup X^-]$, such that, if $S'$ is the embedded copy of $S$, then the following conditions are satisfied.
  \stepcounter{propcounter}
  \begin{enumerate}[label = \textbf{\emph{\Alph{propcounter}\arabic{enumi}}}]
  \item Each vertex of $S'$ in $W^-$ has at most $1$ rightward neighbour in $S'$, unless it is the leftmost vertex in $W^-$, when it has at most $2$ rightward neighbours in $S'$.\label{prop:Snewother1}
		\item There are at most $\frac{\ell}{100}$ vertices in $W^-$ which are not in $S'$.\label{prop:Snewother2}
    \end{enumerate}
\end{restatable}
\begin{proof}
If $W^-=\emptyset$, then the result follows easily from a greedy algorithm (see \cref{lemma_greedy}), thus, we may assume from now on that $W^-\not =\emptyset$. Let $u$ and $v$ be the (distinct, as $S$ is an artificial end with $|S|\geq 1$) connector vertices of $S$ in $T$ and let $P$ be the $u,v$-path in $S$. As $S$ is an artificial end, $|P|\geq \frac{99}{100}|S|$. Then, at least $\frac{98}{100} |S|$ vertices on $P$ have no neighbours in $S$ in $V(S)\setminus V(P)$. Let $t$ be a vertex of $P$ in the middle third of the path such that 
$t$ does not have any neighbours in $V(S)\setminus V(P)$. The vertex $t$ cuts $S$ into two subtrees intersecting in $t$, let us denote these by $S_1$ and $S_2$ (where we include $t$ in both $S_1$ and $S_2$), and label them so that $u\in V(S_1)$ and $v\in V(S_2)$. Similarly, $t$ cuts $T$ into two trees $T_1$ and $T_2$, intersecting in $t$, such that $S_i\subset T_i$ for each $i\in[2]$. 
 For each $i\in[2]$, say $S_i$ contains $\ell_i$ leaves distinct from $u$ and $v$, so $\ell_1+\ell_2=\ell+2$, where we used that $t$ is a leaf of both $S_1$ and $S_2$, but not $T$. Note that, for each $i\in[2]$, $|S_i|\geq |P|/3\geq \frac{99}{300} |S|$, hence, as $S$ is a $(4|W^-|,\ell)$-end, we get that $S_i$ is a proper $(|W^-|,\ell_i)$-end of $T_i$ (with a single connector vertex, either $u$ or $v$).

Let $w$ be the leftmost vertex of $W^-$. We know that $H[W^-]$ is $2$-colourable and $w$ is an isolated vertex, so we can take $W_1^-$, $W_2^-\subset W^-$ such that $W_1^-\cup W_2^-=W^-$, $W_1^-\cap W_2^-=\{w\}$, and $H[W_1^-], H[W_2^-]$ are empty. Let $\phi$ be the given embedding of $S[U]$ in $G[X^-]$. We may assume that $U=\{u,v\}$ by, otherwise, extending $\phi$ arbitrarily to an injective function $\{u,v\}\to X^-$. 

By our choice of $t$, we have that $t$ is a leaf of $S_1$ which is not a neighbour of $u$. Applying~\cref{lemma_properend}, then, for the proper $(|W_1^-|,\ell_1)$-end $S_1$ of $T_1$, we deduce that there exists an embedding $\psi_1$ of $S_1$ in $G[W_1^-\cup(X^-\setminus\{\phi(v)\})]$ such that $\psi_1(u)=\phi(u)$, $\psi_1(t)=w$, and, letting $S_1'$ be the embedded copy of $S_1$, each vertex in $W_1^-\cap V(S_1')$ has at most $1$ rightward edge in $S_1'$, and $|W_1^-\setminus V(S_1')|\leq \frac{\ell_1-1}{100}$. Similarly, applying~\cref{lemma_properend} again, there is an embedding $\psi_2$ of $S_2$ in $G[W_2^-\cup(X^-\setminus V(S_1'))]$ such that $\psi_2(v)=\phi(v)$, $\psi_2(t)=w$, and, letting $S_2'$ be the embedded copy, each vertex in $W_2^-\cap V(S_2')$ has at most $1$ rightward edge in $S_2'$, and $|W_2^-\setminus V(S_2')|\leq \frac{\ell_2-1}{100}$. Putting together the embeddings $\psi_1$ and $\psi_2$, we get an embedding of $S$ in $G[W^-\cup X^-]$, extending the embedding of $\phi$. Let $S'=S_1'\cup S_2'$ denote the embedded copy of $S$. Then, each vertex of $S'$ in $W^-$ has at most one rightward neighbour in $S'$, except the vertex $w$, which has $2$ rightward neighbours. Moreover, $|W^-\setminus V(S')|=|W_1^-\setminus V(S_1')|+|W_2^-\setminus V(S_2')|\leq \frac{\ell_1-1}{100}+\frac{\ell_2-1}{100}=\frac{\ell}{100}$, as desired.
\end{proof}



Finally, we prove~\cref{lemma_eitherend} for both types of ends. We have seen in~\cref{lemma_artificialend} that we can embed an artificial end to cover most of the vertices in a set $W^-\subset W$ which induces a 2-colourable graph, embedding so that each vertex covered has just one rightward edge, apart from the leftmost (`turnaround') vertex. If we only have access to proper ends and no artificial ones, then the same can be achieved by using two proper ends and applying~\cref{lemma_properend} twice (in fact, this way we do not even need a turnaround vertex). Moreover, any (sufficiently large) end contains either an artificial end or two proper ends, by~\cref{lem_endblobs1}. Using these observations, we can now prove~\cref{lemma_eitherend}.

\begin{proof}[Proof of~\cref{lemma_eitherend}] If $W^-=\emptyset$, then the result follows easily from a greedy algorithm (see \cref{lemma_greedy}). Thus, from now on, we may assume that $W^-\not =\emptyset$. Since $S$ is a $(36|W^-|,\ell)$-end of $T$, we have that $|S|\geq 20000\cdot 36|W^-|$, and that $S$ has at most two connector vertices in $T$. Using Lemma~\ref{lem_endblobs1} (with $k=2$), we can find a tree decomposition $S_1,S_2,\dots,S_m$ of $S$ such that, for each $i\in[2]$, $4\cdot20000|W^-|\leq |S_i|\leq 24\cdot 20000|W^-|$, and for each $j\in[m]$, $S_j$ has at most $2$ connector vertices in $T$. Thus, for each $i\in[2]$, $S_i$ is a $(4|W^-|,\ell_i)$-end of $T$ for some $\ell_i$, with $\ell_1+\ell_2\leq \ell$.

First suppose that at least one of $S_1$ and $S_2$ is an artificial end, and note that, without loss of generality, we can assume that it is $S_1$. Let $U$ be a subset of the (at most two) connector vertices of $S$ in $T$, and let $\phi$ be a given embedding of $S[U]$ in $G[X^-]$. Let $U'$ be the set of all connector vertices of $S$ in $T$, together with the set of connector vertices of $S_1$ in $T$. Note that $|U'|\leq 4$, so we can easily extend $\phi$ to an embedding of $S[U']$ in $G[X^-]$. By~\cref{lemma_artificialend}, there is an embedding of $S_1$ in $G[W^-\cup (X^-\setminus \phi(U'\setminus V(S_1)))]$ such that the embedding agrees with $\phi$ on the connector vertices of $S_1$, every vertex in $W^-$ in the image has at most $1$ rightward neighbour in the embedded copy, except possibly the leftmost vertex of $W^-$, which has at most $2$ rightward neighbours, and there are at most $\frac{\ell_1}{100}\leq\frac{\ell}{100}$ vertices in $W^-$ not covered by the embedding. Thus, we get an extension of $\phi$ embedding $S[V(S_1)\cup U]$ in $G[W^-\cup X^-]$ satisfying the conditions above. 
By a greedy algorithm (i.e., by \cref{lemma_greedy} with $\kappa=3$, using $\Delta(H)\leq n/50$ and $|X^-|\geq 5n/6$), we can extend this to an embedding of $S$ in $G[\phi(V(S_1))\cup X^-]\subset G[W^-\cup X^-]$, which then satisfies the required conditions.

Thus, we can assume that both $S_1$ and $S_2$ are proper ends. Let $U'$ be the union of the sets of connector vertices of $S_1, S_2$ and $S$ in $T$, and note that $|U'|\leq 6$. Given a subset $U$ of the connector vertices of $S$ in $T$, and some embedding $\phi$ of $S[U]$ in $G[X^-]$, we can extend $\phi$ to an embedding of $S[U']$ in $G[X^-]$. Let $W_1^-, W_2^-$ be a partition of $W^-$ such that $H[W_i^-]$ is empty for each $i\in [2]$. By~\cref{lemma_properend}, we can find an embedding of $S_1$ in $G[W_1^-\cup (X^-\setminus \phi(U'\setminus V(S_1)))]$ such that the embedding agrees with $\phi$ on the connector vertices of $S_1$ in $T$, and, letting $S_1'$ be the embedded copy of $S_1$, each vertex in $W_1^-\cap V(S_1')$ has at most one rightward edge in $S_1'$, and we have $|W_1^-\setminus V(S_1')|\leq\frac{\ell_1-1}{100}$. Similarly, by~\cref{lemma_properend}, we can find an embedding of $S_2$ in $G[W_2^-\cup (X^-\setminus \left(\phi(U')\cup V(S_1')\right))\cup\phi(U'\cap V(S_2))]$ such that the embedding agrees with $\phi$ on the connector vertices of $S_2$ in $T$, and letting $S_2'$ be the embedded copy, each vertex in $W_2^-\cap V(S_2')$ has at most one rightward edge in $S_2'$, and we have $|W_2^-\setminus V(S_2')|\leq\frac{\ell_2-1}{100}$. Putting together the two embeddings producing $S_1'$ and $S_2'$ with the embedding $\phi$, we get an extension of $\phi$ embedding $S[U'\cup V(S_1)\cup V(S_2)]$ in $G[W^-\cup X^-]$ such that each vertex of $W^-$ in the image has at most one rightward edge in the embedded copy, and at most $\frac{\ell_1-1}{100}+\frac{\ell_2-1}{100}\leq \frac{\ell-2}{100}$ vertices in $W^-$ are not covered by the embedding. By a greedy algorithm again (i.e., by \cref{lemma_greedy} with $\kappa=4$, using again that $\Delta(H)\leq n/50$ and $|X^-|\geq 5n/6$), we can extend this (using vertices in $X^-$ only) to get an extension of $\phi$ embedding $S$ in $G[W^-\cup X^-]$ which then satisfies the required conditions.
\end{proof}


\subsection{Filling in the gaps in $W$ with leaves}\label{section_key5}
In the previous section we have seen (with \cref{lemma_eitherend}) that we can cover each $W_i^-$ in~\cref{lemma_finalembedding} to have `few' gaps, where the number of gaps is a small fraction of the number of leaves of the ends we use. We now state a lemma that will be used to fill in this set of gaps in $\bigcup_{i\in [k]} W_i^-$, using leaves of our tree. The proof is similar to that of~\cref{lemma_onestarlike}, inspired again by the method of Havet, Reed, Stein and Wood~\cite{havet2020variant}, though, as the tree we embed comprises a smaller proportion of the overall vertices, the details are simpler.

\begin{restatable}{lemma}{lemmagapfilling}\label{lemma_gapfilling} Let $1/n\ll \eps\ll \lambda \ll
1$ and $r=\eps n$.
Let $v_1,\ldots,v_n$ be an ordering of the vertices of $K_n$ and let $W=\{v_1,\ldots,v_r\}$ and $X=\{v_{r+1},\ldots,v_{n}\}$. Let $G\subset K_n$ have $\delta(G)\geq 49n/50$, and let $H$ be the complement of $G$ in $K_n$. Suppose that $e(H)\leq \eps n^2$.

Let $\ell$ be a non-negative integer, let $T$ be a tree, and let $S$ be a subtree of $T$ with at most $2$ connector vertices in $T$ such that $S$ has at most $n/4$ vertices and at least $\ell$ leaves outside of its connector vertices in $T$. Let $W^-\subset W$ have size $\ell$, and let $X^-\subset X$ contain at least $3n/4$ vertices, including a set of vertices $A^-\subset X^-$ such that $H[A^-]$ is empty, there are no edges in $H$ between $A^-$ and $W^-$, and $|A^-|\geq 3$. Suppose furthermore that 
each vertex in $W^-$ has at most $2\lambda n$ edges in $H$.

  Then, whenever $U\subset V(S)$ is a subset of the (at most two) connector vertices of $S$ in $T$, then any embedding of $S[U]$ in $G[A^-]$ extends to an embedding of $S$ in $G[W^-\cup X^-]$ such that every vertex in $W^-$ is used as a leaf of the image of $S$.
\end{restatable}

\begin{proof}
	Let $X_0^-=\{v\in X^-: d_G(v,W^-)\geq (1-8\lambda)\ell\}=\{v\in X^-: d_H(v,W^-)< 8\lambda\ell\}$. Note that $e_H(W^-,X^-)\leq \ell\cdot 2\lambda n$, giving $|X^-\setminus X_0^-|\leq \frac{\ell\cdot 2\lambda n}{8\lambda\ell}=n/4$. In particular, $|X_0^-|\geq \frac{1}{2}n$, and note also that $A^-\subseteq X_0^-$.

   Let $L$ be a set of $\ell$ leaves of $S$ which is disjoint from the connector vertices of $S$ in $T$, and, for each $p\in V(S)\setminus L$, let $d_p=d_S(p,L)$ be the number of neighbouring leaves of $p$ in $L$. Let $d=\max_{p\in V(S)\setminus L} d_p$, and let $u$ be a vertex of $S$ with $d_u=d$.
Let $U$ be a subset of the (at most two) connector vertices of $S$ in $T$, and let $\phi$ be a given embedding of $S[U]$ in $G[A^-]$. Let $U'$ be the union of $\{u\}$ and the set of connector vertices of $S$ in $T$ (so that $|U'|\leq 3$), and extend $\phi$ arbitrarily to an embedding of $S[U']$ in $G[A^-]$, using $|A^-|\geq 3$.

	Order the vertices of $V(S-L)\setminus U'$ as $s_1,\dots,s_m$ in such a way that, for all $i\in [m]$, $s_i$ has at most $3$ neighbours to its left or in $U'$ (in the same way as in the proof of~\cref{lemma_greedy}). 
 Now we extend $\phi$ to embed $S-L$ into $G[X_0^-]$, by embedding the vertices one-by-one in this order, such that for each $i\in[m]$, if $Y_i=N_S(s_i)\cap(U'\cup \{s_1,\dots,s_{i-1}\})$ denotes the set of neighbours of $s_i$ to its left or in $U'$, then $\phi(s_i)$ is picked uniformly at random from $\left(X_0^-\cap \bigcap_{v\in Y_i}N_G(\phi(v))\right)\setminus\phi(U'\cup\{s_1,\dots,s_{i-1}\})$. Note that this is always possible, as $|Y_i|\leq 3$ and, hence,
 \begin{equation*}
   \bigg|\bigg(X_0^-\cap\bigcap_{v\in Y_i}N_G(\phi(v))\bigg)\setminus \phi(U'\cup\{s_1,\dots,s_{i-1}\})\bigg|\geq \frac{n}{2}-3\cdot \frac{n}{50}-|S|>0.  
 \end{equation*}

	First, assuming that $d\geq 8\lambda \ell$, we can show using Hall’s generalised matching criterion (see \cref{lem:hall}) that we can complete the embedding of $\phi$ by mapping $L$ to $W^-$. Indeed, if $P\subseteq V(S)\setminus L$ is non-empty, then, by the definition of $X_0^-$, we have $|N_G(\phi(P))\cap W^-|\geq (1-8\lambda) \ell$. Thus, if $\sum_{p\in P}d_p\leq (1-8\lambda)\ell$, then $|N_G(\phi(P))\cap W^-|\geq \sum_{p\in P}d_p$. On the other hand, if  $\sum_{p\in P}d_p> (1-8\lambda)\ell$, then necessarily $u\in P$ and hence $N_G(\phi(P))\supseteq N_G(\phi(u))\supseteq W^-$. So we can indeed apply~\cref{lem:hall} to complete the embedding and cover $W^-$ by the leaves in $L$.

	Thus, from now on, we may assume that $d< 8\lambda\ell$. (Note that this gives $\ell\not=0$ and hence $d\not =0$.) Similarly to the proof of~\cref{lemma_onestarlike}, we will show the following claim. 
 \begin{claim} With probability at least $1/2$, there are at least $\ell-d$ vertices $w\in W^-$ with
\begin{equation}\label{eqn:martingalecondition}
\sum_{p\in V(S)\setminus L:\phi(p)\in N_G(w)}d_p\geq  \frac{\ell}{20}.
\end{equation}\label{clm:new2}
\end{claim}
Before proving this claim, we show that it can be used to finish the proof of the lemma. Take an embedding $\phi$ of $S-L$ in $G[X_0^-]$ as above such that~\eqref{eqn:martingalecondition} fails for at most $d$ vertices in $W^-$, as must exist by the claim. Take a set $W_0^-$ of $d$ vertices in $W^-$ containing all vertices for which~\eqref{eqn:martingalecondition} fails, and embed them as the leaves adjacent to $u$. Note that this is possible, as $\phi(u)\in A^-$. We show using Hall’s generalised matching criterion (see \cref{lem:hall}) that we can complete the embedding of $\phi$ by mapping $L\setminus N_S(u)$ to $W^-\setminus W_0^-$. Consider any non-empty set $P\subset V(S)\setminus (L\cup\{u\})$. By the definition of $X_0^-$, and as $d<8\lambda \ell$, we have $|N_G(\phi(P))\cap (W^-\setminus W_0^-)|\geq (1-8\lambda)\ell-d\geq \frac{19}{20}\ell$. Hence, if $\sum_{p\in P}d_p\leq \frac{19}{20}\ell$, then $|N_G(\phi(P))\cap (W^-\setminus W_0^-)|\geq\sum_{p\in P}d_p$. On the other hand, if $\sum_{p\in P}d_p> \frac{19}{20}\ell$, then, by~\eqref{eqn:martingalecondition}, for all $w\in W^-\setminus W_0^-$ there is some $p\in P$ with $\phi(p)\in N_G(w)$, and hence $N_G(\phi(P))\supseteq W^-\setminus W_0^-$. So we can indeed apply~\cref{lem:hall} to cover $W^-\setminus W_0^-$ by $L\setminus N_S(u)$.

Thus, it is left only to prove Claim~\ref{clm:new2}.
\begin{proof}[Proof of Claim~\ref{clm:new2}] 
It is enough to show that, for all $w\in W^-$,~\eqref{eqn:martingalecondition} fails with probability at most $\frac{d}{2\ell}$. Indeed, if this is the case, then by Markov's inequality, with probability at least $1/2$,~\eqref{eqn:martingalecondition} fails for at most $d$ vertices, as claimed. Thus, fix any $w\in W^-$. For each $i\in[m]$, let
\begin{equation}
Z_{i}=\left\{\begin{array}{ll}d_{s_i} & \text{if $\phi(s_i)\in N_G(w)$}\\
0 & \text{otherwise}.\end{array}\right.
\end{equation} 
Observe that, for all  $i\in[m]$,
  \begin{align*}
    \label{eqn_possiblevalues}
   \PP\left(Z_{i}=d_{s_i}\;|\;\phi(s_1),\dots,\phi(s_{i-1})\right)&=\frac{\big|\left(X_0^-\cap\bigcap_{v\in Y_i\cup\{w\}}N_G(\phi(v))\right)\setminus \phi(U'\cup\{s_1,\dots,s_{i-1}\})\big|}{\left|\left(X_0^-\cap\bigcap_{v\in Y_i}N_G(\phi(v))\right)\setminus \phi(U'\cup\{s_1,\dots,s_{i-1}\})\right|}\nonumber\\&\geq \frac{\frac{n}{2}-4\cdot \frac{n}{50}-|S|}{n}\nonumber\geq \frac{1}{10}.  
 \end{align*}
Let us write $d^*=\ell-\sum_{i\in[m]}d_{s_i}=\sum_{p\in U'}d_p$.
Note that, since $\phi(U')\subset A^-\subset N_G(w)$, we have
\[\sum_{p\in V(S)\setminus L:\phi(p)\in N_G(w)}d_p=d^*+\sum_{i\in[m]}Z_{i}.\]
Using Azuma's inequality (\cref{lem:azuma}) for the submartingale $\sum_{j\in[i]}(Z_{j}-d_{s_j}/10)$, we have 
\begin{align*}
\PP\bigg(\sum_{p\in V(S)\setminus L:\phi(p)\in N_G(w)}d_p\leq \frac{\ell}{20}\bigg)
&=\PP\bigg(\sum_{i\in[m]} Z_{i}+d^*\leq \frac{\sum_{i\in[m]}d_{s_i}+d^*}{20}\bigg)\\
&=\PP\bigg(\sum_{i\in[m]} \bigg(Z_{i}-\frac{d_{s_i}}{10}\bigg)\leq -\frac{\sum_{i\in[m]}d_{s_i}}{20}-\frac{19}{20}d^*\bigg)\\
&\leq \exp\bigg(-\frac{(\sum_{i\in[m]}d_{s_i}/20+19d^*/20)^2}{2\sum_{i\in[m]}d_{s_i}^2} \bigg)\\
&\leq \exp\bigg(-\frac{(\sum_{p\in V(S)\setminus L}d_{p}/20)^2}{2\sum_{p\in V(S)\setminus L}d_{p}^2} \bigg)\\
&\leq
\exp\bigg(-\frac{(\sum_{p\in V(S)\setminus L}d_{p}/20)^2}{2d\sum_{p\in V(S)\setminus L}d_{p}} \bigg)
= 
\exp\left(-\frac{\ell}{800d} \right)
\leq \frac{d}{2\ell},
\end{align*}
where in the last step we used $d\leq 8\lambda \ell$ and $\lambda\ll 1$. This finishes the proof of the claim and hence the lemma.
\renewcommand{\qedsymbol}{$\boxdot$}
\qedhere
\renewcommand{\qedsymbol}{$\square$}
\qedsymbol
\end{proof}
\renewcommand{\qedsymbol}{}
\end{proof}
\renewcommand{\qedsymbol}{$\square$}


\subsection{Completing the embedding}\label{section_key6}
The following lemma is the last step towards proving~\cref{lemma_finalembedding}. It is used to extend our partial embedding covering $W^-$ to a full embedding, while making sure that a protected set $X^\text{forb}$ is covered by well-separated degree $2$ vertices (this is needed in the proof of the main lemma to guarantee that~\itref{prop:S2} and~\itref{prop:S3} are satisfied).
\begin{lemma}\label{lemma_fillingwithbarepaths}
Let $1/n\ll \eps\ll \lambda \ll
1$ and $r=\eps n$.
Let $v_1,\ldots,v_n$ be an ordering of the vertices of $K_n$ and let $W=\{v_1,\ldots,v_r\}$ and $X=\{v_{r+1},\ldots,v_{n}\}$. Let $G\subset K_n$ have $\delta(G)\geq 49n/50$, and let $H$ be the complement of $G$ in $K_n$. Suppose that $e(H)\leq \eps n^2$.

Let $T$ be a tree and let $S$ be a subtree of $T$ with at most $\lambda n$ leaves, at most $2$ connector vertices in $T$, and with $|S|\geq n/7$. Let $X^-\subseteq X$ have size $|S|$, and let $X^\mathrm{forb}\subseteq X^-$ with $|X^\mathrm{forb}|\leq 40\lambda n$.

Then, whenever $U\subset V(S)$ is a subset of the (at most two) connector vertices of $S$ in $T$, then any embedding of $S[U]$ in $G[X^-\setminus X^\mathrm{forb}]$ extends to an embedding of $S$ in $G[X^-]$ such that the vertices in $X^\mathrm{forb}$ are covered by an independent set of non-connector degree $2$ vertices.
\end{lemma}

 We will prove Lemma~\ref{lemma_fillingwithbarepaths} by a natural argument which first embeds the tree with a large set of degree $2$ vertices removed, and then completes the embedding by mapping those degree $2$ vertices appropriately. To do this, we will need a large well-separated set of degree $2$ vertices, as guaranteed by the following lemma, which is a rather special case of a more general lemma of Krivelevich~\cite{krivelevich2010embedding} on bare paths in trees.

 \begin{lemma}[\cite{krivelevich2010embedding}, Lemma 2.1]\label{lemma_barepathsoflengthtwo}
 If $T$ is a tree with $n$ vertices and $\ell$ leaves, then $T$ contains a set $Y$ of degree $2$ vertices such that $|Y|\geq \frac{n-3(2\ell-2)}{3}$ and the sets $\{y\}\cup N_T(y)$ (over all $y\in Y$) are pairwise disjoint.
 \end{lemma}

\begin{proof}[Proof of~\cref{lemma_fillingwithbarepaths}]
	Let $X^{\text{low}}=\{v\in X^-:d_H(v)\geq n/1000\}$, and observe that $|X^{\text{low}}|\leq \frac{2e(H)}{n/1000}\leq 2000\eps n$. By~\cref{lemma_barepathsoflengthtwo}, we can find a set $Y$ of degree $2$ vertices in $S$ such that $|Y|\geq n/22$, and the sets $\{y\}\cup N_T(y)$ ($y\in Y$) are pairwise disjoint from each other and also from the set of connector vertices of $S$ in $T$.

Let $U$ be a subset of the connector vertices of $S$ in $T$, and let $\phi$ be a given embedding of $S[U]$ in $G[X^-\setminus X^\text{forb}]$. By a greedy algorithm, we can extend $\phi$ to an embedding of $S[V(S)\setminus Y]$ in $G[(X^-\setminus(X^\text{forb}\cup X^{\text{low}}))\cup \phi(U)]$. Indeed,~\cref{lemma_greedy} applies, as
\[
|(X^-\setminus(X^\text{forb}\cup X^{\text{low}}))\cup \phi(U)|-(|S|-|Y|)-2\Delta(H)\geq |S|-40\lambda n-2000\eps n-\Big(|S|-\frac{n}{22}\Big)-2\cdot \frac{n}{50}>0.
\]

Let $B=X^-\setminus \phi(V(S)\setminus Y)$. Note that $|B|=|Y|$ and $X^\text{forb}\subset B$. For each $y\in Y$, let $x_y,z_y$ be the two neighbours of $y$ in $S$ (labelled arbitrarily), and let $B_y=N_G(\phi(x_y))\cap N_G(\phi(z_y))\cap B$. Note that if we can match $Y$ and $B$ with some bijection $\psi$ so that $\psi(y)\in B_y$ for all $y\in Y$, then, putting together $\psi$ and $\phi$ gives an embedding of $S$ in $G[X^-]$, extending the original embedding of $S[U]$, such that $B$ is covered by $Y$ (and hence $X^\text{forb}$ is covered by an independent set of non-connector degree $2$ vertices). Thus, to complete the proof of the lemma, it suffices to find such a matching.
We will show that such a bijection $\psi$ exists using Hall's matching criterion (see \cref{lem:hall}).

Note that, for all $y\in Y$, we have $\phi(x_y),\phi(z_y)\in X^-\setminus X^{\text{low}}$, and hence $|B_y|\geq |B|-2n/1000=|Y|-n/500$. Thus, if $Y'\subset Y$ is non-empty with $|Y'|\leq |Y|-n/500$, then $|\bigcup_{y\in Y'}B_y|\geq |Y'|$. On the other hand, if $Y'\subset Y$ with $|Y'|> |Y|-n/500$, then $|Y'|>n/22-n/500>n/50$. But $\delta(G)\geq 49n/50$, so, for any $b\in B$, there are at most $n/50$ choices of $y\in Y$ such that $\{x_y,z_y\}\not \subset N_G(b)$. It follows that if $Y'\subset Y$ with $|Y'|> |Y|-n/500$ then $\bigcup_{y\in Y'}B_y=B$, so the conditions of~\cref{lem:hall} are indeed satisfied and the result follows. 
\end{proof}


\subsection{Deducing the second key lemma}\label{section_key7}
Finally, we can prove our second key lemma, which we restate first for convenience.
\setcounter{restatedpropcounter2}{\value{propcounter}}
\setcounter{propcounter}{\value{restatedpropcounter}}

\lemmafinalembedding*

\setcounter{propcounter}{\value{restatedpropcounter2}}

\begin{proof}
Let $U\subset V(T)$ be given with $|U|\leq 2$, and let $\phi$ be a given embedding of $T[U]$ in $G[A^-]$. From now on, whenever a subtree $T'$ of $T$ intersects $U$, each vertex in $U\cap V(T')$ will also be regarded as a connector vertex. (This can be made precise, for example, by regarding $T$ as a subtree of a larger tree with an additional leaf edge added incident to each point of $U$.) Thus, for example, $T$ is a tree with at most $2$ connector vertices.

Using Lemma~\ref{lem_endblobs0} and Lemma~\ref{lem_endblobs1}, we can take a tree decomposition of $T$ into subtrees $F_1,F_2, F_3, \dots,F_q$ (for some $q\geq 3$) such that each of these trees contain at most $2$ connector vertices (including the vertices in $U$), $|F_3|\geq n/7$, and for each $i\in[2]$ we have $n/1000\leq|F_i|\leq n/100$. (Indeed, we may use~\cref{lem_endblobs0} to divide $T$ into $4$ trees with at most $2$ connector vertices each, one of which has between $n/7$ and $6n/7$ vertices. Then, taking the largest one of the other three trees, which has at least $n/21$ vertices, we may use~\cref{lem_endblobs1} to divide it into several trees with at most $2$ connector vertices each, such that two of those trees have size between $n/1000$ and $6 n/1000$.) Let $F_1$ and $F_2$ have $\ell$ and $\ell'$ non-connector leaves, respectively. Without loss of generality, we may assume that $\ell\leq \ell'$.

Let $C$ be the set of connector vertices of $F_2$ (so that $|C|\leq 2$), and extend $\phi$ arbitrarily to an embedding of $T[U\cup C]$ in $G[A^-]$, using $|A^-|\geq 16$. Pick also a subset $A^\text{forb}\subset A^-\setminus \phi(U\cup C)$ of size $3$, to be reserved for a later application of~\cref{lemma_gapfilling}. Let $X^\text{low}=\{v\in X\setminus A^-:d_H(v)\geq \frac{n}{200}\}$, and note that $|X^\text{low}|\leq \frac{2e(H)}{n/200}\leq 400\eps n$. Let us write
\[X^\text{forb}=X^\text{low}\cup (A\setminus A^-),\]
so that $|X^\text{forb}|\leq 400\eps n+20r=420\eps n$. We reserve $X^\text{forb}$ to be covered at the very end by $F_3$, in an application of~\cref{lemma_fillingwithbarepaths}.

For each $i\in [k]$, let $m_i=|W_i^-|$, so that
\[
10^6\bigg(\sum_{i\in [k]}(m_i+1)\bigg)\leq 10^6(|W|+r)\leq 2\cdot 10^6r\leq |F_1|/150,
\]
as $|F_1|\geq n/1000$. Using Lemma~\ref{lemma_endblobs}, we can then take a tree decomposition of $F_1$ into $T_1,\ldots,T_{k'}$ (for some $k'\geq k$) so that, for each $i\in [k]$, $10^6(m_i+1)\leq |T_i|\leq 6\cdot 10^6(m_i+1)$, and each $T_i$ has at most $2$ connector vertices. For each $i\in [k]$, let $T_i$ have $\ell_i$ non-connector leaves, so that $T_i$ is then a $(36|W_i^-|,\ell_i)$-end of $T$ and $\sum_{i\in[k]}\ell_i\leq \ell$, as $\ell$ is the number of non-connector leaves of $F_1$. Let $C_0$ be the set of connector vertices in $T_1,\dots,T_{k'}$, and let $C'=U\cup C\cup C_0\cup\bigcup_{i\in[k']\setminus[k]}V(T_i)$ (equivalently, $C'$ is the set obtained from $U\cup C\cup V(F_1)$ after removing the non-connector vertices of $T_i$ for each $i\in[k]$). As $|C'|\leq |F_1|+4\leq n/50$, we can (by~\cref{lemma_greedy} with $\kappa=|U\cup C|\leq 4$) greedily extend $\phi$ to an embedding of $T[C']$ in $G[X^-\setminus (X^\text{forb}\cup A^\text{forb})]$. 

We can then extend $\phi$ to be defined on each $V(T_i)$, $i\in [k]$, as follows. For each $i\in [k]$ in turn, write 
$$X^\text{used}_{i-1}=X^-\cap \phi\bigg(\bigg(C'\cup\bigcup_{j\in[i-1]}V(T_j)\bigg)\setminus V(T_i)\bigg)$$ for the set of vertices in $X^-$ already in the image of $\phi$ (but not used as the image of a connector vertex of $T_i$). Then set
$$X_i^-=X^-\setminus \left(X^\text{used}_{i-1}\cup X^\text{forb}\cup A^\text{forb}\right).$$
Note that $|X_i^-|\geq |X^-|-(|F_1|+4)-420\eps n-3\geq (\frac{99}{100}n-|W^-|)-n/50>5n/6$. Also, $|T_i|\leq n/50$ as $T_i$ is a subgraph of $F_1$.
Thus, recalling \itref{cond_finalembedlemma2} and \itref{cond_finalembedlemma3} too, we can apply~\cref{lemma_eitherend} to get an embedding of $T_i$ in $W_i^-\cup X_i^-$ (agreeing with our embedding $\phi$ on the connector vertices) such that at most $\frac{\ell_i}{100}$ vertices in $W_i^-$ are not in the image, and all the vertices in $W_i^-$ in the image have at most $1$ rightward edge, except the leftmost vertex of $W_i^-$, which has at most $2$ rightward edges. This gives an extension of $\phi$ which covers $V(T_i)$.

After the final one of these step-by-step extensions, we have an extension of $\phi$ embedding $T[U\cup C\cup V(F_1)]$ in $G[W^-\cup (X^-\setminus (X^\text{forb}\cup A^\text{forb}))]$ such that at most $\sum_{i\in [k]}\frac{\ell_i}{100}\leq \frac{\ell}{100}$ vertices of $W^-$ are uncovered, and each vertex of $W^-$ in the image has at most $1$ rightward edge in the embedded copy, except, for each $i\in[k]$, the leftmost vertex of $W_i^-$, which has at most $2$ rightward edges.

We will now use~\cref{lemma_gapfilling} to fill in the gaps left in $W^-$, using the subtree $F_2$. Let $W_0^-=W^-\setminus \phi(V(F_1))$ be the set of vertices in $W^-$ not yet covered, let $X_{k+1}^{\text{used}}=X^-\cap \phi((U\cup V(F_1))\setminus C)$ be the set of points in $X^-$ that are already used (apart from the images of connector vertices of $F_2$), and let $X_0^-=X^-\setminus (X_{k+1}^\text{used}\cup X^\text{forb})$. Write $A_0^-=\phi(C)\cup A^\text{forb}$. We will now check that we can apply \cref{lemma_gapfilling} for the graph $G$ (with complement $H$), tree $T$ with subtree $F_2$, $W_0^-$ taking the role of $W^-$, $X_0^-$ taking the role of $X_0$, and $A^-$ being given by $A_0^-$.

First, note that, using the bounds for $|X_i^-|$ ($i\in[k])$ above, we easily get $|X_0^-|\geq 3n/4$. Also, note that $A_0^-\subset X_0^-$ and $|A_0^-|\geq 3$. Moreover, as $A_0^-\subset A^-$, we have that $H[A_0^-]$ is empty, and there are no edges between $W_0^-$ and $A_0^-$ in $H$. Also, $F_2$ has $\ell'\geq \ell\geq |W_0^-|$ non-connector leaves. Thus, as \eref{cond_finalembedlemma3} holds, all conditions of~\cref{lemma_gapfilling} are satisfied, and we can find an embedding of $F_2$ in $W_0^-\cup X_0^-$, agreeing with $\phi$ on the set $C$ of connector vertices, such that each point of $W^-_0$ is in the image and is covered by a leaf. Thus, we get an extension of $\phi$ embedding $T[U\cup V(F_1)\cup V(F_2)]$ in $G[W^-\cup (X^-\setminus X^\text{forb})]$ such that all points in $W^-$ are covered (by non-connector vertices), and each such point has at most $1$ rightward edge in the embedded copy, except, for each $i\in [k]$, the leftmost vertex of $W^-_i$, which has at most $2$ rightward edges. Thus, if we can extend this $\phi$ to embed $T$ in $G[W^-\cup X^-]$, \itref{prop:S1} will be satisfied.

It is left, then, to embed the subtrees $F_3,\dots,F_q$ appropriately using the remaining vertices of $X^-$. As $\Delta(H)\leq n/50$ and $T[U\cup V(F_1)\cup V(F_2)]$ has at most $4$ connected components, and $|F_3|\geq n/7$, we can greedily (by~\cref{lemma_greedy}) extend $\phi$ to be an embedding of $T[U\cup\bigcup_{i\in[q]\setminus\{3\}} V(F_i)]$ in $G[W^-\cup (X^-\setminus X^\text{forb})]$. Let $X^\text{final}=X^-\setminus\phi(V(T)\setminus V(F_3))$ be the set of vertices to be covered by $F_3$, i.e., the vertices not yet covered together with the images of the connector vertices of $F_3$, so that $|X^\text{final}|=|F_3|$ and $X^\text{forb}\subset X^\text{final}$. Recall that $|F_3|\geq n/7$, $|X^\text{forb}|\leq 420\eps n$, and $T$, and hence $F_3$, has at most $\lambda n$ leaves. Hence, by~\cref{lemma_fillingwithbarepaths}, there is an embedding of $F_3$ in $G[X^\text{final}]$, agreeing with $\phi$ on the connector vertices of $F_3$, such that the set $X^\text{forb}$ is covered by an independent set of non-connector degree $2$ vertices. Thus, we get an extension of $\phi$ embedding $T$ in $W^-\cup X^-$, satisfying all the required conditions. Indeed, we have already seen that \itref{prop:S1} will hold, and \itref{prop:S2} follows immediately from the fact that $X^\text{forb}$ is covered by degree $2$ vertices. Finally, \itref{prop:S3} holds, since  $A\setminus A^-$ is covered by an isolated set of non-connector vertices in $F_3$, and no vertex in $W^-$ is covered by a vertex of $F_3$. 
\end{proof}




\section*{Acknowledgements}
The authors would like to thank Mat\'ias Pavez-Sign\'e for discussions around the methods in~\cite{havet2020variant}, and the referee for their detailed review which improved this paper.

\bibliographystyle{abbrv}
\bibliography{tpc}

\begin{thebibliography}{10}

\bibitem{allen2021tree}
P.~Allen, J.~B{\"o}ttcher, D.~Clemens, J.~Hladk{\'y}, D.~Piguet, and A.~Taraz.
\newblock The tree packing conjecture for trees of almost linear maximum
  degree.
\newblock {\em arXiv preprint arXiv:2106.11720}, 2021.

\bibitem{palmer}
J.~Balogh and C.~Palmer.
\newblock On the tree packing conjecture.
\newblock {\em SIAM Journal on Discrete Mathematics}, 27(4):1995--2006, 2013.

\bibitem{bollotrees}
B.~Bollob{\'a}s.
\newblock Some remarks on packing trees.
\newblock {\em Discrete Mathematics}, 46(2):203--204, 1983.

\bibitem{bollobas1998modern}
B.~Bollob{\'a}s.
\newblock {\em Modern {G}raph {T}heory}.
\newblock Springer, 1998.

\bibitem{handbook}
B.~Bollobás.
\newblock Extremal graph theory.
\newblock In {\em Handbook of {C}ombinatorics, {vol.\ II}}. Elsevier, 1995.

\bibitem{bottcher2016approximate}
J.~B{\"o}ttcher, J.~Hladk{\'y}, D.~Piguet, and A.~Taraz.
\newblock An approximate version of the tree packing conjecture.
\newblock {\em Israel {J}ournal of {M}athematics}, 211(1):391--446, 2016.

\bibitem{ferber2019packing}
A.~Ferber and W.~Samotij.
\newblock Packing trees of unbounded degrees in random graphs.
\newblock {\em Journal of the London Mathematical Society}, 99(3):653--677,
  2019.

\bibitem{fishburn1983packing}
P.~C. Fishburn.
\newblock Packing graphs with odd and even trees.
\newblock {\em Journal of Graph Theory}, 7(3):369--383, 1983.

\bibitem{gyar}
A.~Gy{\'a}rf{\'a}s and J.~Lehel.
\newblock Packing trees of different order into {$K_n$}.
\newblock In {\em Combinatorics (Proc. Fifth Hungarian Colloq., Keszthely,
  1976), vol{.\ }I}, Colloquia Mathematica Societatis János Bolyai 18, pages
  463--469. North-Holland, Amsterdam, 1978.

\bibitem{havet2020variant}
F.~Havet, B.~Reed, M.~Stein, and D.~R. Wood.
\newblock A variant of the {E}rd{\H{o}}s-{S}{\'o}s conjecture.
\newblock {\em Journal of Graph Theory}, 94(1):131--158, 2020.

\bibitem{hobbs1987packing}
A.~M. Hobbs, B.~A. Bourgeois, and J.~Kasiraj.
\newblock Packing trees in complete graphs.
\newblock {\em Discrete mathematics}, 67(1):27--42, 1987.

\bibitem{joos2019optimal}
F.~Joos, J.~Kim, D.~K{\"u}hn, and D.~Osthus.
\newblock Optimal packings of bounded degree trees.
\newblock {\em Journal of the European Mathematical Society},
  21(12):3573--3647, 2019.

\bibitem{keevash2020ringel}
P.~Keevash and K.~Staden.
\newblock Ringel's tree packing conjecture in quasirandom graphs.
\newblock {\em Journal of the European Mathematical Society}, 27(5):1769--1826,
  2025.

\bibitem{krivelevich2010embedding}
M.~Krivelevich.
\newblock Embedding spanning trees in random graphs.
\newblock {\em SIAM Journal on Discrete Mathematics}, 24(4):1495--1500, 2010.

\bibitem{montgomery2021proof}
R.~Montgomery, A.~Pokrovskiy, and B.~Sudakov.
\newblock A proof of {R}ingel’s conjecture.
\newblock {\em Geometric and Functional Analysis}, 31(3):663--720, 2021.

\bibitem{straight1979packing}
H.~J. Straight.
\newblock Packing trees of different size into the complete graph.
\newblock {\em Annals of the New York Academy of Sciences}, 328(1):190--192,
  1979.

\bibitem{wormald1999differential}
N.~C. Wormald.
\newblock The differential equation method for random graph processes and
  greedy algorithms.
\newblock In {\em Lectures on approximation and randomized algorithms}, pages
  73--155, 1999.

\bibitem{zak}
A.~{\.Z}ak.
\newblock Packing large trees of consecutive orders.
\newblock {\em Discrete Mathematics}, 340(2):252--263, 2017.

\bibitem{zaks1977decomposition}
S.~Zaks and C.~L. Liu.
\newblock Decomposition of graphs into trees.
\newblock In {\em Proceedings of the Eighth Southeastern Conference on
  Combinatorics, Graph Theory and Computing}, Congressus Numerantium XIX, pages
  643--654, 1977.

\end{thebibliography}

\end{document}